\documentclass[a4paper,10pt]{article}   
\usepackage[utf8]{inputenc}    
\usepackage{graphicx}
\usepackage{subcaption} 
\usepackage{float}   
\usepackage{fancybox}		  
\usepackage{makeidx} 
\usepackage{verbatim}
\usepackage{listings} 
\usepackage{fancyvrb}
\usepackage{amsfonts}
\usepackage{color}
\usepackage{colortbl}
\usepackage{color}
\usepackage{dsfont}
\usepackage{epsfig}
\usepackage{verbatim}
\usepackage{listings}
\usepackage{bbm}
\usepackage{tikz}
\definecolor{dgreen}{rgb}{0.0, 0.5, 0.0}
\definecolor{byzantium}{rgb}{0.44, 0.16, 0.39}

\usepackage{amssymb}
\usepackage{amsbsy}
\usepackage{amsmath}
\usepackage{enumerate}
\usepackage{stmaryrd}
\usepackage{mathtools}
\usepackage{tikz}
\usepackage{bbm}

\usepackage{fancyhdr}

\usepackage{amsthm}

\newtheorem{prop}{Proposition}
\newtheorem{theorem}{Theorem}
\newtheorem{maintheorem}{Theorem}
\newtheorem{hyp}{Hypothesis}
\newtheorem{lemma}{Lemma}
\newtheorem{rem}{Remark}[section]
\newtheorem{definition}{Definition}
\newtheorem{example}{Example}

\usepackage[margin=1in]{geometry}

\def\N {\mathbb{N}}

\def\R {\mathbb{R}}

\newcommand{\Lphi}{L_\phi}
\newcommand{\sumi}{\sum_{i=1}^N}
\newcommand{\sumj}{\sum_{j=1}^N}
\newcommand{\Lip}{\mathrm{Lip}}
\newcommand{\elts}{\{1,\cdots,N\}}
\newcommand{\se}{s_*}
\renewcommand{\P}{\mathcal{P}}

\newcommand{\PR}{\P(\R^d)}

\newcommand{\tx}{\tilde{x}}
\newcommand{\tm}{\tilde{m}}
\newcommand{\txn}{x_N}
\newcommand{\tmn}{m_N}
\newcommand{\one}{\mathbbm{1}}
\newcommand{\schema}[1]{\b{\sc #1}}

\DeclarePairedDelimiter\floor{\lfloor}{\rfloor}
\newcommand{\Pdn}{P_\mathrm{d}^N}
\newcommand{\Pcn}{P_\mathrm{c}^N}
\newcommand{\xin}{\xi_N}
\newcommand{\zn}{\zeta_N}
\newcommand{\ba}{\phi}

\newcommand{\ess}{\mathrm{ess}}
\newcommand{\bzn}{\bar{\zeta}_N^\epsilon}
\newcommand{\bxin}{\bar{\xi}_N^\epsilon}

\DeclareMathOperator*{\esssup}{ess\,sup}

\newcommand{\CLu}{\mathcal{C}^1([0,T];L^\infty(I;\R^d \times \R))}
\newcommand{\CLp}{\mathcal{C}^1([0,T];L^\infty(I;\R^d \times \Rp))}
\newcommand{\CLx}{\mathcal{C}([0,T];L^\infty(I;\R^d))}
\newcommand{\CLm}{\mathcal{C}([0,T];L^\infty(I;\R))}

\newcommand{\CLmm}{\mathcal{C}^1([0,T];L^2(I;\R))}

\newcommand{\Mx}{M_{x_0}}
\newcommand{\Mm}{M_{m_0}}
\newcommand{\Kx}{K_{x_0}}
\newcommand{\Km}{K_{m_0}}
\newcommand{\tT}{\tilde{T}}
\newcommand{\Rp}{\R^{*+}}
\newcommand{\bS}{\bar{S}}

\begin{document}

\title{Mean-field and graph limits for collective dynamics models with time-varying weights}

\author{
Nathalie Ayi\thanks{Sorbonne Universit\'e, Universit\'e Paris-Diderot SPC, CNRS, Laboratoire Jacques-Louis Lions, Paris, France} 
 \and
Nastassia Pouradier Duteil\thanks{Sorbonne Universit\'e, Inria, Universit\'e Paris-Diderot SPC, CNRS, Laboratoire Jacques-Louis Lions, Paris, France} }

\maketitle

\bibliographystyle{abbrv}


\abstract{In this paper, we study a model for opinion dynamics where the influence weights of agents evolve in time via an equation which is coupled with the opinions' evolution. We explore the natural question of the large population limit with two approaches: the now classical mean-field limit and the more recent graph limit. After establishing the existence and uniqueness of solutions to the models that we will consider, we provide a rigorous mathematical justification for taking the graph limit in a general context. Then, establishing the key notion of \textit{indistinguishability}, which is a necessary framework to consider the mean-field limit, we prove the subordination of the mean-field limit to the graph one in that context. This actually provides an alternative (but weaker) proof for the mean-field limit. We conclude by showing some numerical simulations to illustrate our results.}

\section{Introduction}

Over the past few years, there has been a soaring interest for the study of multi-agent collective behavior models. 
Indeed, they can be applied in many different areas: biology with the study of collective flocks and swarms \cite{MR2310046,MR2205679,MR2064375,1582249}, aviation \cite{MR1617559}, and social dynamics \cite{HK} (among many others). 
In this article, we will take a particular interest in models for opinion dynamics. 
The first ones were introduced in the 50's \cite{social,social2} and already stressed the highly non-linear nature of social interactions. Some early models for consensus formation have been introduced by DeGroot \cite{DG} and more recently by Hegselmann and Krause \cite{HF,HK,MR1792007}. Since then, a wealth of models have been developed in order to study one of the key features of collective dynamics: the intriguing emergence of global patterns from local interaction rules. This phenomenon is referred to as \emph{self-organization} \cite{ACMPPRT17,MR3714980,MR2343706,JabinMotsch14,MR2064375,1582249}.

Most social dynamics models have a common structure: the dynamics of the agents' opinions are described by a system of ODEs in which the agents interact pairwise, i.e.
\begin{equation}\label{eq:HKintro}
\displaystyle \frac{d}{dt} x_i(t) = \frac{1}{N} \sumj a_{ij}\,(x_j(t)-x_i(t)),
\end{equation}
where $x_i$ represents the time-evolving opinion of agent $i$.
Depending on the nature of the interaction coefficients $a_{ij}$, these models can be roughly classified in two categories.
In the first one, interactions are pre-determined by a given interaction network, which represents the inherent structure of the population's interactions \cite{olfati2007,WS98}. Then each pairwise interaction coefficient $a_{ij}$ is non-zero if and only if the edge $(i,j)$ is part of the underlying graph of interactions. 
The second approach only considers the state space of opinions and defines the interaction coefficients as a function of the pairwise distances: $a_{ij}:=a(\|x_i-x_j\|)$. In this case, every agent can potentially interact with any other if the distance separating them belongs to the support of the interaction function $a$: there is no underlying network (see for instance \cite{MT14}).\\

Recently, a variant of this model has introduced the concept of \emph{weights of influence} \cite{McQuadePiccoliPouradierDuteil19,MR3965293}. 
In this augmented model, each agent is not only defined by its opinion $x_i$, but also by its weight $m_i$.
Then, the influence of an agent $j$ on an agent $i$'s opinion is proportional to its weight $m_j$. These weights are assumed to evolve in time via an equation which is coupled with the opinions' evolution. In other words, the evolution of each agent's opinion does not only depend on its proximity with another agent, but also on the charisma or popularity of the latter - and this charisma also evolves in time. This can be formulated as a system of $2N$ ODEs, as follows:
\begin{equation}\label{eq:syst-gen-intro}
\displaystyle \frac{d}{dt} x_i(t) = \frac{1}{N} \sumj m_j(t)\,a(\|x_i-x_j\|)(x_j(t)-x_i(t)), \qquad \frac{d}{dt} m_i(t) = \psi_i(x(t),m(t)).
\end{equation}

Interestingly, although the interaction coefficients are given by a function of the pairwise distances between opinions, in this approach we can again view the opinions as nodes of an underlying network. The corresponding weighted graph is \emph{non-symmetric}
(the edge between $i$ and $j$ being weighted by $m_i$ in one direction and by $m_j$ in the other), and \emph{time-evolving} (the weights' dynamics being coupled with the dynamics of the nodes). \\

As in all models of collective dynamics, several natural questions arise. 
A first one concerns the \emph{large time limit}, that is the asymptotic behavior of the system. 
Many works in the literature have delved into the question of self-organization, i.e. the spontaneous emergence of well-organized group patterns such as consensus, alignment, clustering  or dancing equilibrium \cite{MR3714980,MR2343706,MR3392625}. This was studied  for the augmented model with time-varying weights in \cite{McQuadePiccoliPouradierDuteil19}.\\

In this paper, we will explore another natural question, the \emph{large population limit}.
When the number of agents tends to infinity, the previous system of $2N$ equations becomes unmanageable, a problem well-known as the \emph{curse of dimension}.
A common answer to this issue consists of studying the \emph{mean-field limit} of the system.

First introduced in the context of gas dynamics (see \cite{braun1977} for instance), when describing particles interacting via a force, a mean-field limit is a limit in which the number of particles $N$ is large ($N$ goes to infinity) but is such that the interaction between particles is both weak enough so that the forces applying on one particle remain finite at the limit, and strong enough so that all the particles continue to interact. 
The mean-field limit process consists of representing the population by its density probability, instead of following each agent's individual trajectory. 
In the case of the classical opinion dynamics \eqref{eq:HKintro}, the limit measure $\mu(t,x)$ represents the density of agents with opinion $x$ at time $t$.
The mean-field limit of this system is now a classical result, 
and one can show that the limit measure satisfies a non-local transport equation \cite{Dobrushin79}.

In the context of the augmented system with time-varying weights \eqref{eq:syst-gen-intro}, the limit measure $\mu(t,x)$ represents the total weight of the agents with opinion $x$ at time $t$. It was shown in \cite{PouradierDuteil21} that it solves a non-local transport equation with non-local source. \\

However, there is a limitation to the mean-field approach. 
Since it describes the population by its density, it requires all particles to be \emph{indistinguishable}. This not only entails a significant information loss, but also greatly reduces the span of models that can be studied. It is also incompatible with the graph viewpoint.
In particular, in the case of the augmented model \eqref{eq:syst-gen-intro} with time-varying weights, it requires strong assumptions on the mass dynamics $\psi_i$.
This leads us to the other approach that will be central to this paper: the graph limit method. \\

In 2014, Medvedev used techniques from the recent theory of graph limit \cite{MR2455626,MR2277152,MR2274085,LS,MR3012035} to derive rigorously the continuum limit of dynamical models on deterministic graphs \cite{Medvedev14}. 
In the present paper, we extend this idea to our collective dynamics model with time-varying weights, adopting the graph point of view described above. 
The central point of this approach consists of describing the infinite population by two functions $x(s)$ and $m(s)$ over the space of continuous indices $s$.
The discrete system of ODEs is then shown to converge as $N$ goes to infinity to a system of two non-local diffusive equations in the space of continuous indices. We show that this approach is more general than the mean-field one, and the Graph Limit can be derived for a much greater variety of models. \\

In the case of dynamics preserving the indistinguishability of particles, we show that both the graph limit and the mean-field limit can be derived. 
In particular, we show that there is a hierarchy between the two limit equations: the mean-field limit equation can be derived from the graph limit one.
This subordination of the mean-field limit equation to the graph limit equation, pointed out in \cite{BiccariKoZuazua19}, is natural:
Indeed, the mean-field limit process eliminates all individuality from the particles by considering only the population density. 
Thus, there would be no hope of recovering the graph limit equation from the mean-field one. \\

The paper is organized as follows: in Section \ref{sec:presentation}, we present the model and state the main results. In Section \ref{sec:graphlim}, we focus on the graph limit. We start by establishing the existence and uniqueness of a solution to the graph limit equation,
and then prove the convergence of the discrete system to the graph limit equation. 
In Section \ref{sec:mfl}, we study the mean-field limit. We first define the key notion of indistinguishability for a particle system. We then prove the subordination of the mean-field limit to the graph one and finish by a weaker but alternative proof of the mean-field limit based on that subordination. 
Lastly, we present some numerical simulations in Section \ref{sec:numeric} with concrete models to illustrate our results.

\section{Presentation of the model and main results}
\label{sec:presentation}

We study a social dynamics model with time-varying weights introduced in \cite{McQuadePiccoliPouradierDuteil19}.
From here onward, $d\in\N$ will represent the dimension of the space of opinions, and $N\in\N$ will represent the number of agents whose opinions evolve in $\R^d$.

More specifically, let $x^N=(x_i^{N})_{i\in\elts}:[0,T]\rightarrow(\R^d)^N$ represent the \emph{opinions} (or positions) of $N$ agents, and let $m^N=(m_i^{N})_{i\in\elts}:[0,T]\rightarrow \R^N$ represent their individual \emph{weights of influence}. Each opinion's time-evolution is affected by the opinion of each neighboring agent via the interaction function $\phi\in\Lip(\R^d; \R)$, proportionally to the neighboring agent's weight of influence. In turn, the agents' weights are assumed to evolve in time and their dynamics may depend on the opinions and weights of all the other agents, via functions $\psi_i^{(N)}:(\R^d)^N\times\R^N\rightarrow\R$.
Given a set $(x_i^{0,N})_{i\in\elts}$ of initial opinions and $(m_i^{0,N})_{i\in\elts}$ of initial weights, the evolution of the opinions and weights are given by the following system:

\begin{equation}\label{eq:syst-gen}
\begin{cases}
\displaystyle \frac{d}{dt} x_i^{N}(t) = \frac{1}{N} \sumj m_j^N(t)\, \phi(x_j^{N}(t)-x_i^{N}(t))\\
\displaystyle  \frac{d}{dt} m_i^{N}(t) = \psi_i^{(N)}(x^{N}(t),m^{N}(t)), \qquad i\in\elts,
\end{cases}
\end{equation}
supplemented by the initial conditions 
\begin{equation*}
\forall i\in\elts, \quad x_i^{N}(0) = x_i^{0,N} \quad \text{ and } \quad m_i^{N}(0) = m_i^{0,N} 
\end{equation*}
such that 
\begin{equation}
\label{eq:sum_M}
\sumi m_i^{0,N}=N.
\end{equation}

We point out that the choice $m_i^{0,N}=1$ and $\psi_i^{(N)}\equiv 0$ for all $i\in\elts$  brings us back to the classical Hegselmann-Krause model for opinion dynamics \cite{HK}:
\begin{equation}\label{eq:HK}
\displaystyle \frac{d}{dt} x_i^{N}(t) = \frac{1}{N} \sumj \phi(x_j^{N}(t)-x_i^{N}(t)) \qquad i\in\elts.
\end{equation}
This model has been thoroughly studied in the literature (see \cite{ACMPPRT17} for a (non-exhaustive) review) and provides a well-known example of emergence of global patterns, such as convergence to consensus or clustering, from local interaction rules. The augmented model with time-varying weights \eqref{eq:syst-gen} was also shown to exhibit richer types of long-term behavior, such as the emergence of a single (or several) leader(s) \cite{McQuadePiccoliPouradierDuteil19}. 
\begin{rem}
In the literature, the interaction function can be found of the form $\phi(x):=a(\|x\|)x$ for some continuous function $a\in C(\R^+;\R)$ (see \cite{BiccariKoZuazua19,HK}).
In other works, the interaction between agents takes the form of the gradient of an interaction potential $W:\R\rightarrow\R$, i.e. $\phi(x):=\nabla W(\|x\|)$ (see for instance \cite{CarrilloChoiHauray14}).
Here, we will keep the general notation $\phi$, also used in \cite{JabinMotsch14}, which can cover these various cases.
\end{rem}

From here onward, we will make the following assumptions on the interaction function: 
\begin{hyp}\label{hyp:phi}
The interaction function $\phi$ satisfies $\phi(0)=0$ and  $\phi\in\Lip(\R^d; \R)$, with $\|\phi\|_\Lip=L_\phi$.
\end{hyp}


Notice that at this stage, we have not made any assumptions on the interaction functions $\psi_i$. Actually, unlike the position dynamics, the weight dynamics are allowed to differ for each agent $i$. In this paper, the dependence of $\psi_i$ on the opinions $x^N$ and the weights $m^N$ will take two main forms, that will be specified in the subsequent sections.

The aim of this work is to derive the continuum limit of these dynamics, that is when the number of agents goes to infinity.
We will show that using the graph limit method, we obtain the following limit equation (that we will refer to as the graph limit equation):
\begin{equation}
\left\{\begin{array}{l}
\displaystyle \partial_t x(t,s) = \int_I m(t,s_*)\phi(x(t,\se) - x(t,s)) d\se \\
\partial_t m(t,s) = \psi(s,x(t,\cdot),m(t,\cdot)),
\end{array}\right.
\label{eq:GraphLimit-gen}
\end{equation} 
where $x\in C([0,T];L^\infty(\R^d))$ and $m\in C([0,T];L^\infty(\R))$ are {associated with} the respective continuum limits of $x^N$ and $m^N$.
Here, $s$ represents the continuous index variable taking values in $I:=[0,1]$, as introduced in \cite{Medvedev14} and \cite{BiccariKoZuazua19}, and $\psi:I\times C([0,T];L^\infty(\R^d))\times C([0,T];L^\infty(\R))\rightarrow\R$ will have to be specified. Notice that the dependence of $\psi_i^{(N)}$ on the index $i$ in the microscopic dynamics \eqref{eq:syst-gen} is translated by the dependence of $\psi$ on the continuous variable $s$ in the limit \eqref{eq:GraphLimit-gen}. Similarly, the dependence of $\psi_i$ on all agents' opinions $x^N(t)$ and weights $m^N(t)$ is encoded by the non-local dependence of $\psi$ on the functions $x(t,\cdot)$ and $m(t,\cdot)$.

\begin{example}\label{ex}
A simple example of mass dynamics depending non-locally on the opinions and weights can be given by functions of the form:
\begin{equation}
\label{eq:psi-part}
\displaystyle \psi(s,x(t,\cdot),m(t,\cdot))= m(t,s) \int_I    m(t,\tilde{s}) S(x(t,s),x(t,\tilde{s})) d\tilde{s} 
\end{equation}
where $S:\R^d\times\R^d\rightarrow\R$. Note that in this example, $\psi$  depends on the continuous index $s$ only through $x$ and $m$.
More specific examples of mass dynamics $\psi$ will be presented in Section \ref{sec:numeric}. The choice \eqref{eq:modelsimuGL} of Section \ref{sec:numericindisting} provides another example of mass dynamics depending only on the opinions and weights, and not on the individual indices. The choice \eqref{eq:modelsimuGL2} of Section \ref{sec:numericnotindisting} provides an example of mass dynamics depending explicitly on  $s$.
\end{example}

In order to give a meaning to the limit, 
we will reformulate the discrete system \eqref{eq:syst-gen} in a continuous way, using two operators $\Pcn$ and $\Pdn$ respectively transforming vectors into piecewise-constant functions and $L^\infty$ functions into $N$-dimensional vectors.
From here onward, subscripts {in ($x_N$,$m_N$)} will indicate functions over the continuous space $I$  while superscripts  {in ($x^N$) , ($m^N$)} will indicate vectors of $(\R^d)^N$ or $\R^N$.
 
Given initial conditions $x_0\in L^\infty(I;\R^d)$ and $m_0\in L^\infty(I;\R)$ for the continuous dynamics \eqref{eq:GraphLimit-gen},
satisfying 
\begin{equation}
\label{eq:integral_egal_a_1}
\int_I m_0(s) ds =1,
\end{equation}
we can define initial conditions for the microscopic dynamics \eqref{eq:syst-gen}. For each $N\in\N$, we define 
\begin{equation}
\label{eq:ICx}
\begin{cases}
\displaystyle x^{0,N} = \Pdn(x_0) := \left( N \int_{\frac{i-1}{N}}^{\frac{i}{N}} x_0(s) ds \right)_{i \in \{1, \dots, N \}} \in (\R^d)^N \\
\displaystyle  m^{0,N} = \Pdn(m_0) := \left( N \int_{\frac{i-1}{N}}^{\frac{i}{N}} m_0(s) ds \right)_{i \in \{1, \dots, N \}} \in (\R)^N.
\end{cases}
\end{equation}
An schematic illustration of the transformation $\Pdn$ is provided in Figure \ref{fig:Transfo}.
Notice that condition \eqref{eq:integral_egal_a_1} implies that 
 \eqref{eq:sum_M} is fulfilled.\\

Now, for all $t\in [0,T]$, the solution $(x^N(t),m^N(t))$ to the microscopic system \eqref{eq:syst-gen} at time $t$ can be transformed  into a pair of piecewise-constant functions $s\mapsto(x_N(t,s),m_N(t,s))$ via the following operation: for all $s\in I$,
\begin{equation}
\label{eq:xN}
\begin{cases}
\displaystyle x_N(t,s) = \Pcn(x^{N}(t)):= \sum_{i=1}^N x_i^{N}(t) \mathbf{1}_{[\frac{i-1}{N}, \frac{i}{N})}(s)\\
\displaystyle m_N(t,s) = \Pcn(m^{N}(t)) := \sum_{i=1}^N m_i^{N}(t) \mathbf{1}_{[\frac{i-1}{N}, \frac{i}{N})}(s).
\end{cases}
\end{equation}
The transformation $\Pcn$ is also illustrated in Figure \ref{fig:Transfo}.
In turn, this transformation will allow us to define the discrete weight dynamics $\psi_i^{(N)}$ from the continuous ones. More specifically, given a functional $\psi: I \times L^2(I;\R^d)\times L^2(I;\R) \rightarrow \R$, we 
define
$\psi_i^{(N)}$ in the following way:
\begin{equation}\label{eq:psi}
\forall i\in\elts,\qquad \psi_i^{(N)}(x^{N}(t),m^{N}(t)) = N \int_{\frac{i-1}{N}}^{\frac{i}{N}} \psi(s,x_N(t,s),m_N(t,s)) ds, 
\end{equation}
where $x_N=\Pcn(x^N)$ and $m_N=\Pcn(m^N)$ as defined in \eqref{eq:xN}. 

\begin{example}
With this definition, for $\psi$ taken as \eqref{eq:psi-part} in Example \ref{ex}, the system \eqref{eq:syst-gen} becomes 
\begin{equation*}
\begin{cases}
\displaystyle \frac{d}{dt} x_i^{N}(t) = \frac{1}{N} \sumj m_j^N(t) \phi(x_j^{N}(t)-x_i^{N}(t))\\
\displaystyle  \frac{d}{dt} m_i^{N}(t) = \frac{1}{N} m_i^{N}(t) \sum_{j=1}^N m_j^{N}(t) S(x_i^{N}(t),x_j^{N}(t)) , \qquad i\in\elts.
\end{cases}
\end{equation*}
\end{example}

\begin{figure}
\includegraphics[trim= 2cm 6cm 2cm 6cm, clip=true, width=\textwidth]{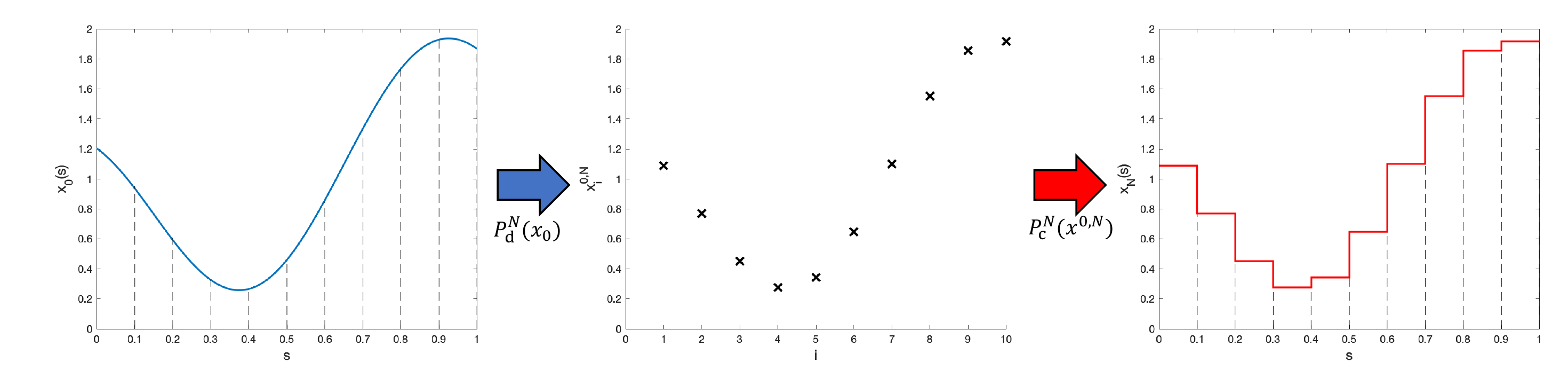}
\caption{Illustration of the two transformations $\Pdn$ and $\Pcn$ for $N=10$ and $x_0\in L^\infty(I;\R)$. Notice that $\lim_{N\rightarrow\infty} \Pcn(\Pdn(x_0))= x_0$.}\label{fig:Transfo}
\end{figure}

These transformations allow us to reveal an equivalence between the microscopic system \eqref{eq:syst-gen} and the continuous one \eqref{eq:GraphLimit-gen} in the case of piecewise-constant functions. More precisely, we  have the following
\begin{prop}\label{prop:equiv}
The vectors $(x^{N},m^{N})\in \mathcal{C}([0,T]; \R^d)^{N}\times \mathcal{C}([0,T];\R)^{N}$ satisfy the differential system~\eqref{eq:syst-gen} with initial condition $x^{0,N}\in(\R^d)^N$, $m^{0,N}\in\R^d$,
 and mass dynamics given by \eqref{eq:psi},
  if and only if the piecewise constant functions $x_N\in C([0,T];L^2(I;\R^d))$ and $m_N\in \mathcal{C}([0,T];L^2(I;\R))$ defined by
\eqref{eq:xN} 
satisfy the following system of integro-differential equations:
\begin{equation}\label{eq:syst-contN}
\begin{cases}
\displaystyle \partial_t x_N(t,s) = \int_I m_N(t,\se)\, \phi(\txn(t,\se)-\txn(t,s)) \,d\se\\
\displaystyle \partial_t \tmn(t,s) = N \int_{\frac{1}{N}\floor{sN}}^{\frac{1}{N}(\floor{sN}+1)} \psi(\se, \txn(t,\cdot),\tmn(t,\cdot)) \,d\se 
\end{cases}
\end{equation}
with initial conditions $\displaystyle x_N(0,s) = \sum_{i=1}^N x_i^{0,N} \mathbf{1}_{[\frac{i-1}{N}, \frac{i}{N})}(s)$ and  $\displaystyle m_N(0,s) = \sum_{i=1}^N m_i^{0,N} \mathbf{1}_{[\frac{i-1}{N}, \frac{i}{N})}(s)$.
\end{prop}
\begin{rem}\label{Rem:x0m0}
Notice that according to the Lebesgue differentiation theorem, under the assumptions that $x_0\in L^1(I)$ and $m^0\in L^1(I)$, it holds
$$\lim_{N\rightarrow\infty} x_N(0,s) = x_0(s) \quad \text{ and } \quad \lim_{N\rightarrow\infty}m_N(0,s) = m_0(s)  $$  
for almost every $s\in I$. Figure \ref{fig:Transfo} illustrates the relationship between $x_0$ and $x_N(0,\cdot)=\Pcn(\Pdn(x_0))$ for a finite $N$.
\end{rem}

The main interest of this proposition is to have recast the discrete system \eqref{eq:syst-gen} into the framework of $L^\infty(I)$ functions. This will allow us to stay in this framework to state the convergence to a limit function, also belonging to $L^\infty(I)$.
Using this trick, we can now state one of the main results of this paper, namely the convergence of the solution to system \eqref{eq:syst-contN} to the solution to system \eqref{eq:GraphLimit-gen}.
\begin{theorem}\label{Th:GraphLimit}
Let $x_0\in L^\infty(I;\R^d)$ and $m_0\in L^\infty(I;\R)$ satisfying \eqref{eq:integral_egal_a_1}, and consider a functional $\psi:I \times L^2(I;\R^d)\times L^2(I;\R) \rightarrow \R$.
Suppose that the function $\phi$ satisfies Hyp. \ref{hyp:phi}, that 
$x\mapsto\psi(\cdot,x,m)$ and $m\mapsto\psi(\cdot,x,m)$ 
are uniformly Lipschitz functions in the $L^2$ norm, and that $m\mapsto\psi(\cdot,x,m)$ is sublinear in the $L^\infty$ norm.
Let $x^{0,N}=\Pdn(x_0)$ and $m^{0,N}=\Pdn(m_0)$ given by \eqref{eq:ICx}.
Then the solution $(\txn,\tmn)$ to \eqref{eq:syst-contN} with initial conditions 
\[\displaystyle x_N(0,s) = \sum_{i=1}^N x_i^{0,N} \mathbf{1}_{[\frac{i-1}{N}, \frac{i}{N})}(s) \quad \text{ and } \quad \displaystyle m_N(0,s) = \sum_{i=1}^N m_i^{0,N} \mathbf{1}_{[\frac{i-1}{N}, \frac{i}{N})}(s)
\]
  converges when $N$ tends to infinity in the $\mathcal{C}([0,T];L^2(I))$ topology.
  More specifically, there exists $(x,m)\in \mathcal{C}([0,T];L^2(I,\R^d))\times \mathcal{C}([0,T];L^2(I,\R))$ such that 
\begin{equation}\label{eq:convxm}
\displaystyle \|x-\txn\|_{\mathcal{C}([0,T];L^2(I,\R^d))} \xrightarrow[N\rightarrow+\infty]{} 0 
\quad \text{ and } \quad 
\|m-\tmn\|_{\mathcal{C}([0,T];L^2(I,\R))} \xrightarrow[N\rightarrow+\infty]{} 0.
\end{equation}
Furthermore, the limit functions $x$ and $m$ are solutions to the integro-differential system \eqref{eq:GraphLimit-gen} supplemented by the initial conditions $x(0,\cdot) = x_0$ and $m(0,\cdot) = m_0$.
\end{theorem}

Secondly, we will show that if the mass dynamics satisfy an \emph{indistinguishability} property (that we will define in Section \ref{sec:indisting}), we can take this limit process further and derive the mean-field limit of the system from the continuum graph limit.  The mean-field limit will be shown to be a solution to the following transport equation with source (see \cite{PiccoliRossi14, PouradierDuteil21}): 
\begin{equation*}
\partial_t \mu_t(x) + \nabla\cdot ( V[\mu_t](x) \mu_t(x)) = h[\mu_t](x)
\end{equation*}
where the non-local transport vector field $V$ and source term $h$ will be respectively defined from the interaction function $\phi$ and the mass dynamics $\psi$ (see Theorem \ref{Th:mfl}).

\section{The Graph Limit}
\label{sec:graphlim}

From here onward, in all of Section \ref{sec:graphlim}, we will assume the following properties for the mass dynamics $\psi$.
\begin{hyp}\label{hyp:psi}
The function $\psi:I\times L^\infty(I;\R^d)\times L^\infty(I;\R)$ is assumed to satisfy the following Lipschitz properties: there exists $L_\psi>0$ such that for all $(x_1,x_2,m_1,m_2)\in L^2(I)^4$,
\begin{equation}\label{eq:psilip2}
\begin{cases}
 \|\psi(\cdot,x_1,m_1)-\psi(\cdot,x_2,m_1)\|_{L^2(I)} \, \leq \,L_\psi \|x_1-x_2\|_{L^2(I)}\\
 \|\psi(\cdot,x_1,m_1)-\psi(\cdot,x_1,m_2)\|_{L^2(I)} \, \leq \, L_\psi \|m_1-m_2\|_{L^2(I)}.
\end{cases}
\end{equation}
Assume also that there exists $C_\psi>0$ such that for all $(x,m)\in L^\infty(I,\R^d\times\R)$, for all $s\in I$, 
\begin{equation}\label{eq:psisublin}
|\psi(s, x, m)| \leq C_\psi( 1+\|m\|_{L^\infty(I)}).
\end{equation}
\end{hyp}

Although the assumption of sublinear growth \eqref{eq:psisublin} may seem restrictive, it is necessary in order to prevent the blow-up in finite-time of the weight function $m$.
It is coherent with the framework of the Graph Limit developed in~\cite{Medvedev14} on graphs with $L^\infty$ weights. Indeed, we can view our system \eqref{eq:GraphLimit-gen} as the evolution of the opinions $x$ on a weighted non-symmetric graph with weights $W(s,\se)=m(t,\se)$.

\subsection{Well posedness of the graph limit model}
This paragraph is devoted to proving the existence and uniqueness of a solution to the graph limit equation \eqref{eq:GraphLimit-gen}. 
We start by proving existence and uniqueness for the decoupled system in the general case, where the mass dynamics are assumed to depend on the continuous index $s$.

In order to prove the well-posedness of system \eqref{eq:GraphLimit-gen}, we start by studying a decoupled system in which the dynamics of the opinions and of the weights are independent.

\begin{lemma}\label{Lemma:WellPos-decoupled}
Let ${\tilde{x} \in \CLx}$ and ${\tm \in \CLm}$. Let $x_0\in L^\infty(I;\R^d)$ and $m_0\in L^\infty(I;\R)$.
Let $\ba$ satisfy Hyp. \ref{hyp:phi} and $\psi$ satisfy Hyp. \ref{hyp:psi}. 
Then for any $T>0$, there exists a unique solution $(x,m)\in\CLu$ to the decoupled integro-differential system 
\begin{equation*}
\left\{\begin{array}{l}
\displaystyle \partial_t x(t,s) = \int_I \tm(t,s_*)\ba(x(t,s) - x(t,\se)) d\se; \qquad x(\cdot,0)=x_0 \\
\partial_t m(t,s) = \psi(s,\tx(t,\cdot),m(t,\cdot)); \qquad m(\cdot,0)=m_0.
\end{array}\right.
\end{equation*}
\end{lemma}

\begin{proof}
Since the two equations are decoupled, we will treat them independently. 
Let $\tT>0$ (to be specified later such that $\tT \leq T$) and $\Mx$ be a metric subspace of $\mathcal{C}([0,\tT];L^\infty(I;\R^d))$ consisting of functions $x$ satisfying $x(\cdot,0) = x_0$.
Let $\Kx$ be the operator defined by:
$$
\begin{array}{l l l }
\Kx: & \Mx & \rightarrow \Mx \\
 & x & \displaystyle  \mapsto (\Kx x):(t,s)\mapsto x_0(s) + \int_0^t \int_I \tm(\tau,s_*)\ba(x(\tau,s) - x(\tau,\se)) d\se d\tau.
\end{array}
$$
 We will show that $\Kx$ is contracting for the norm $\|\cdot\|_{\Mx} := \sup_{[0,\tT]} \esssup_{I} \|\cdot\|$.
Let $(x_1,x_2)\in \mathcal{C}([0,\tT];L^\infty(I;\R^d))^2)$. 
Then for all $s\in I$, for all $t\leq \tT$,
\begin{equation*}
\begin{split}
\| \Kx x_1 - \Kx x_2 \| (t,s) & = \left \|  \int_0^t \int_I \tm(\tau,s_*)\left[ \ba(x_1(\tau,s) - x_1(\tau,\se))-\ba(x_2(\tau,s) - x_2(\tau,\se))\right] d\se d\tau \right \| \\
& \leq  \int_0^t \int_I |\tm(\tau,s_*)| \Lphi \left\| (x_1(\tau,s) - x_1(\tau,\se))- (x_2(\tau,s) - x_2(\tau,\se))\right \| d\se d\tau \\
 & \leq  \int_0^t \int_I |\tm(\tau,s_*)| \Lphi \left( \| x_1(\tau,s)- x_2(\tau,s)\| + \| x_1(\tau,\se) - x_2(\tau,\se)\| \right) d\se d\tau  \\
\end{split}
\end{equation*}
from which we get: 
$$
\| \Kx x_1 - \Kx x_2 \|_{\Mx} \leq 2 \Lphi \tT \sup_{t\in [0,T]} \|\tm(t,\cdot)\|_{L_1(I;\R)} \quad \|x_1-x_2\|_{\Mx}.
$$
We remark that $\sup_{t\in [0,T]} \|\tm(t,\cdot)\|_{L_1(I;\R)}$ is well defined since  $\tm\in \CLm$ and $I=[0,1]$.
Since $\tm$ is given, choosing $\tT\leq (4 \Lphi  \sup_{t\in [0,T]} \|\tm(t,\cdot)\|_{L_1(I;\R)})^{-1}$ ensures that $\Kx$ is contracting on $[0,\tT]$.
By the Banach contraction mapping principle, there exists a unique solution $x\in \mathcal{C}([0,\tT];L^\infty(I;\R^d))$.
We then take $x(\cdot,\tT)$ as the initial condition, and the local solution can be extended to $[0, 2\tT]$, and by repeating the same argument, to $[0,T]$. Moreover, since the integrand in $K_{x_0}$ is continuous as a map $L^\infty(I) \to L^\infty(I) $,  $x$ is continuously differentiable and $x$ belongs to $\mathcal{C}^1([0,T];L^\infty(I;\R^d))$.\\
We now show existence and uniqueness of $m\in \CLmm$, solution to the second decoupled equation. 
Let $\Mm$ be a metric subspace of $\mathcal{C}([0,\tT];L^2(I;\R))$ consisting of functions $m$ satisfying $m(\cdot,0) = m_0$ with again $\tT>0$ to be specified later such that $\tT \leq T$.
Let $\Km$ be the operator defined by:
$$
\begin{array}{l l l }
\Km: & \Mm & \rightarrow \Mm \\
 & m &\displaystyle   \mapsto (\Km m):(t,s)\mapsto m_0(s) + \int_0^t \psi(s,\tx(\tau,\cdot),m(\tau,\cdot)) d\tau.
\end{array}
$$
 We will show that $\Km$ is contracting for the norm $\|\cdot\|_{\Mm} := \sup_{[0,\tT]} \|\cdot\|_{L^2(I)}$.
Let $(m_1,m_2)\in \mathcal{C}([0,\tT];L^2(I;\R))^2$.
Then for all $s\in I$, for all $t\leq \tT$,
\begin{equation*}
\begin{split}
\int_I | \Km m_1 - \Km m_2 |^2 (t,s) ds & = \int_I \left |  \int_0^t \psi(s,\tx(\tau,\cdot),m_1(\tau,\cdot))- \psi(s,\tx(\tau,\cdot),m_2(\tau,\cdot)) d\tau \right | ^2 ds \\
& \leq t \int_0^t \int_I \left | \psi(s,\tx(\tau,\cdot),m_1(\tau,\cdot))- \psi(s,\tx(\tau,\cdot),m_2(\tau,\cdot)) \right | ^2  ds \, d\tau \\
& \leq t \int_0^t L_\psi^2 \|m_1(\tau,\cdot)-m_2(\tau,\cdot)\|_{L^2(I)}^2 \, d\tau,
\end{split}
\end{equation*}
where the first inequality is a consequence of Cauchy-Schwarz and Jensen's inequalities and Fubini's theorem, and the second inequality comes from \eqref{eq:psilip2}.
We obtain: 
\begin{equation*}
\sup_{[0,\tT]} \| \Km m_1 - \Km m_2 \|_{L^2(I)}^2  \leq \tT^2 L_\psi^2 \sup_{[0,\tT]} \|  m_1 - m_2 \|_{L^2(I)}^2 
\end{equation*}
which implies: $\| \Km m_1 - \Km m_2 \|_{\Mm} \leq \tT L_\psi \| m_1 - m_2 \|_{\Mm}$.
Thus, if $\tT\leq \frac{1}{2L_\psi}$, by the Banach contraction mapping principle, there exists a unique solution $m\in \mathcal{C}^1([0,\tT];L^2(I;\R))$.
We then take $m(\tT),\cdot)$ as the initial condition, and the local solution can be extended to $[0, 2\tT]$, and by repeating the same argument, to $[0,T]$.
We thus showed that there exists a unique solution $m\in \mathcal{C}([0,T];L^2(I;\R))$. To prove that $m\in \CLm$, we will use the second assumption on $ \psi$ given by the sub-linearity~\eqref{eq:psisublin}. For all $(t,s)\in I\times [0,T]$,
$$
|m(t,s)| = \left | m_0(s) + \int_0^t \psi(s,x(\tau,\cdot),m(\tau,\cdot)) d\tau \right|
\leq |m_0(s)| + \int_0^t C_\psi (1+ \|m(\tau,\cdot)\|_{L^\infty(I)}) d\tau. 
$$
This implies
$$
\| m(t,\cdot)\|_{L^\infty(I)} \leq (\|m_0\|_{L^\infty(I)} + C_\psi t ) +  \int_0^t C_\psi \|m(\tau,\cdot)\|_{L^\infty(I)} d\tau
$$
and from Gronwall's lemma, 
\begin{equation}\label{eq:mbound}
\| m(t,\cdot)\|_{L^\infty(I)} \leq (\|m_0\|_{L^\infty(I)} + C_\psi t ) e^{C_\psi t}.
\end{equation}
Hence $m\in \CLm$. As previously, since the integrand in $K_{m_0}$ is continuous as a map $L^\infty(I) \to L^\infty(I) $,  $m$ is continuously differentiable and $m$ belongs to $\mathcal{C}^1([0,T];L^\infty(I;\R))$. This concludes the proof.

\end{proof}

\begin{rem}\label{Rem:Gronwall}
We have used the following result: 
If $u(t)\leq \alpha(t) + \int_0^t \beta(\tau) u(\tau) d\tau$, where $\beta$ is positive and $\alpha$ is non-decreasing,
then  $u(t) \leq \alpha(t)\exp(\int_0^t \beta(\tau) d\tau )$.
\end{rem}

By the previous lemma, we have proven that the two decoupled integro-differential equations are well-posed in $\CLu$. Using this result, we are now ready to demonstrate the well-posedness of the fully coupled system \eqref{eq:GraphLimit-gen}.

\begin{theorem}\label{Th:GL-wellpos}
 Let $x_0\in L^\infty(I;\R^d)$ and $m_0\in L^\infty(I;\R)$.
Let $\ba$ satisfy Hyp. \ref{hyp:phi} and $\psi$ satisfy Hyp. \ref{hyp:psi}. 
Then for any $T>0$, there exists a unique solution $(x,m)\in {\CLu}$ to the integro-differential system 
\begin{equation}
\label{eq:graphlimitcoupled}
\left\{\begin{array}{l}
\displaystyle \partial_t x(t,s) = \int_I m(t,s_*)\ba(x(t,s) - x(t,\se)) d\se; \qquad x(\cdot,0)=x_0 \\
\partial_t m(t,s) = \psi(s,x(t,\cdot),m(t,\cdot)); \qquad m(\cdot,0)=m_0.
\end{array}\right.
\end{equation}
\end{theorem}

\begin{proof}
The proof will consist of proving the convergence of a sequence of functions $(x^n,m^n)_{n\in\N}$ of $\mathcal{C}([0,T];L^\infty(I;\R^d \times \R))$ defined as follows: 
\begin{itemize}
\item For almost every $s \in I$, for all $t \in [0,T]$, $x^0(t,s) = x_0(s)$ and $m^0(t,s) = m_0(s)$.
\item For all $n\in \N^*$, 
\begin{equation*}
\left\{\begin{array}{l}
\displaystyle \partial_t x^n(t,s) = \int_I m^{n-1}(t,s_*)\ba(x^n(t,s) - x^n(t,\se)) d\se; \qquad x^n(\cdot,0)=x_0 \\
\partial_t m^n(t,s) = \psi(s,x^{n-1}(t,\cdot),m^n(t,\cdot)); \qquad m^n(\cdot,0)=m_0.
\end{array}\right.
\end{equation*}
\end{itemize}
From Lemma \ref{Lemma:WellPos-decoupled}, the sequence is well defined and each term $(x^n,m^n)$ is indeed in  {$\mathcal{C}([0,T];L^\infty(I;\R^d \times \R))$}.
We begin by highlighting the fact that the terms of the sequence are uniformly bounded in $L^\infty(I;\R^d\times\R)$.
Indeed, from equation \eqref{eq:mbound}, we know that for every $n\in \N$, for all $t\in [0,T]$,
$$ \| m^n(t,\cdot)\|_{L^\infty(I)} \leq M_T := (\|m_0\|_{L^\infty(I)} + C_\psi T ) e^{C_\psi T}.$$
Moreover, we can now use this bound to estimate the growth of $x$, noticing that Hypothesis \ref{hyp:phi} implies that $\|\ba(r)\| \leq \Lphi \|r\|$, from the fact that $\ba(0)=0$. We estimate:
$$
\|x^n(t,s)\| \leq \|x^n(0,s)\| + M_T \int_0^t  \int_I \| \ba(x^n(\tau,\se)-x^n(\tau,s) ) \|  d\se d\tau \leq \|x_0(s)\| + 2 M_T \Lphi \int_0^t   \|x^n(\tau,\cdot) \|_{L^\infty(I)} d\tau .
$$
Then Gronwall's lemma implies that for every $n\in \N$, for all $t\in [0,T]$,
\begin{equation*}
\|x^n(t,\cdot)\|_{L^\infty(I)} \leq X_T := \|x_0\|_{L^\infty(I)} e^{2 M_T \Lphi T}.
\end{equation*}
We now prove that the sequence $(x^n,m^n)_{n\in \N}$ is a Cauchy sequence.
For every $n\in \N^*$,
\begin{equation*}
\begin{split}
|x^{n+1}-x^n|(t,s) = & \bigg | \int_0^t  \int_I m^{n}(\tau,s_*)\ba(x^{n+1}(\tau,s) - x^{n+1}(\tau,\se)) d\se d\tau \\
& - \int_0^t \int_I m^{n-1}(\tau,s_*)\ba(x^n(\tau,s) - x^n(\tau,\se)) d\se d\tau \bigg | \\
= & \bigg | \int_0^t  \int_I (m^{n}(\tau,s_*)-m^{n-1}(\tau,s_*)) \ba(x^{n+1}(\tau,s) - x^{n+1}(\tau,\se)) d\se d\tau \\
& + \int_0^t \int_I m^{n-1}(\tau,s_*)\left( \ba(x^{n+1}(\tau,s) - x^{n+1}(\tau,\se)) - \ba(x^n(\tau,s) - x^n(\tau,\se)) \right) d\se d\tau \bigg | \\
\leq &  \int_0^t  \int_I |m^{n}(\tau,s_*)-m^{n-1}(\tau,s_*)| 2 \Lphi\, X_T \, d\se d\tau \\
& + \int_0^t \int_I M_T \, \Lphi \|(x^{n+1}(\tau,s) - x^{n+1}(\tau,\se)) - (x^n(\tau,s) - x^n(\tau,\se)) \| d\se d\tau  \\
\end{split}
\end{equation*}
Then for all $t\in [0,T]$, for every $n\in \N^*$, 
\begin{equation*}
\begin{split}
|x^{n+1}-x^n|^2(t,s) \leq & 8 t  \Lphi^2 \, X_T^2 \int_0^t  \int_I |m^{n}(\tau,s_*)-m^{n-1}(\tau,s_*)|^2  \, d\se d\tau \\
& + 2 \Lphi^2 M_T^2 t \int_0^t \int_I 2(\|(x^{n+1}(\tau,s) - x^n(\tau,s)\|^2 + \| x^{n+1}(\tau,\se)- x^n(\tau,\se) \|^2) d\se d\tau
\end{split}
\end{equation*}
and we get
\begin{equation*}
\begin{split}
\|x^{n+1}-x^n\|_{L^2(I)}^2(t) \leq & 8 t  \Lphi^2 \, X_T^2 \int_0^t  \|m^{n}-m^{n-1}\|_{L^2(I)}^2 (\tau)  d\tau + 8t \Lphi^2 M_T^2  \int_0^t \|x^{n+1} - x^n\|_{L^2(I)}^2 (\tau) d\tau.
\end{split}
\end{equation*}
A similar computation for $m$ gives for every $n\in\N$
\begin{equation*}
\begin{split}
|m^{n+1}-m^n|(t,s) = & \bigg | \int_0^t  (\psi(s,x^n(\tau,\cdot),m^{n+1}(\tau,\cdot)) - \psi(s,x^{n-1}(\tau,\cdot),m^{n}(\tau,\cdot)) d\tau \bigg | \\
\leq &  \int_0^t \bigg ( |\psi(s,x^n(\tau,\cdot),m^{n+1}(\tau,\cdot)) - \psi(s,x^n(\tau,\cdot),m^{n}(\tau,\cdot)) | \\
& + \psi(s,x^n(\tau,\cdot),m^{n}(\tau,\cdot)) - \psi(s,x^{n-1}(\tau,\cdot),m^{n}(\tau,\cdot)) | \bigg ) d\tau . \\
\end{split}
\end{equation*}
Squaring and integrating yields:
\begin{equation}\label{eq:ineqmn}
\begin{split}
\|m^{n+1}-m^n\|_{L^2(I)}^2(t) 
\leq &  2 t \int_0^t \int_I |\psi(s,x^n(\tau,\cdot),m^{n+1}(\tau,\cdot)) - \psi(s,x^n(\tau,\cdot),m^{n}(\tau,\cdot)) |^2 ds \, d\tau \\
& + 2t \int_0^t \int_I | \psi(s,x^n(\tau,\cdot),m^{n}(\tau,\cdot)) - \psi(s,x^{n-1}(\tau,\cdot),m^{n}(\tau,\cdot)) | ^2 ds \, d\tau  \\
\leq &  2 t L_\psi^2 \int_0^t  \|m^{n+1}-m^{n}\|_{L^2(I)}^2(\tau) d\tau +  2 t L_\psi^2  \int_0^t \| x^n - x^{n-1}\|_{L^2(I)}^2(\tau) d\tau .
\end{split}
\end{equation}
Thus, denoting $A_T = \max(8 T  \Lphi^2 \, X_T^2,\, 8 T  \Lphi^2 \,  M_T^2,\,  2 T L_\psi^2)$, and denoting by $(u_n)_{n\in\N}$ the sequence defined by
$u_n(t) := \|x^{n+1}-x^n\|_{L^2(I)}^2(t) + \|m^{n+1}-m^n\|_{L^2(I)}^2(t)$,  we obtain for every $n\in \N^*$ and all $t\in [0,T]$:
\begin{equation*}
\begin{split}
u_{n}(t) \leq 
A_T \int_0^t u_n(\tau) d\tau  + A_T \int_0^t u_{n-1}(\tau)  d\tau .
\end{split}
\end{equation*}
Since $t\mapsto  A_T \int_0^t u_{n-1}(\tau)  d\tau$ is non-decreasing, Gronwall's lemma implies (see Remark \ref{Rem:Gronwall}):
$$
u_{n}(t) \leq A_T \; e^{A_T T} \int_0^t u_{n-1}(\tau)  d\tau .
$$ 
Denoting $U_0 := \sup_{[0,T]} u_0(t)$, one can easily show by induction that for all $t\in [0,T]$, for all $n\in \N$, 
$$
u_n(t) \leq \frac{(A_T e^{A_T T}t)^n}{n!} U_0.
$$
This is the general term of a convergent series, hence for all $t\in [0,T]$, $\lim_{n\rightarrow \infty } u_n(t) = 0$, which implies that for all $t\in [0,T]$, $\|x^{n+1}-x^n\|_{L^2(I)}(t)$ and $\|m^{n+1}-m^n\|_{L^2(I)}(t)$ also converge to $0$ as $n$ tends to infinity.
Thus, $(x^n)_{n\in\N}$ and $(m^n)_{n\in\N}$ are Cauchy sequences in the Banach spaces $\mathcal{C}([0,T];L^2(I;\R^d))$ and $\mathcal{C}([0,T];L^2(I;\R))$.
One can easily show that their limits $(x,m)$ satisfy the integro-differential system \eqref{eq:GraphLimit-gen}. Furthermore, from the uniform bounds on $\|x^n(t,\cdot)\|_{L^\infty(I)}$ and $\|m^n(t,\cdot)\|_{L^\infty(I)}$, we deduce that $(x,m)\in \mathcal{C}([0,T];L^\infty(I;\R^d \times \R))$, with $\|x(t,\cdot)\|_{L^\infty(I)} \leq X_T$ and $\|m(t,\cdot)\|_{L^\infty(I)} \leq M_T$ for all $t\in [0,T]$. Finally, as previously, looking at the integrand in the integral formulation of \eqref{eq:graphlimitcoupled}, we deduce that $(x,m)\in \mathcal{C}^1([0,T];L^\infty(I;\R^d \times \R))$ which concludes the proof of existence.
\\

Let us now deal with the uniqueness. Let us assume that there exist two solutions to the equation~\eqref{eq:graphlimitcoupled} denoted $(x_1,m_1)$ and  $(x_2,m_2)$ with the same initial condition. Then, we have 
\begin{equation*}
\left\{\begin{array}{l}
\displaystyle (x_1-x_2)(t,s) = \int_0^t \int_I m_1(\tau,\se) \phi(x_1(\tau,\se)-x_1(\tau,s))d\se d\tau\\
\displaystyle ~~~~~~~~~~~~~~~~~~~~~~~~~~~~~~~~~~~~~  -  \int_0^t \int_I m_2(\tau,\se) \phi(x_2(\tau,\se)-x_2(\tau,s))d\se d\tau \\
\displaystyle (m_1-m_2)(t,s) =  \int_0^t (\psi(s,x_1(\tau,\cdot), m_1(\tau,\cdot)) -\psi(s,x_2(\tau,\cdot), m_2(\tau,\cdot))) d\tau
\end{array}\right.
\end{equation*}
that we rewrite 
\begin{equation*}
\left\{\begin{array}{l}
\displaystyle (x_1-x_2)(t,s) = \int_0^t \int_I (m_1-m_2)(\tau,\se) \phi(x_1(\tau,\se)-x_1(\tau,s))d\se d\tau\\
\displaystyle ~~~~~~~~~~~~~~~~~~~~~~~~~~~~~~~~~~~~~  +  \int_0^t \int_I m_2(\tau,\se) (\phi(x_1(\tau,\se)-x_1(\tau,s)) -\phi(x_2(\tau,\se)-x_2(\tau,s)))d\se d\tau \\
\displaystyle (m_1-m_2)(t,s) =  \int_0^t (\psi(s,x_1(\tau,\cdot), m_1(\tau,\cdot)) -\psi(s,x_1(\tau,\cdot), m_2(\tau,\cdot))) d\tau\\
\displaystyle ~~~~~~~~~~~~~~~~~~~~~~~~~~~~~~~~~~~~~  + \int_0^t (\psi(s,x_1(\tau,\cdot), m_2(\tau,\cdot)) -\psi(s,x_2(\tau,\cdot), m_2(\tau,\cdot))) d\tau
\end{array}\right.
\end{equation*}
Thus, we have 
\begin{equation*}
\begin{array}{l}
\displaystyle |x_1-x_2|(t,s) \leq \int_0^t \int_I 2 \Lphi X_T |m_1-m_2|(\tau,\se) d\se  d\tau + \int_0^t \int_I M_T\Lphi (|x_1-x_2|(\tau,\se) + |x_1-x_2|(\tau,s)) d\se  d\tau
\end{array}
\end{equation*}
from which we deduce 
\begin{equation*}
\begin{array}{l}
\displaystyle |x_1-x_2|^2(t,s) \leq  8 \Lphi^2 X_T^2 t  \int_0^t \int_I |m_1-m_2|^2(\tau,\se) d\se d\tau \\ 
 \displaystyle ~~~~~~~~~~~~~~~~~~~~~~~~~~~~~~~~~~~~~  + 2  M_T^2 \Lphi^2 t   \int_0^t \int_I |x_1-x_2|^2(\tau,\se) d\se d\tau +   2  M_T^2 \Lphi^2 t   \int_0^t  |x_1-x_2|^2(\tau,s)  d\tau.
\end{array}
\end{equation*}
Thus, we have, 
\begin{equation*}
\begin{array}{l}
\displaystyle \|x_1-x_2\|^2_{L^2(I)}(t) \leq  A_T \left( \int_0^t \|m_1-m_2\|^2_{L^2(I)}(\tau)  d\tau + \int_0^t \|x_1-x_2\|^2_{L^2(I)}(\tau) \right) d\tau.
\end{array}
\end{equation*}
Similarly, 
\begin{equation*}
\begin{array}{rl}
\displaystyle \|m_1-m_2\|^2_{L^2(I)}(t) &\leq 2 t  \int_0^t \|\psi(\cdot,x_1(\tau),m_1(\tau)) - \psi(\cdot,x_1(\tau),m_2(\tau)) \|_{L^2(I)}^2 d\tau \\ 
 &\displaystyle ~~~~~~~~~~~~~~~~~~~~~~~~~~~~~~~~~~~~~  +  2 t  \int_0^t \|\psi(\cdot,x_1(\tau),m_2(\tau)) - \psi(\cdot,x_2(\tau),m_2(\tau)) \|_{L^2(I)}^2 d\tau\\
  &\displaystyle  \leq 2t L_\psi  \int_0^t \left(\|m_1-m_2\|^2_{L^2(I)}(\tau) + \|x_1-x_2\|^2_{L^2(I)}(\tau) \right) d\tau. 
\end{array}
\end{equation*}
Finally, 
\begin{equation*}
\|x_1-x_2\|^2_{L^2(I)}(t) +  \|m_1-m_2\|^2_{L^2(I)}(t) \leq  2 A_T \int_0^t \left(\|x_1-x_2\|^2_{L^2(I)}(\tau) + \|m_1-m_2\|^2_{L^2(I)}(\tau)  \right) d\tau.
\end{equation*}
By Gronwall lemma, we deduce that, for all $t \in [0,T]$, 
\begin{equation*}
\|x_1-x_2\|^2_{L^2(I)}(t) +  \|m_1-m_2\|^2_{L^2(I)}(t) = 0,
\end{equation*}
which concludes the proof of uniqueness.
\end{proof}

\subsection{Well posedness of the microscopic system}

In this paragraph, we state the  existence and uniqueness results  of a solution to the discrete system \eqref{eq:syst-gen}.
We do not provide the proof since it consists on a straightforward adaptation at the discrete level of the proof established for the continuous case, the graph limit equation, based on a use of the fixed point theorem.  

\begin{theorem}
Let $(x^{0,N},m^{0,N}) \in \R^{dN} \times \R^N$.
Let $\ba$ satisfy Hyp.~\ref{hyp:phi}, $\psi$ satisfy Hyp.~\ref{hyp:psi}, 
and $\psi_i^{(N)}$ be defined  by \eqref{eq:psi}.
Then for any $T>0$, there exists a unique solution $(x^{N},m^{N}) \in \mathcal{C}^1([0,T];\R^{dN} \times \R^N)$ to the discrete system \eqref{eq:syst-gen} with initial condition $(x^{0,N},m^{0,N}) \in \R^{dN} \times \R^N$. 
Moreover, there exists constants $\overline{X}$ and  $\overline{M}$ such that for all $t \in [0,T]$, for all $i \in \{1, \dots, N \}$,
\begin{equation}
\label{eq:bornex}
\| x_i^{(N)}(t)\| \leq  \overline{X} 
\end{equation}
and 
\begin{equation}
\label{eq:bornem}
| m_i^{(N)}(t)| \leq  \overline{M}.
\end{equation}
\end{theorem}

\begin{rem}
Although this theorem provides existence and uniqueness of solutions to system \eqref{eq:syst-gen} for the special class of mass dynamics given by \eqref{eq:psi}, it can easily be adapted to general mass dynamics $\psi^{(N)}_i(x^N,m^N)$ satisfying the two conditions:
there exist $L_\psi>0$ and $C_\psi>0$ such that
for all $x^N,y^N\in (\R^d)^N$, for all $m^N,p^N\in\R^N$,
\begin{equation*}
\begin{cases}
 \|\psi^{(N)}(x^N,m^N)-\psi^{(N)}(y^N,m^N)\| \, \leq \,L_\psi \|x^N-y^N\| \\
 \|\psi^{(N)}(x^N,m^N)-\psi^{(N)}(x^N,p^N)\| \, \leq \, L_\psi \|m^N-p^N\|,
\end{cases}
\end{equation*}
where $\|\cdot\|$ denotes the standard Euclidean norm in $(\R^d)^N$ or in $\R^N$,
and for every $i\in\elts$
\begin{equation*}
\displaystyle |\psi_i( x^N, m^N)| \leq C_\psi( 1 + max_{i\in\elts} |m^N_i| ).
\end{equation*}
\end{rem}

\subsection{Convergence to the graph limit equation}

Let us now  prove the main result of this article, namely that under our assumptions, the solution to the discrete problem \eqref{eq:syst-gen} converges to the solution to the integro-differential equation \eqref{eq:GraphLimit-gen} when $N$ goes to infinity.
 For clarity purposes, we restate the main theorem announced in Section \ref{sec:presentation}.
\begin{maintheorem}
Let $x_0\in L^\infty(I;\R^d)$ and $m_0\in L^\infty(I;\R)$ satisfying \eqref{eq:integral_egal_a_1}, and consider a functional $\psi:I \times L^2(I;\R^d)\times L^2(I;\R) \rightarrow \R$.
Suppose that the function $\phi$ satisfies Hyp.~\ref{hyp:phi}, and that $\psi$ satisfies Hyp.~\ref{hyp:psi}.
Let $x^{0,N}:=\Pdn(x_0)$ and $m^{0,N}:=\Pdn(m_0)$ be given by \eqref{eq:ICx}.
Then the solution $(\txn,\tmn)$ to \eqref{eq:syst-contN} with initial conditions 
\[\displaystyle x_N(0,s) = \sum_{i=1}^N x_i^{0,N} \mathbf{1}_{[\frac{i-1}{N}, \frac{i}{N})}(s) \quad \text{ and } \quad \displaystyle m_N(0,s) = \sum_{i=1}^N m_i^{0,N} \mathbf{1}_{[\frac{i-1}{N}, \frac{i}{N})}(s)
\]
  converges when $N$ tends to infinity in the $\mathcal{C}([0,T];L^2(I))$ topology.
  More specifically, there exist $(x,m)\in \mathcal{C}([0,T];L^\infty(I,\R^d))\times \mathcal{C}([0,T];L^\infty(I,\R))$ such that 
\begin{equation}\label{eq:convxm}
\displaystyle \|x-\txn\|_{\mathcal{C}([0,T];L^2(I,\R^d))} \xrightarrow[N\rightarrow+\infty]{} 0 
\quad \text{ and } \quad 
\|m-\tmn\|_{\mathcal{C}([0,T];L^2(I,\R))} \xrightarrow[N\rightarrow+\infty]{} 0.
\end{equation}
Furthermore, the limit functions $x$ and $m$ are solutions to the integro-differential system \eqref{eq:GraphLimit-gen} supplemented by the initial conditions $x(0,\cdot) = x_0$ and $m(0,\cdot) = m_0$.
\end{maintheorem}

\begin{rem}
Theorem \ref{Th:GraphLimit} proves the convergence of the solution to \eqref{eq:syst-contN} to the solution to \eqref{eq:GraphLimit-gen}. This is equivalent to proving the convergence of the solution to the discrete system \eqref{eq:syst-gen} to the solution to \eqref{eq:GraphLimit-gen}, as shown in Proposition \ref{prop:equiv}. 
\end{rem}

\begin{proof}
The proof is done following the graph limit method  of \cite{Medvedev14}.
Let $\xin:= \txn - x$ and $\zn := \tmn -m$. 
We will also use the slight abuse of notation $y(t,\cdot):=y(t)$, with $y$ standing for $x$, $\txn$, $m$ or $\tmn$.
We compute 
\begin{equation*}
\begin{split}
\frac{\partial\xin(t,s)}{\partial t} = & \int_I \tmn(t,\se)\, \ba(\txn(t,\se)-\txn(t,s)) \, d\se - \int_I m(t,\se)\, \ba(x(t,\se)-x(t,s)) \, d\se  \\
 = & \int_I \tmn(t,\se)\, \left[ \ba(\txn(t,\se)-\txn(t,s))- \ba(x(t,\se)-x(t,s)) \right] \, d\se \\
& +  \int_I (\tmn(t,\se)- m(t,\se)) \, \ba(x(t,\se)-x(t,s)) \, d\se.
\end{split}
\end{equation*}
By multiplying by $\xin$ and integrating over $I$, we obtain
\begin{equation}\label{eq:dxi}
\begin{split}
\frac{1}{2}\int_I \frac{\partial\xin(t,s)^2}{\partial t} ds
 = & \int_{I^2} \tmn(t,\se)\, \left[ \ba(\txn(t,\se)-\txn(t,s))- \ba(x(t,\se)-x(t,s))  \right] \xin(t,s) \, d\se\, ds \\
& +  \int_{I^2} \zn(t,\se) \xin(t,s) \, \ba(x(t,\se)-x(t,s)) \, d\se\, ds.
\end{split}
\end{equation}
We study the first term. Since the solution to \eqref{eq:syst-gen} satisfies \eqref{eq:bornex}-\eqref{eq:bornem}, we have 
\begin{equation}
\label{eq:bornemN}
\sup_{t \in [0,T]} \esssup_{s \in I} |m_N(t,s)| \leq \overline{M}.
\end{equation}
Then, since $\ba$ is Lipschitz, there exists $L>0$ such that
\begin{equation*}
\begin{split}
 & \left | \int_{I^2} \tmn(t,\se)\, \left[ \ba(\txn(t,\se)-\txn(t,s))- \ba(x(t,\se)-x(t,s)) 
  \right]  \xin(t,s)\, d\se\, ds \right| \\ \leq \;& \overline{M} L \int_{I^2} |\xin(t,s)| \,\left| \txn(t,\se)-\txn(t,s)- x(t,\se)+x(t,s) \right| \, d\se\, ds \\
\leq \;& \overline{M} L \int_{I^2} |\xin(t,s)| \,\left| \xin(t,\se) - \xin(t,s) \right| \, d\se\, ds  \leq 2 \overline{M} L \|\xin(t)\|^2_{L^2(I)} 
\end{split}
\end{equation*}
We now look at the second term of \eqref{eq:dxi}. 
Since $x \in \mathcal{C}([0,T];L^\infty(I;\R^d))$ and $\ba$ is continuous, there exists a constant $\displaystyle{C_\phi:={\ess\sup}_{(s,\se,t)\in I^2\times [0,T]}|\ba(x(t,\se)-x(t,s))| }$ which is finite, and we have the following bound:
\begin{equation*}
\begin{split}
& \int_{I^2} \zn(t,\se) \xin(t,s) \, \ba(x(t,\se)-x(t,s)) \, d\se\, ds 
\leq \; C_\phi \; \int_{I^2} |\zn(t,\se) \xin(t,s)| d\se\, ds \\
\leq & \; C_\phi \; \|\xin(t)\|_{L^1(I)}\, \|\zn(t)\|_{L^1(I)} \; \leq \;  C_\phi \; \|\xin(t)\|_{L^2(I)}\, \|\zn(t)\|_{L^2(I)}.
\end{split}
\end{equation*}
Hence from \eqref{eq:dxi}, 
\begin{equation}\label{eq:dxi2}
\frac{1}{2}\frac{d}{dt}\|\xin(t)\|_{L^2(I)}^2 \, \leq \, 2 \overline{M} L \|\xin(t)\|^2_{L^2(I)} +  C_\phi \; \|\xin(t)\|_{L^2(I)}\, \|\zn(t)\|_{L^2(I)}.
\end{equation}
We now compute 
\begin{equation*}
\begin{split}
\frac{\partial\zn(t,s)}{\partial t} = & N \int_{\frac{1}{N}\floor{sN}}^{\frac{1}{N}(\floor{sN}+1)} \psi(\se,\txn(t),\tmn(t)) \,d\se - \psi(s,x(t),m(t))\\
= & N \int_{\frac{1}{N}\floor{sN}}^{\frac{1}{N}(\floor{sN}+1)} \left[ \psi(\se,\txn(t),\tmn(t)) - \psi(\se,x(t),m(t)) \right] \,d\se  \\
& + N \int_{\frac{1}{N}\floor{sN}}^{\frac{1}{N}(\floor{sN}+1)}  \psi(\se,x(t),m(t)  \,d\se  - \psi(s,x(t),m(t)).
\end{split}
\end{equation*}
Multiplying by $\zn(t,s)$ and integrating over $I$, we get: 
\begin{equation}\label{eq:dzn}
\begin{split}
\frac{1}{2}\int_I \frac{\partial\zn(t,s)^2}{\partial t} \, ds = & \int_I N \int_{\frac{1}{N}\floor{sN}}^{\frac{1}{N}(\floor{sN}+1)}  \left[ \psi(\se,\txn(t),\tmn(t)) - \psi(\se,x(t),m(t)) \right] d\se\,  \zn(t,s) \, ds  \\
& + \int_I\left[ N \int_{\frac{1}{N}\floor{sN}}^{\frac{1}{N}(\floor{sN}+1)}  \psi(\se,x(t),m(t))  \,d\se  - \psi(s,x(t),m(t)) \right] \zn(t,s)\, ds \\
\leq & \; \|h_N(t)\|_{L^2(I)} \|\zn(t)\|_{L^2(I)} \; + \; \|g_N(t)\|_{L^2(I)} \, \|\zn(t)\|_{L^2(I)} ,
\end{split}
\end{equation}
where the last inequality was obtained from the Cauchy-Schwarz inequality, and we denoted by  {$h_N$ and $g_N$} the functions 
$$
h_N:(t,s) \mapsto N \int_{\frac{1}{N}\floor{sN}}^{\frac{1}{N}(\floor{sN}+1)}  \left[ \psi(\se,\txn(t),\tmn(t)) - \psi(\se,x(t),m(t)) \right] d\se$$
and 
$$
g_N:(t,s) \mapsto N \int_{\frac{1}{N}\floor{sN}}^{\frac{1}{N}(\floor{sN}+1)}  \psi(\se,x(t),m(t))  \,d\se  - \psi(s,x(t),m(t)).
$$
We start with the first term:
\begin{equation*}
\begin{split}
\|h_N(t)\|_{L^2(I)}^2 & =  \int_I \left( N \int_{\frac{1}{N}\floor{sN}}^{\frac{1}{N}(\floor{sN}+1)}  \left[ \psi(\se,\txn(t),\tmn(t)) - \psi(\se,x(t),m(t)) \right] d\se\, \right)^2ds \\
& = \sum_{i=1}^N\, \int_{\frac{i-1}{N}}^{\frac{i}{N}} N^2 \left(  \int_{\frac{i-1}{N}}^{\frac{i}{N}}  \left[ \psi(\se,\txn(t),\tmn(t)) - \psi(\se,x(t),m(t)) \right] d\se\, \right)^2ds \\
& = \sum_{i=1}^N\,  N \left(  \int_{\frac{i-1}{N}}^{\frac{i}{N}}  \left[ \psi(\se,\txn(t),\tmn(t)) - \psi(\se,x(t),m(t)) \right]\cdot 1 \, d\se\, \right)^2 \\
& \leq \sum_{i=1}^N\,  N \left( \left[ \int_{\frac{i-1}{N}}^{\frac{i}{N}}  \left[ \psi(\se,\txn(t),\tmn(t)) - \psi(\se,x(t),m(t)) \right]^2 d\se\,\right]^{1/2} \left[ \int_{\frac{i-1}{N}}^{\frac{i}{N}} 1^2  d\se\,\right]^{1/2}\right)^2  \\
& \leq \sum_{i=1}^N\,   \int_{\frac{i-1}{N}}^{\frac{i}{N}} \left[ \psi(\se,\txn(t),\tmn(t)) - \psi(\se,x(t),m(t)) \right]^2 d\se  \\
& = \|  \psi(\cdot,\txn(t),\tmn(t)) - \psi(\cdot,x(t),m(t)) \|_{L^2(I)}^2,
\end{split}
\end{equation*}
using the Cauchy Schwarz inequality.
Then, from \eqref{eq:psilip2}, 
\begin{equation*}
\begin{split}
\|h_N(t)\|_{L^2(I)} & \leq \|  \psi(\cdot,\txn(t),\tmn(t)) - \psi(\cdot,x(t),\tmn(t)) \|_{L^2(I)} + \| \psi(\cdot,x(t),\tmn(t)) - \psi(\cdot,x(t),m(t)) \|_{L^2(I)} \\
& \leq L_\psi   ( \|  \txn(t)-x(t) \|_{L^2(I)} + \| \tmn(t)-m(t) \|_{L^2(I)}) = L_\psi   ( \|  \xin(t) \|_{L^2(I)} + \| \zn(t) \|_{L^2(I)}).
\end{split}
\end{equation*}
We now study the second term of \eqref{eq:dzn}.
According to Lebesgue's differentiation theorem, for almost every $s\in I$,
$$
\lim_{N\rightarrow+\infty} N \int_{\frac{1}{N}\floor{sN}}^{\frac{1}{N}(\floor{sN}+1)} \psi(\se,x(t),m(t))  \,d\se  = \psi(s,x(t),m(t)). 
$$
which implies 
\begin{equation}\label{eq:h2}
\lim_{N\rightarrow+\infty} \|g_N(t)\|_{L^2(I)} = 0.
\end{equation}
Summing up the contributions of $h_N$ and $g_N$, from \eqref{eq:dzn} we obtain: 
\begin{equation}\label{eq:dzn2}
\frac{1}{2}\frac{d}{dt}\|\zn(t)\|_{L^2(I)}^2 \, \leq \, L_\psi   \left( \|  \xin(t) \|_{L^2(I)} + \| \zn(t) \|_{L^2(I)}  \right) \|\zn(t)\|_{L^2(I)} +  \|g_N(t)\|_{L^2(I)} \|\zn(t)\|_{L^2(I)}.
\end{equation}
Now, for all $\epsilon >0$ and $t\in[0,T]$, let $\bzn(t):=\sqrt{\|\zn(t)\|_{L^2(I)}^2+\epsilon}$ and $\bxin(t):=\sqrt{\|\xin(t)\|_{L^2(I)}^2+\epsilon}$. 
From \eqref{eq:dxi2} and \eqref{eq:dzn2},
\begin{equation*}
\begin{cases}
\frac{1}{2}\frac{d}{dt}(\bxin(t)^2) \, \leq \, 2 L \overline{M}  \, \bxin(t)^2 +  C_\phi \; \bxin(t) \, \bzn(t)\\
\frac{1}{2}\frac{d}{dt}(\bzn(t)^2) \, \leq \, L_\psi   \left( \bxin(t) + \bzn(t) \right) \bzn(t) +  \|g_N(t)\|_{L^2(I)} \bzn(t),
\end{cases}
\end{equation*}
and since for all $t\in [0,T]$, $\bxin(t)>0$ and $\bzn(t)>0$, this implies:
\begin{equation*}
\begin{cases}
\frac{d}{dt}\bxin(t) \, \leq \, 2 L \overline{M}  \, \bxin(t) +  C_\phi  \, \bzn(t)\\
\frac{d}{dt}\bzn(t) \, \leq \, L_\psi   \left( \bxin(t) + \bzn(t) \right) +  \|g_N(t)\|_{L^2(I)} .
\end{cases}
\end{equation*}
Summing up, we get
$$
\frac{d}{dt}(\bxin(t)+\bzn(t)) \, \leq \, K (\bxin(t)+\bzn(t)) + \|g_N(t)\|_{L^2(I)} 
$$
where $K:=\max\{ 2 L \overline{M} ,\, C_\phi,\, L_\psi\}$. Now, from Gronwall's inequality, for all $t\in [0,T]$,
$$
\bxin(t)+\bzn(t) \, \leq \, (\bxin(0)+\bzn(0))e^{Kt} + \int_0^t  \|g_N(\tau)\|_{L^2(I)}e^{K(t-\tau)}d\tau.
$$
Since $\epsilon$ is arbitrary, we obtain:
$$
\sup_{t\in [0,T]} \left(\|\xin(t)\|_{L^2(I)}+\|\zn(t)\|_{L^2(I)} \right) \leq \left(\|\xin(0)\|_{L^2(I)}+\|\zn(0)\|_{L^2(I)}+\int_0^T  \|g_N(\tau)\|_{L^2(I)}d\tau \right) e^{KT}.
$$
As seen in Remark \ref{Rem:x0m0}, $\lim_{N\rightarrow +\infty} \xin(0,s) = \lim_{N\rightarrow +\infty} \txn(0,s)-x(0,s) = 0$ and 
$\lim_{N\rightarrow +\infty} \zn(0,s) = \lim_{N\rightarrow +\infty} \tmn(0,s)-m(0,s) = 0$ for almost all $s\in I$, which implies 
that $\lim_{N\rightarrow +\infty}\|\xin(0)\|_{L^2(I)}=0$ and $\lim_{N\rightarrow +\infty}\|\zn(0)\|_{L^2(I)}=0$.
From \eqref{eq:h2}, for all $t$, $\lim_{N\rightarrow +\infty}\|g_N(t)\|_{L^2(I)}=0$, so using the dominated convergence theorem for the last term, we can finally deduce the convergence result \eqref{eq:convxm}.
\end{proof}

\section{Relation between Graph Limit and Mean-field Limit}
\label{sec:mfl}

In Section \ref{sec:graphlim}, we have derived the \emph{Graph Limit} of the microscopic model \eqref{eq:syst-gen} when the number of agents $N$ goes to infinity, showing that the limit functions representing the opinions and weights satisfy a system of integro-differential equations. 

The aim of this section is to relate the \emph{Graph Limit} that we obtained with the \emph{Mean-Field Limit}, much more studied in the field of collective dynamics. 
However, the Mean-Field Limit can only be derived for a particular form of mass dynamics that satisfy an \emph{indistinguishability} property. We begin by shedding light on this concept. 

\subsection{Indistinguishability and mean-field limit}
\label{sec:indisting}

In the context of the classical Hegselmann-Krause model without weights \eqref{eq:HK}, the Mean-Field Limit process consists of representing the population by its density rather than by a collection of individual opinions. 
The limit density $\nu(t,x)$ represents the (normalized) quantity of agents with opinion $x$ at time $t$ and satisfies a non-local transport equation, where the transport vector field $V$ is defined by the convolution of $\nu$ with the interaction function $\phi$. The proof of the limit lies on the fact that the empirical measure 
\begin{equation*}
\nu^N(t,x) := \frac{1}{N} \sum_{i=1}^N \delta(x-x^{N}_i(t)).
\end{equation*}
satisfies the very same transport equation. It is crucial to notice that there is an irretrievable information loss in this formalism. Indeed the empirical measure keeps a count of the number of agents with opinion $x$ at time $t$, but loses track of the individual labeling of agents (i.e. the indices). 

In the case of our augmented system \eqref{eq:syst-gen} with time-evolving weights, we generalize the notion of empirical measure by defining 
\begin{equation}\label{eq:empmes}
\mu^N(t,x) := \frac{1}{N} \sum_{i=1}^N m^{N}_i(t) \delta(x-x^{N}_i(t)).
\end{equation}
We stress once again the information loss: the empirical measure only keeps track of the total weight of the group of agents with opinion $x$, but loses track of the individual labeling, the individual weights, and the number of agents at each point $x$.

More specifically, we draw attention to the fact that the empirical measure is invariant :
\begin{enumerate}[(i)]
\item by relabeling of the indices: for every $(x^N,m^N)\in (\R^d)^N\times\R^N$, for any $\sigma$ permutation of $\elts$, 
$$
\frac{1}{N} \sum_{i=1}^N m^{N}_i \delta(x-x^{N}_i)=\frac{1}{N} \sum_{i=1}^N m^{N}_{\sigma(i)} \delta(x-x^{N}_{\sigma(i)})
$$
\item by grouping of the agents: for every $(x^N,m^N)\in (\R^d)^N\times\R^N$, for every $I\subset\elts$, such that $x^N_i = x_I$ for all $i\in I$,
$$
\frac{1}{N} \sum_{i=1}^N m^{N}_i \delta(x-x^{N}_i) = \frac{1}{N} \left[ \left(\sum_{i\in I} m^{N}_{i}\right) \delta(x-x_{I}) + \sum_{i\in\elts\setminus I} m^{N}_i \delta(x-x^{N}_i) \right].
$$
\end{enumerate}
Figure \ref{fig:Indist} illustrates this invariance by comparing two microscopic systems corresponding to the same empirical measure.
This illustrates the fact that contrarily to the graph limit seen in the previous section, the mean-field limit process entails a non-reversible information loss. The empirical measure only retains the information concerning the total mass of agents at each point, and is incapable of differentiating between differently-labeled agents or between agents grouped at the same point. 

\begin{figure}[h!]
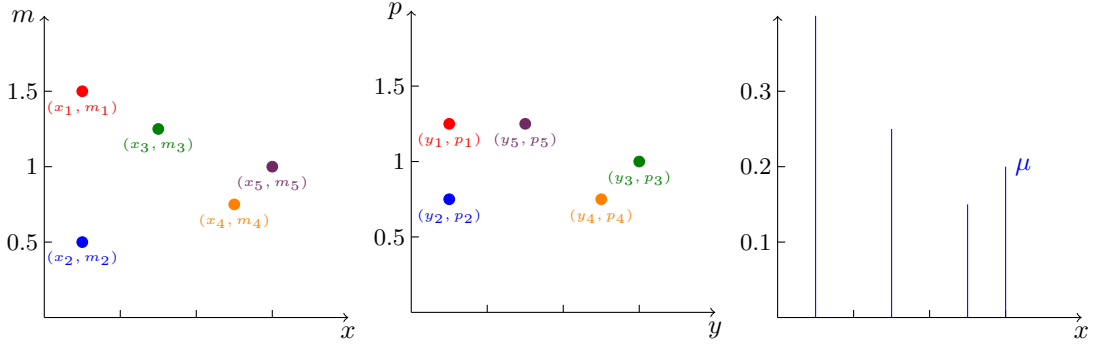

\begin{center}
\begin{tikzpicture}[scale=1]
\input{Micro_xm_1}
\end{tikzpicture}%
\begin{tikzpicture}[scale=1]
\input{Micro_xm_2}
\end{tikzpicture}%
\begin{tikzpicture}[scale=1]
\draw[->,] (0,0) -- (0,4);
\draw[->,] (0,0) -- (4,0);
\draw[] (4,0) node[below] {$x$};
\draw[blue] (3,2) node[right] {$\mu$};
\draw[-,] (1,0) -- (1,0.1);
\draw[-,] (2,0) -- (2,0.1);
\draw[-,] (3,0) -- (3,0.1);
\draw[-,] (0,1) -- (0.1,1) node[left] {\small $0.1\,$};
\draw[-,] (0,2) -- (0.1,2) node[left] {\small$0.2\,$};
\draw[-,] (0,3) -- (0.1,3) node[left] {\small$0.3\,$};

\draw[-,blue] (0.5,0) -- (0.5,4);
\draw[-,blue] (1.5,0) -- (1.5,2.5);
\draw[-,blue] (2.5,0) -- (2.5,1.5);
\draw[-,blue] (3,0) -- (3,2);
\end{tikzpicture}%
\end{center}
\caption{Schematic representation of two microscopic sets of agents $(x^5,m^5)\in\R^5\times \R^5$ and $(y^5,p^5)\in\R^5\times \R^5$ corresponding to the same empirical measure $\mu^5\in\mathcal{P}(\R)$. Left: Representation of $(x^5,m^5)$ with $x^5 = (0.5, 0.5, 1.5, 2.5, 3 )$ and $m^5 = (1.5,0.5,1.25,0.75,1)$. Center: Representation of $(y^5,p^5)$ with $y^5 = (0.5, 0.5, 3, 2.5,1.5 )$ and $p^5 = (1.25,0.75,1,0.75,1.25)$. Right: Symbolic representation of the empirical measure  $\mu^5 = \frac{1}{5}( 2\,\delta_{0.5} + 1.25 \,\delta_{1.5} + 0.75 \,\delta_{2.5}+ \,\delta_{3})$.}\label{fig:Indist}
\end{figure}

Hence, in order to study the mean-field limit, we will require System \eqref{eq:syst-gen} to satisfy the following \textit{indistinguishability} property:
\begin{definition}\label{Def:indist}
We say that system \eqref{eq:syst-gen} satisfies \emph{indistinguishability} if
for all $J\subset \elts$, 
for all initial conditions $(x^0, m^0)\in \R^{dN}\times\R^N$ and $(y^0, p^0)\in \R^{dN}\times\R^N$ satisfying 
\begin{equation}\label{eq:prop-indist_0}
\begin{cases}
x_i^0 = y_i^0 = x_j^0 = y_j^0 \qquad \text{ for all } (i,j)\in J^2 \\
x_i^0 = y_i^0 \qquad \text{ for all } i\in \elts \\
m_i^0 = p_i^0 \qquad \text{ for all } i\in J^c \\
\sum_{i\in I} m_i^0 = \sum_{i\in I} p_i^0,
\end{cases}
\end{equation}
the solutions $t\mapsto (x(t),m(t))$ and $t\mapsto (y(t),p(t))$ to system \eqref{eq:syst-gen} with respective initial conditions
$(x^0, m^0)$ and $(y^0, p^0)$ satisfy for all $t\geq 0$,
\begin{equation*}
\begin{cases}
x_i(t) = y_i(t) = x_j(t) = y_j(t) \qquad \text{ for all } (i,j)\in J^2 \\
x_i(t) = y_i(t) \qquad \text{ for all } i\in \elts \\
m_i(t) = p_i(t) \qquad \text{ for all } i\in J^c \\
\sum_{i\in I} m_i(t) = \sum_{i\in I} p_i(t).
\end{cases}
\end{equation*}
\end{definition}
In the above definition and in all that follows, $J^c$ denotes the complement of the set $J$ in $\elts$, i.e. $J^c=\elts\setminus J$. 

We begin by noticing that part of the required property is automatically satisfied by the general system \eqref{eq:syst-gen}, without having to specify further the mass dynamics. Namely, we prove that if two agents initially start at the same position, they stay with the same position at all time.

\begin{prop}\label{Prop:equal}
Let $(x^{0,N}_i)_{i\in\elts}\in\R^{dN}$ and $(m^{0,N}_i)_{i\in\elts}\in\R^{N}$. Let $\psi_i^{(N)}:(\R^d)\times\R\rightarrow\R$ and let $\phi\in\Lip(\R^d;\R)$ such that the system of ODE \eqref{eq:syst-gen} with initial condition $\left( x^{0,N},m^{0,N}\right)$ admits a unique solution $(x^N,m^N)$.
If $x_k^{0,N}=x_l^{0,N}$ for some $(k,l)\in \elts^2$, then it holds
$x_k^N(t) = x_l^N(t)$ for all $t\geq 0$.
\end{prop}
\begin{proof}
Let $(x^0_i)_{i\in\elts}\in\R^{dN}$ and $(m^0_i)_{i\in\elts}\in\R^{N}$
Without loss of generality, suppose that $x_1^0 = x_2^0$.
Now consider the slightly modified differential system 
\begin{equation}\label{eq:syst-mod12}
\begin{cases}
\displaystyle \frac{d}{dt} x_i= \frac{1}{N} \sum_{j=3}^N m_j a(\|x_i-x_j\|) (x_j-x_i), \qquad x_i(0) = x_i^0, \qquad \text{ for } i=1,2 \\
\displaystyle \frac{d}{dt} x_i = \frac{1}{N} \sumj m_j a(\|x_i-x_j\|) (x_j-x_i), \qquad x_i(0) = x_i^0, \qquad \text{ for all } i\in\{3,\cdots,N\} \\
\displaystyle \frac{d}{dt}  m_i = \psi_i^N(x,m), \qquad m_i(0) = m_i^0, \qquad \text{ for all } i\in\elts.
\end{cases}
\end{equation}
System \eqref{eq:syst-mod12} has a unique solution, that we denote by $(\tx,\tm)$.
Notice that $\tx_1$ and $\tx_2$ satisfy the same differential equation, so since $x_1^0=x_2^0$, we have $\tx_1(t) = \tx_2(t)$ for all $t\geq 0$.
Furthermore, one easily sees that $(\tx, \tm)$ is also solution to \eqref{eq:syst-gen}. By uniqueness, we conclude that the unique solution $(x,m)$ to \eqref{eq:syst-gen} satisfies $x_1(t) = x_2(t)$ for all $t\geq 0$.
\end{proof}

Thus, if System \eqref{eq:syst-gen} is well-posed, part of the requirements for indistinguishability stated in Def. \ref{Def:indist} are automatically met. However, one can easily show that for general weight dynamics $\psi^{(N)}$, System \eqref{eq:syst-gen} does not satisfy indistinguishability.

For this reason, from here onward, we will focus on a particular class of weight dynamics given by
\begin{equation}
\label{eq:psi_i_gen}
\psi_i^{(N)}(x,m) = m_i(t)\frac{1}{N^k} \sum_{j_1=1}^N \cdots \sum_{j_k=1}^N m_{j_1}(t)\cdots m_{j_k}(t) S(x_i(t), x_{j_1}(t), \cdots x_{j_k}(t)),
\end{equation} 
where $k\in\N$ and $S:(\R^d)^{k+1}\rightarrow \R$ satisfies the following assumptions: 
\begin{hyp}\label{hyp:S}
 $S\in C((\R^d)^{k+1}; \R)$ is globally bounded and Lipschitz. More specifically, there exist $\bS$, $L_S>0$ s.t. 
\begin{equation*}\label{eq:psiSkbound}
 \forall y\in(\R^d)^{k+1}, \quad |S(y)|\leq \bS.
\end{equation*}
and  
\begin{equation*}\label{eq:psiSklip}
 \forall y\in (\R^d)^{k+1}, \forall z\in (\R^d)^{k+1}, | S(y_0,\cdots,y_k) - S(z_0,\cdots,z_k) | \leq L_S \sum_{i=0}^k |y_i-z_i|.
\end{equation*}
Furthermore, we require that $S$ satisfy the following skew-symmetry property: there exists $(i,j)\in \{0,\cdots,k\}^2$ such that for all $y\in(\R^d)^{k+1}$, 
\begin{equation}\label{eq:condS}
  S(y_0,\cdots, y_i,\cdots,y_j,\cdots,y_k) =-S(y_0,\cdots, y_j,\cdots,y_i,\cdots,y_k) .
\end{equation}
\end{hyp}

 We show that with the weight dynamics given by \eqref{eq:psi_i_gen} and Hyp. \ref{hyp:S}, 
 System \eqref{eq:syst-gen} satisfies the indistinguishability property of Def. \ref{Def:indist}.

\begin{prop}
Let $S\in C((\R^d)^{k+1}; \R)$ satisfy Hyp. \ref{hyp:S}, and consider the collective dynamics system \eqref{eq:syst-gen}, with mass dynamics given by \eqref{eq:psi_i_gen}.
Then the system satisfies the indistinguishability property of Def. \ref{Def:indist}.
\end{prop}

\begin{proof}
For conciseness and clarity, we prove the statement for the case $k=1$, i.e. $S\in C(\R^2;\R)$ and
\begin{equation}\label{eq:Sbis}
\psi_i^{(N)}(x,m) = \frac{1}{N} m_i \sumj m_j S(x_i,x_j).
\end{equation}
The proof in the general case $k\in\N$ is essentially identical.

Let $J\subset \elts$, $J^c = \elts\setminus J$ and $N_c:=|I^c|\leq N$ denote the cardinal of $I^c$.
Let $(x^0, m^0)\in \R^{dN}\times\R^N$ and $(y^0, p^0)\in \R^{dN}\times\R^N$ satisfy \eqref{eq:prop-indist_0}.
Let $x_J^0:= x_i^0 = y_i^0$ for $i\in J$.

Let us start by considering a different system of dimension $d(N_c+1)+(N_c+1)$:
\begin{equation}\label{eq:syst-condensed}
\begin{cases}
\displaystyle \frac{d}{dt}  \tx_i = \frac{1}{N} \tm_J a(\|\tx_i-\tx_J\|) (\tx_J-\tx_i) + \frac{1}{N} \sum_{j\in J^c} \tm_j a(\|\tx_i-\tx_j\|) (\tx_j-\tx_i) , \qquad \text{ for all } i\in J^c \\
\displaystyle \frac{d}{dt}  \tx_J =  \frac{1}{N} \sum_{j\in J^c} \tm_j a(\|\tx_J-\tx_j\|) (\tx_j-\tx_J),  \\
\displaystyle \frac{d}{dt}  \tm_i = \frac{1}{N} \tm_J \tm_i S(\tx_i,\tx_J) + \frac{1}{N} \tm_i \sum_{j\in J^c} \tm_j S(\tx_i,\tx_j), \qquad \text{ for all } i\in J^c \\
\displaystyle \frac{d}{dt}  \tm_J = \frac{1}{N} \tm_J^2 S(\tx_J,\tx_J) + \frac{1}{N} \tm_J \sum_{j\in J^c} \tm_j S(\tx_J,\tx_j).
\end{cases}
\end{equation}
Given a set of initial conditions, the Cauchy problem associated with \eqref{eq:syst-condensed} has a unique solution.

Let us now consider the solutions $t\mapsto (x(t),m(t)$ and $t\mapsto(y(t), p(t))$ to \eqref{eq:syst-gen} with mass dynamics given by \eqref{eq:Sbis} with respective initial conditions $(x^0, m^0)$ and $(y^0, p^0)$.
First of all, from Prop. \ref{Prop:equal}, $x_i(t)=x_j(t)$ and $y_i(t) = y_j(t)$ for all $(i,j)\in J^2$ and all $t\geq 0$. 
Let $t\mapsto x_J(t):=x_i(t)$ and $t\mapsto y_J(t):=y_i(t)$ for $i\in J$. 
We can compute: 
\begin{equation*}
\begin{cases}
\displaystyle \frac{d}{dt}  x_i = \frac{1}{N} (\sum_{j\in J} m_j) a(\|x_i-x_J\|) (x_J-x_i) + \frac{1}{N} \sum_{j\in J^c} m_j a(\|x_i-x_j\|) (x_j-x_i)  \qquad \text{ for all } i\in J^c,\\
\displaystyle \frac{d}{dt}  x_J =  \frac{1}{N} \sum_{j\in J^c} m_j a(\|x_J-x_j\|) (x_j-x_J),  \\
\displaystyle \frac{d}{dt}  m_i = \frac{1}{N} (\sum_{j\in J} m_j) m_i S(x_i,x_J) + \frac{1}{N} m_i \sum_{j\in J^c} m_j S(x_i,x_j) \qquad \text{ for all } i\in J^c , \\
\displaystyle \frac{d}{dt}(\sum_{j\in J} m_j) = \frac{1}{N} (\sum_{j\in J} m_j)^2 S(x_J,x_J) + \frac{1}{N} (\sum_{j\in J} m_j) \sum_{j\in J^c} m_j S(x_J,x_j).
\end{cases}
\end{equation*}
This shows that $( (x_i)_{i\in J}, x_J, (m_i)_{i\in J}, (\sum_{j\in J} m_j))$ satisfies the differential system \eqref{eq:syst-condensed}. 
Similarly, we can show that $( (y_i)_{i\in J}, y_J, (p_i)_{i\in J}, (\sum_{j\in J} p_j))$ satisfies \eqref{eq:syst-condensed}.
Furthermore, from \eqref{eq:prop-indist_0}, 
$$( (x_i^0)_{i\in J}, x_J^0, (m_i^0)_{i\in J}, (\sum_{j\in J} m_j^0))=( (y_i^0)_{i\in J}, y_J^0, (p_i^0)_{i\in J}, (\sum_{j\in J} p_j^0)).$$
By uniqueness, for all $t\geq 0$ it holds:
$$( (x_i)_{i\in J}, x_J, (m_i)_{i\in J}, (\sum_{j\in J} m_j)) = ( (y_i)_{i\in J}, y_J, (p_i)_{i\in J}, (\sum_{j\in J} p_j)).$$
\end{proof}

\subsection{General context and results}\label{sec:MFpres}

Let $\PR$ denote the set of probability measures on $\R^d$. 
As exposed in Section \ref{sec:indisting}, the Mean-Field Limit process only makes sense for a subclass of mass dynamics that satisfy an indistinguishability property (Def. \ref{Def:indist}).
Hence, from here onward, we will focus on the indistinguishability-preserving mass dynamics given by \eqref{eq:psi_i_gen}.
In this framework, the convergence of the empirical measure $\mu^N$ to a limit measure $\mu$ was proven in \cite{PouradierDuteil21}, and can be stated more precisely as follows:

\begin{theorem}\label{Th:mfl}
Let $T>0$.
For each $N\in\N$, let $(x^{0,N}_i)_{i\in\elts}\in\R^{dN}$ and $(m^{0,N}_i)_{i\in\elts}\in\R^{N}$ satisfying \eqref{eq:sum_M}.
Let $S\in C((\R^d)^{k+1}; \R)$ satisfying Hyp. \ref{hyp:S}.
For all $t\in [0,T]$, let $t\mapsto(x^N(t), m^N(t))$ be the corresponding solution to \eqref{eq:syst-gen}-\eqref{eq:psi_i_gen} with initial data $(x^{0,N},m^{0,N})$, and let $\mu^N_t:= \frac{1}{N}\sum_{i=1}^N m_i^N(t) \delta_{x_i^N(t)}\in \PR$ be the corresponding empirical measures.
If there exists $\mu_0 \in \PR$ such that 
$$
\lim_{N\rightarrow\infty} \rho(\mu^N_0,\mu_0) = 0, 
$$
then for all $t\in [0,T]$, 
$$
\lim_{N\rightarrow\infty} \rho(\mu^N_t,\mu_t) = 0, 
$$
where $\mu_t$ is the solution to the transport equation with source 
\begin{equation}\label{eq:mfl}
\partial_t \mu_t(x) + \nabla\cdot ( V[\mu_t](x) \mu_t(x)) = h[\mu_t](x)
\end{equation} 
with the non-local vector-field given by
\[
\forall \mu\in\PR,\quad \forall x\in\R^d,  \quad V[\mu](x) = \int_{\R^d} \phi(y-x) d\mu(y),
\]
the non-local source term given by
\[
\forall \mu\in \PR, \quad \forall x\in\R^d, \quad h[\mu](x) = \left(\int_{(\R^d)^k} S(x,y_1,\cdots,y_k) d\mu(y_1)\cdots d\mu(y_k)\right) \mu(x).
\]
and with initial condition $\mu_{t=0} = \mu_0$.
\end{theorem}
\noindent In Theorem \ref{Th:mfl}, $\rho$ denotes the Bounded Lipschitz distance, also known as the generalized Wasserstein distance $W^{1,1}_1$ \cite{PiccoliRossi14,PouradierDuteil21}.

The aim of this section is to explore the link between the Graph Limit equation \eqref{eq:GraphLimit-gen} and the Mean-Field Limit equation \eqref{eq:mfl}. 
In the spirit of \cite{BiccariKoZuazua19}, we will actually show that the Mean-Field Limit is subordinated to the Graph Limit.
Moreover, this link can be used to revisit the proof of the mean-field limit of the discrete system \eqref{eq:syst-gen}. It is important to note however that the result is weaker than Theorem \ref{Th:mfl} and its interest lies mainly in its ability to link the two approaches. We will prove the following result : 
\begin{theorem}
\label{Th:mfl2}
Let $T>0$.
Let $x_0\in L^\infty(I;\R^d)$ and $m_0\in L^\infty(I;\Rp)$ satisfying \eqref{eq:integral_egal_a_1}. 
Let $\phi$ satisfy Hyp.~\ref{hyp:phi}.
Let $(\txn,\tmn)$ denote the solution to \eqref{eq:syst-contN} with mass dynamics given by \eqref{eq:psi_i_gen}, with $S$ satisfying Hyp.~\ref{hyp:S} and with initial conditions  
\[\displaystyle x_N(0,s) = \sum_{i=1}^N x_i^{0,N} \mathbf{1}_{[\frac{i-1}{N}, \frac{i}{N})}(s) \quad \text{ and }\quad \displaystyle m_N(0,s) = \sum_{i=1}^N m_i^{0,N} \mathbf{1}_{[\frac{i-1}{N}, \frac{i}{N})}(s),
\]
where $x^{0,N}=\Pdn(x_0)$ and $m^{0,N}=\Pdn(m_0)$ are defined by \eqref{eq:ICx}.\\
Then there exist $x\in \mathcal{C}([0,T];L^\infty(I;\R^d))$ and $m\in \mathcal{C}([0,T];L^\infty(I;\R))$ such that
\[
\displaystyle \|x-\txn\|_{\mathcal{C}([0,T];L^2(I,\R^d))} \xrightarrow[N\rightarrow+\infty]{} 0 
\quad \text{ and } \quad 
\|m-\tmn\|_{\mathcal{C}([0,T];L^2(I,\R))} \xrightarrow[N\rightarrow+\infty]{} 0.
\]
Moreover, $(x,m)$ are solutions to the integro-differential system \eqref{eq:GraphLimit-gen} with $\psi:=\psi_{S,k}$ defined by 
\begin{equation*}
\psi_{S,k}(s,x,m) = m(s) \int_{I^k} m(s_1) \cdots m(s_k) \, S(x(s), x(s_1), \cdots, x(s_k) ) \; ds_1\cdots ds_k.
\end{equation*}
In addition, let $\tilde{\mu}\in C([0,T];\PR)$ be defined by
\begin{equation*}
\tilde{\mu}_t(x) := \int_I m(t,\se) \delta(x-x(t,\se)) d\se.
\end{equation*}
Then, the empirical measure $\mu^N$ defined in \eqref{eq:empmes} converges weakly to $\tilde{\mu}$, and $\tilde{\mu}$ is a solution to the transport equation with source term \eqref{eq:mfl}.
\end{theorem}

Theorem \eqref{Th:mfl2} contains many different results.
We show the well-posedness of the system of integro-differential equations given by \eqref{eq:GraphLimit-gen} with $\psi:=\psi_{S,k}$ in Section \ref{sec:WellPos-Indist}. The well-posedness of the microscopic system \eqref{eq:syst-gen}-\eqref{eq:psi_i_gen} will be stated in Section \ref{sec:WellPos-Indist-micro}.
Section \ref{sec:mflsubord} is devoted to proving that $\tilde{\mu}$ is a weak solution to \eqref{eq:mfl}.

\subsection{Well-posedness of the Graph Limit equation}\label{sec:WellPos-Indist}

From here onward, as mentioned in Section \ref{sec:indisting}, we will consider a particular class of mass dynamics which 
preserves mass as well as \emph{indistinguishability}.

Let $\psi_{S,k}:I\times L^\infty(I;\R^d)\times L^\infty(I;\R)\rightarrow \R$ be defined by
\begin{equation}\label{eq:psiSk}
\psi_{S,k}(s,x,m) = m(s) \int_{I^k} m(s_1) \cdots m(s_k) \, S(x(s), x(s_1), \cdots, x(s_k) ) \; ds_1\cdots ds_k
\end{equation}
where 
$k\in\N$ and as previously, $S\in C((\R^d)^{k+1}; \R)$ satisfies Hyp. \ref{hyp:S}.
\begin{rem}
The condition \eqref{eq:condS} of Hyp. \ref{hyp:S} implies that 
\begin{equation}\label{eq:psiSkint}
\int_{I} \psi_{S,k}(s,x,m) ds = 0.
\end{equation}
Actually, this condition \eqref{eq:psiSkint} is sufficient to prove all the subsequent results, but for simplicity purposes, we will keep referring to the more particular (and more tangible) condition \eqref{eq:condS}.
\end{rem}

We notice that at first glance, the mass dynamics $\psi_{S,k}$ given by \eqref{eq:psiSk} do not satisfy 
Hyp. \ref{hyp:psi}, which was necessary in order to prove the existence and uniqueness of the solution to equation \eqref{eq:GraphLimit-gen}. 
The aim of this section is to show that nevertheless, the system of coupled integro-differential equations \eqref{eq:GraphLimit-gen} with $\psi= \psi_{S,k}$ is well-posed, as long as Hyp. \ref{hyp:S} is satisfied.
The proof of existence and uniqueness will rely on the following a priori observations: 

\begin{prop}\label{Prop:psiSkprop}
Let $(x_0,m_0)\in L^\infty(I,\R^d\times\Rp)$. Consider a solution $(x,m)$ on an interval $[0,T]$ to the integro-differential system \eqref{eq:GraphLimit-gen} with initial condition $(x_0,m_0)$ and weight dynamics $\psi_{S,k}$ given by \eqref{eq:psiSk} and Hyp.~\ref{hyp:S}. 
Then, the following properties hold: 
\begin{enumerate}[i)]
\item For almost every $s\in I$ and all $t\in [0,T]$, $m(t,s)>0$
\item For all $t\in [0,T]$, $\int_I m(t,s) ds = M_0 := \int_I m_0(s) ds$ 
\item For almost every $s\in I$ and all $t\in [0,T]$, $m(t,s)\leq m_0(s) \exp(M_0^k \bS t)$
\end{enumerate}
\end{prop}

\begin{rem}
With the assumption \eqref{eq:integral_egal_a_1}, Properties $(ii)$ and $(iii)$ simplify to
\[\int_I m(\cdot,s) ds \equiv 1 \qquad \text{ and } \qquad m(t,s)\leq m_0(s) \exp(\bS t).
\]
\end{rem}

\begin{proof}
The second property is an immediate consequence of the antisymmetry property \eqref{eq:condS}.\\
We now focus on the first point. Suppose that there exists $t_- \in [0,T]$  and a non-negligeable set $\overline{\mathcal{N}}$ in I, such that for $s \in \overline{\mathcal{N}}$, we have $m(s,t_-)\leq 0$. Let
$$t^* := \inf\{t\in [0,T], \, \text{there exists a non-negligeable set } \overline{\mathcal{N}}_k \text{ such that for } s \in \overline{\mathcal{N}}_k, ~  m(t,s)\leq 0\}.$$ Since $m_0(s)>0$ for almost every $s\in I$, there exists $s^* \in \overline{\mathcal{N}}_k$ such that $m_0(s^*) >0$  and by continuity, $m(t^*,s^*)=0$. Moreover, by definition of $t^*$, for $t<t^*$ and  for almost every $s\in I$, we have $m(t,s)>0$. 
Using the global bound on $S$, we then compute, for all $t\leq t^*$ : 
\begin{equation*}
\begin{split}
\partial_t m(t,s^*) = & m(t,s^*) \int_{I^k} m(t,s_1) \cdots m(t,s_k) \, S(x(t,s), x(t,s_1), \cdots, x(t,s_k) ) \; ds_1\cdots ds_k \\
\geq& - \bS \, m(t,s^*) \int_{I^k} m(t,s_1) \cdots m(t,s_k) \, ds_1\cdots ds_k =  - \bS \, M_0^k \, m(t,s^*)
\end{split}
\end{equation*}
where we used the second property.
Thus, from Gronwall's lemma, we obtain that for all $t\leq t^*$, $m(t,s^*)\geq m_0(s^*) e^{-\bS \, M_0^k\, t} > 0 $, which contradicts $m(t^*,s^*)=0$. We deduce that $m(t,s)>0$ for almost every $s\in I$ and all $t\in [0,T]$. \\
The third point can be proven easily by using the positivity of the weights and the boundedness of $S$ by $\bS$.
\end{proof}

We now prove that the results of Lemma \ref{Lemma:WellPos-decoupled} also hold with $\psi=\psi_{S,k}$ satisfying \eqref{eq:psiSk} and Hyp. \ref{hyp:S}.

\begin{lemma}\label{Lemma:WellPos-decoupled-psiSk}
Let $\tx \in \CLx$ and $\tm \in \CLm$. Let $x_0\in L^\infty(I;\R^d)$ and $m_0\in L^\infty(I;\Rp)$.
Let $\ba$ satisfy Hyp. \ref{hyp:phi} and let $\psi_{S,k}$ satisfy \eqref{eq:psiSk} and Hyp. \ref{hyp:S}. \\
Then for any $T>0$, there exists a unique solution $(x,m)\in\CLp$ to the decoupled integro-differential equations 
\begin{equation*}
\left\{\begin{array}{l}
\displaystyle \partial_t x(t,s) = \int_I \tm(t,s_*)\ba(x(t,s) - x(t,\se)) d\se; \qquad x(\cdot,0)=x_0 \\
\partial_t m(t,s) = \psi_{S,k}(s,\tx(t,\cdot),m(t,\cdot)); \qquad m(\cdot,0)=m_0.
\end{array}\right.
\label{eq:GraphLimit-decoupled-psiSk}
\end{equation*}
\end{lemma}

\begin{proof}
Since the two equations are decoupled, the proof of existence and uniqueness of the solution $x\in\CLx$ to the first equation was already done in Lemma \ref{Lemma:WellPos-decoupled}. We focus on the well-posedness of the second equation.
Let $M_0:=\int_I m_0(s) ds$.
Let $\Mm'$ be a metric subspace of $\mathcal{C}([0,T];L^1(I;\Rp))$ consisting of functions $m$ satisfying $m(\cdot,0) = m_0$ and $\int m(t,s) ds = M_0$ for all $t\in [0,T]$.
Let $\Km'$ be the operator defined by:
$$
\begin{array}{l l l }
\Km': & \Mm' & \rightarrow \Mm' \\
 & m &\displaystyle   \mapsto (\Km' m):(t,s)\mapsto m_0(s) + \int_0^t \psi(s,\tx(\tau,\cdot),m(\tau,\cdot)) d\tau.
\end{array}
$$
From Proposition \ref{Prop:psiSkprop}, it is clear that $\Km'$ maps $\Mm'$ onto $\Mm'$.
Let $\tT>0$. We will show that $\Km'$ is contracting for the norm $\|\cdot\|_{\Mm'} := \sup_{[0,\tT]} \|\cdot\|_{L^1(I)}$.
Let $(m_1,m_2)\in \Mm'^2$.
Then for all $t\leq \tT$,
\begin{equation*}
\begin{split}
\int_I & | \Km' m_1 - \Km' m_2 | (t,s) ds  = \int_I \left |  \int_0^t \psi_{S,k}(s,\tx(\tau,\cdot),m_1(\tau,\cdot))- \psi_{S,k}(s,\tx(\tau,\cdot),m_2(\tau,\cdot)) d\tau \right |  ds \\
& \leq \int_I \int_0^t \int_{I^k} \Big | [m_1(s) m_1(s_1) \cdots m_1(s_k) - m_2(s) m_2(s_1) \cdots m_2(s_k)] \, S(\tx(s), \tx(s_1), \cdots, \tx(s_k) ) \Big | \; ds_1\cdots ds_k \, dt \, ds\\
& \leq \bS \int_0^t \int_{I^{k+1}} \left | m_1(s) m_1(s_1) \cdots m_1(s_k) - m_2(s) m_2(s_1) \cdots m_2(s_k) \right | \;ds \, ds_1\cdots ds_k \, dt \\
& \leq \bS \int_0^t \Big [ \int_{I^{k+1}} | m_1(s) - m_2(s) | m_1(s_1) \cdots m_1(s_k) \;ds \, ds_1\cdots ds_k  \\
& \qquad +  \int_{I^{k+1}} m_2(s) | m_1(s_1) - m_2(s_1) | m_1(s_2) \cdots m_1(s_k) \;ds \, ds_1\cdots ds_k   \\
& \qquad + \cdots +  \int_{I^{k+1}} m_2(s) m_2(s_1) \cdots m_2(s_{k-1}) | m_1(s_k) - m_2(s_k) | \;ds \, ds_1\cdots ds_k \, \Big ] dt \\
& \leq \bS (k+1) \int_0^t \int_I | m_1(s) - m_2(s) | ds M_0^k \, dt,
\end{split}
\end{equation*}
where, from the second line onward we omitted the time dependence for compactness of notation.
We obtain: 
\begin{equation*}
 \| \Km' m_1 - \Km' m_2 \|_{\Mm'}  \leq \tT \bS (k+1) M_0^k   \|  m_1 - m_2 \|_{\Mm'}.
\end{equation*}
Thus, if $\tT\leq (2\bS (k+1) M_0^k)^{-1}$, by the Banach contraction mapping principle, there exists a unique solution $m\in \mathcal{C}([0,\tT];L^1(I;\Rp))$.
We then take $m(\tT,\cdot)$ as the initial condition, and the local solution can be extended to $[0, 2\tT]$, and by repeating the same argument, to $[0,T]$.
We thus showed that there exists a unique solution $m\in \mathcal{C}([0,T];L^1(I;\Rp))$. 
Furthermore, the third property of Prop. \ref{Prop:psiSkprop} implies that $m\in \mathcal{C}([0,T];L^\infty(I;\Rp))$.
Lastly, since the integrand is continuous with respect to $\tau$, $m(\cdot,s)$ is continuously differentiable for almost all $s\in I$, which proves that ${m\in \mathcal{C}^1([0,T];L^\infty(I;\Rp))}$.
\end{proof}

The proof of Theorem \ref{Th:GL-wellpos} relied on the fact that $\psi$ satisfies \eqref{eq:psilip2}. Although $\psi_{S,k}$ does not satisfy \eqref{eq:psilip2}, we notice that a similar property does hold as long as $m$ belongs to $L^\infty$. 

\begin{lemma}\label{Lemma:psiSklip2}
Let $\psi_{S,k}$ satisfy \eqref{eq:psiSk} and Hyp. \ref{hyp:S}. Let $(m_1, m_2) \in L^\infty(I)^2$ with $max(\|m_2\|_{L^\infty(I)},\|m_1\|_{L^\infty(I)}) \leq M_1$.
Then for all $(x_1,x_2)\in L^\infty(I)^2$, it holds: 
\begin{equation}\label{eq:psiSklip2}
\begin{cases}
 \|\psi_{S,k}(\cdot,x_1,m_1)-\psi_{S,k}(\cdot,x_2,m_1)\|_{L^2(I)} \, \leq \, (k+1) L_S M_1^{k+1} \|x_1-x_2\|_{L^2(I)}\\
 \|\psi_{S,k}(\cdot,x_1,m_1)-\psi_{S,k}(\cdot,x_1,m_2)\|_{L^2(I)} \, \leq \, \sqrt{k+1} \bS M_1^k \|m_1-m_2\|_{L^2(I)}.
\end{cases}
\end{equation}
\end{lemma}

\begin{proof}
We start by computing
\begin{equation*}
\begin{split}
\|\psi_{S,k} & (\cdot,x_1,m_1)-\psi_{S,k}(\cdot,x_2,m_1)\|_{L^2(I)}^2 \\
& = \int_I \bigg | \int_{I^k} m_1(s_0)\cdots m_1(s_k) (S(x_1(s_0),\cdots,x_1(s_k))-S(x_2(s_0),\cdots,x_2(s_k))) ds_1 \cdots ds_k \bigg |^2 ds_0 \\
& \leq \int_{I^{k+1}} m_1(s_0)^2\cdots m_1(s_k)^2 \Big | S(x_1(s_0),\cdots,x_1(s_k))-S(x_2(s_0),\cdots,x_2(s_k))\Big |^2 ds_0 \cdots ds_k  \\
& \leq \int_{I^{k+1}} m_1(s_0)^2\cdots m_1(s_k)^2 L_S^2 \Big (\sum_{i=0}^k \|x_1(s_i)-x_2(s_i)\|\Big )^2 ds_0 \cdots ds_k  \\
& \leq M_1^{2(k+1)} L_S^2 \int_{I^{k+1}} (k+1) \sum_{i=0}^k \|x_1(s_i)-x_2(s_i)\|^2 ds_0 \cdots ds_k  \\
& \leq (k+1)^2 L_S^2 M_1^{2(k+1)} \|x_1-x_2\|_{L^2(I)}^2
\end{split}
\end{equation*}
from which we get the first inequality of \eqref{eq:psiSklip2}.
For the second inequality, we compute:
\begin{equation*}
\begin{split}
\|\psi_{S,k} & (\cdot,x_1,m_1)-\psi_{S,k}(\cdot,x_1,m_2)\|_{L^2(I)}^2 \\
& = \int_I \bigg | \int_{I^k} \big ( m_1(s_0)\cdots m_1(s_k) - m_2(s_0)\cdots m_2(s_k) \big ) S(x_1(s_0),\cdots,x_1(s_k)) ds_1 \cdots ds_k \bigg |^2 ds_0 \\
& \leq \bS^2  \int_{I^{k+1}} \big | m_1(s_0)\cdots m_1(s_k) - m_2(s_0)\cdots m_2(s_k) \big |^2  ds_0 \cdots ds_k  \\
& \leq \bS^2  \int_{I^{k+1}} \sum_{i=0}^k  \Big( \prod_{j=0}^{i-1} m_2(s_j)^2\Big) | m_1(s_i)-m_2(s_i)|^2 \Big( \prod_{j=i+1}^{k} m_1(s_j)^2\Big)  ds_0 \cdots ds_k  \\
& \leq \bS^2 M_1^{2k} \int_{I^{k+1}} \sum_{i=0}^k | m_1(s_i)-m_2(s_i)|^2 ds_0 \cdots ds_k  = (k+1) \bS^2 M_1^{2k}  \|m_1-m_2\|_{L^2(I)}^2.
\end{split}
\end{equation*}
\end{proof}

We are now fully equipped to prove the well-posedness of the coupled system, with $\psi=\psi_{S,k}$.  

\begin{theorem}\label{Th:WellPos-psiSk}
Let $x_0\in L^\infty(I;\R^d)$ and $m_0\in L^\infty(I;\Rp)$.
Let $\ba$ satisfy Hyp. \ref{hyp:phi} and let $\psi_{S,k}$ satisfy \eqref{eq:psiSk} and Hyp. \ref{hyp:S}. \\
Then for any $T>0$, there exists a unique solution $(x,m)\in\CLp$ to the integro-differential system 
\begin{equation*}
\left\{\begin{array}{l}
\displaystyle \partial_t x(t,s) = \int_I m(t,s_*)\ba(x(t,s) - x(t,\se)) d\se; \qquad x(\cdot,0)=x_0 \\
\partial_t m(t,s) = \psi_{S,k}(s,x(t,\cdot),m(t,\cdot)); \qquad m(\cdot,0)=m_0.
\end{array}\right.
\end{equation*}
\end{theorem}

\begin{proof}
The proof is almost identical to the one of Theorem \ref{Th:GL-wellpos}.
The only difference lies in the inequality \eqref{eq:ineqmn}.
We notice that from the third point of Proposition \ref{Prop:psiSkprop}, we can apply Lemma \ref{Lemma:psiSklip2} with $m_1:=m^n$ and $m_2:= m^{n+1}$,
 and $\|m^n\|_{L^\infty(I)} \leq M_T:=\| m_0\|_{L^\infty(I)} \exp(M_0^k \bS T)$.
We can thus replace \eqref{eq:ineqmn} by:
\begin{equation*}
\begin{split}
\|m^{n+1}-& m^n\|_{L^2(I)}^2(t) \\
\leq &  2 t \int_0^t \int_I |\psi_{S,k}(s,x^n(\tau,\cdot),m^{n+1}(\tau,\cdot)) - \psi_{S,k}(s,x^n(\tau,\cdot),m^{n}(\tau,\cdot)) |^2 ds \, d\tau \\
& + 2t \int_0^t \int_I | \psi_{S,k}(s,x^n(\tau,\cdot),m^{n}(\tau,\cdot)) - \psi_{S,k}(s,x^{n-1}(\tau,\cdot),m^{n}(\tau,\cdot)) | ^2 ds \, d\tau  \\
\leq &  2 t (k+1) \bS^2 M_1^{2k} \int_0^t  \|m^{n+1}-m^{n}\|_{L^2(I)}^2(\tau) d\tau +  2 t  (k+1)^2 L_S^2 M_1^{2(k+1)} \int_0^t \| x^n - x^{n-1}\|_{L^2(I)}^2(\tau) d\tau .
\end{split}
\end{equation*}
The rest of the proof of existence is identical, replacing the previous definition of $A_T$ by 
$$A_T = \max(8 T  \Lphi^2 \, X_T^2,\, 8 T  \Lphi^2 \,  M_T^2,\,  2 T (k+1) \bS^2 M_T^{2k}, 2T (k+1)^2 L_S^2 M_T^{2(k+1)} ).$$ 
The proof of uniqueness can be adapted similarly.
\end{proof}

\subsection{Well-posedness of the microscopic system}\label{sec:WellPos-Indist-micro}

The existence and uniqueness for the microscopic system \eqref{eq:syst-gen} can be shown in the case of functions $\psi_{S,k}$ satisfying Hyp. \ref{hyp:S} instead of Hyp. \ref{hyp:psi}, in the same way that we adapted the proof of well-posedness for the integro-differential equations in Section \ref{sec:WellPos-Indist}. For this reason, we do not provide the proof, which would be redundant, and merely state the result. 

\begin{theorem}
Let $(x^{0,N},m^{0,N}) \in \R^{dN} \times \R^N$.
Let $\ba$ satisfy Hyp. \ref{hyp:phi}, let $\psi_{S,k}$  satisfy \eqref{eq:psiSk} and Hyp. \ref{hyp:S}, and $\psi_i^{(N)}$ be defined  by 
\begin{equation*}
\forall i\in\elts,\qquad \psi_i^{(N)}(x^{N}(t),m^{N}(t)) = N \int_{\frac{i-1}{N}}^{\frac{i}{N}} \psi_{S,k}(s,x_N(t,\cdot),m_N(t,\cdot)) ds. 
\end{equation*}\\
Then for any $T>0$, there exists a unique solution $(x^{N},m^{N}) \in \mathcal{C}^1([0,T];\R^{dN} \times \R^N)$ to the discrete system \eqref{eq:syst-gen} with initial condition $(x^{0,N},m^{0,N})$. 
Furthermore, the solution to the system has the following properties :
\begin{enumerate}
\item[(i)] $\forall i \in \{1, \dots, N \}$, $\forall t \in [0,T]$, $m_i^{(N)}(t) >0$,
\item[(ii)]$\forall t \in [0,T]$, $\displaystyle{\sum_{i=1}^N m_i^{(N)}(t) = N}$,
\item[(iii)] there exist constants $\overline{X}$ and  $\overline{M}$ such that for all $t \in [0,T]$, for all $i \in \{1, \dots, N \}$,
\begin{equation*}
\| x_i^{(N)}(t)\| \leq  \overline{X} 
\quad \text{ and } \quad
| m_i^{(N)}(t)| \leq  \overline{M}.
\end{equation*}
\end{enumerate}
\end{theorem}

\begin{rem}
The theorem can also be stated for functions $\psi_i^{(N)}$ given by \eqref{eq:psi_i_gen} satisfying Hyp. \ref{hyp:S}.
\end{rem}

\subsection{Subordination of the mean-field equation to the graph limit equation}\label{sec:mflsubord}

The first part of Theorem \ref{Th:mfl2} consists of deriving the Graph Limit, as in Theorem \ref{Th:GraphLimit}, but for mass dynamics  \eqref{eq:psiSk} satisfying Hyp. \ref{hyp:S} instead of Hyp. \ref{hyp:psi}. 
Nevertheless, as for the proof of the existence of a solution to the graph limit equation, 
the convergence proof can be adapted in a straightforward way, by using the assumptions on $\psi_{S,k}$ given by Hyp. \ref{hyp:S}.
Thus, we can show that the solution to the discrete system \eqref{eq:syst-gen}-\eqref{eq:psi_i_gen} converges to the functions $(x,m)\in \mathcal{C}([0,T]; L^2(I;\R^d))\times \mathcal{C}([0,T]; L^2(I;\R))$, solutions to the integro-differential system \eqref{eq:GraphLimit-gen} with $\psi:I \times L^2(I;\R^d)\times L^2(I;\R)$ defined in \eqref{eq:psiSk}, as stated in the following:

\begin{prop}
\label{prop:conv}
Let $x_0\in L^\infty(I;\R^d)$ and $m_0\in L^\infty(I;\Rp)$ satisfying \eqref{eq:integral_egal_a_1}.
Let $\phi$ satisfying Hyp. \ref{hyp:phi} and $\psi:=\psi_{S,k}$ satisfying \eqref{eq:psiSk} and Hyp. \ref{hyp:S}. 
Then the solution $(\txn,\tmn)$ to \eqref{eq:syst-contN}
with initial conditions 
$x_N(0,s)=\Pcn(\Pdn(x_0))$ and $m_N(0,s)=\Pcn(\Pdn(m_0))$ defined by \eqref{eq:ICx}-\eqref{eq:xN}
converges when $N$ tends to infinity in the $\mathcal{C}([0,T];L^2(I))$ topology, i.e. there exists $(x,m)\in \mathcal{C}([0,T];L^2(I,\R^d))\times \mathcal{C}([0,T];L^2(I,\R))$ such that 
\[
\displaystyle \|x-\txn\|_{\mathcal{C}([0,T];L^2(I,\R^d))} \xrightarrow[N\rightarrow+\infty]{} 0 \quad \text{ and } \quad
\|m-\tmn\|_{\mathcal{C}([0,T];L^2(I,\R))} \xrightarrow[N\rightarrow+\infty]{} 0.
\]
Furthermore, the limit functions $x$ and $m$ are solutions to the integro-differential system \eqref{eq:GraphLimit-gen} with $\psi=\psi_{S,k}$, 
supplemented by the initial conditions $x(0,\cdot) = x_0$ and $m(0,\cdot) = m_0$.
\end{prop}

 Let us now study the link between the non-local diffusive model coming from the graph limit equation~\eqref{eq:GraphLimit-gen} and the non-local transport equation with source \eqref{eq:mfl} obtained by the mean-field approach. This will give us an alternative proof of the mean-field limit. 

\begin{prop}
\label{Prop:link_mfl-gl}
Let $(x,m)\in  \mathcal{C}([0,T]; L^2(I;\R^d))\times \mathcal{C}([0,T]; L^2(I;\R))$such that 
\begin{equation}
\label{eq:GL-TE}
\begin{cases}
\partial_t x(t,s) = \displaystyle \int_I m(t,\se)   \phi(x(t,\se)-x(t,s)) d\se \\
\partial_t m(t,s) = \displaystyle m(t,s) \int_{I^k} m(t,s_1) \cdots m(t,s_k) \, S(x(t,s), x(t,s_1), \cdots, x(t,s_k) ) \; ds_1\cdots ds_k.
\end{cases}
\end{equation}
Let $\tilde{\mu}\in\PR$ be defined by 
\begin{equation}\label{eq:def_mu}
\tilde{\mu}_t(x) := \int_I m(t,\se) \delta(x-x(t,\se)) d\se.
\end{equation}
Then $\tilde{\mu}$ satisfies the transport equation with source \eqref{eq:mfl}. 
\end{prop}
\begin{proof}
Given a test function $\varphi=\varphi(x)$, we consider the term
$$(\tilde{\mu}_t,\varphi) := \int_{\R^d} \varphi(x) d \tilde{\mu}_t(x).$$ 
Let us study the quantity $\displaystyle{\frac{d}{dt} (\tilde{\mu}_t,\varphi)}.$ We start by noticing that 
$$
(\tilde{\mu}_t,\varphi) = \displaystyle \int_{\R^d} \varphi(x) d \tilde{\mu}_t(x)=\displaystyle  \int_I m(t,s) \varphi(x(t,s)) ds
$$
by definition of $\tilde{\mu}_t$.  Therefore, it holds
$$
\displaystyle  \frac{d}{dt} (\tilde{\mu}_t,\varphi) = \displaystyle  \int_I \partial_tm(t,s) \varphi(x(t,s)) ds +  \int_I m(t,s) <\partial_t x(t,s), \nabla_x \varphi(x(t,s))> ds
$$
where $<\cdot,\cdot>$ denotes the inner product in $\R^d$. Let us deal with the first term. Using \eqref{eq:GL-TE}, we have 
\[
\begin{split}
& \displaystyle \int_I \partial_t m(t,s) \varphi(x(t,s)) ds\\ 
=& \displaystyle \int_{I^{k+1}}  m(t,s)  m(t,s_1) \cdots m(t,s_k) \, S(x(t,s), x(t,s_1), \cdots, x(t,s_k) )\varphi(x(t,s))  \; ds\, ds_1\cdots ds_k\\
=&  \displaystyle  \int_{\R^d}    \varphi(x) dh[\tilde{\mu}_t](x).
\end{split}
\]
Let us now deal with the second term. We start by rewriting the right-hand side of the first equation of \eqref{eq:GL-TE}. 
 $$
 \displaystyle \int_I m(t,\se)  \phi(x(t,\se)-x(t,s)) d\se 
= \int_{\R^d}   \phi(x_*-x(t,s))  d\tilde{\mu}_t(x_*)
 = V[\tilde{\mu}_t](x(t,s)).
$$
 Thus, we obtain
\[
\begin{split}
 \displaystyle \int_I m(t,s) <\partial_t x(t,s), \nabla_x \varphi(x(t,s))> ds  = &  \displaystyle \int_I m(t,s) <V[\tilde{\mu}_t](x(t,s)),  \nabla_x \varphi(x(t,s))> ds \\
   \displaystyle  = &  \displaystyle \int_{\R^d}   <V[\tilde{\mu}_t](x),  \nabla_x \varphi(x)> d \tilde{\mu}_t(x).
\end{split}
  \]
  Finally, we obtain
  \begin{equation*}
\frac{d}{dt} (\tilde{\mu}_t,\varphi)  = \int_{\R^d}   <V[\tilde{\mu}_t](x),  \nabla_x \varphi(x)> d \tilde{\mu}_t(x) +  \int_{\R^d}    \varphi(x) dh[\tilde{\mu}_t](x)
 \end{equation*}
 which is the weak version of \eqref{eq:mfl}.
\end{proof}

 In order to prove Theorem \ref{Th:mfl2}, what is left is to show the convergence of $\mu^N$ to $\tilde{\mu}$, where $\mu^N$ is the empirical measure for the microscopic system defined in \eqref{eq:empmes}, and $\tilde{\mu}$ is defined from the solution to the graph limit equation by \eqref{eq:def_mu}. 
The key point in order to do that is to rewrite $\mu^N$ using the functions $x_N$ and $m_N$ introduced to perform the graph limit and to pass to the limit in that expression. More precisely, we prove the following proposition:

\begin{prop}\label{prop:convmu}
Let $x_0\in L^\infty(I;\R^d) $ and $m_0\in L^\infty(I;\R^d)$.
Let $(x^{N},m^{N})\in \mathcal{C}([0,T]; \R^d)^{N}\times \mathcal{C}([0,T];\R)^{N}$ satisfy the differential system~\eqref{eq:syst-gen} with initial condition 
$x^{0,N}=\Pdn(x_0)$ and $m^{0,N}=\Pdn(m_0)$ given by \eqref{eq:ICx} 
 and mass dynamics given by \eqref{eq:psi}-\eqref{eq:psiSk}.
%
Let $\mu^N$ be the empirical measure associated with $(x^N,m^N)$, i.e. for all $t\in [0,T]$,
\[
\mu^N_t(x) := \frac{1}{N} \sum_{i=1}^N m^N_i(t) \delta_{x^N_i(t)}.
\]
Secondly, 
let $(x,m)$ be the solutions to the integro-differential system \eqref{eq:GraphLimit-gen} with weight dynamics given by \eqref{eq:psiSk} and initial conditions given by $x(0,\cdot)=x_0$ and $m(0,\cdot)=m_0$.
Let
\[
\tilde{\mu}_t(x) := \int_I m(t,s) \delta(x-x(t,s)) ds.
\]

Then, for all test function $\varphi\in \mathcal{C}^\infty_c(\R^d)$, and all $t\in [0,T]$, it holds 
\begin{equation*}
\lim_{N\rightarrow\infty} \int_{\R^d} \varphi(x) (d\mu^N_t(x) - d\tilde{\mu}_t(x))  =0.
\end{equation*}
\end{prop}
\begin{proof}
Let $x_N= \Pcn(x^N)\in C([0,T];L^\infty(I;\R^d))$ and $m_N=\Pcn(m^N)\in \mathcal{C}([0,T];L^\infty(I;\R))$ defined by
\eqref{eq:xN} be the solutions to the integro-differential system \eqref{eq:syst-contN}-\eqref{eq:psiSk} with initial conditions $x_N(0,\cdot) = \Pcn(x^{0,N})$ and  $m_N(0,\cdot) = \Pcn(m^{0,N})$ given by \eqref{eq:ICx}. 
We begin by showing that for all test function $\varphi\in \mathcal{C}^\infty_c(\R^d)$,
\begin{equation}\label{eq:equivmeas}
\int_{\R^d} \varphi(x) d\mu^N_t(x) = \int_{\R^d} \varphi(x) d\tilde{\mu}_t^N(x) ,
\end{equation}
where
$\tilde{\mu}^N_t\in\PR$ is the measure defined by 
\[
\tilde{\mu}^N_t(x) := \int_I m_N(t,s) \delta(x-x_N(t,s)) ds.
\]
Since $x_N(t,s) = \sum_{i=1}^N x_i^{N}(t) \mathbf{1}_{[\frac{i-1}{N}, \frac{i}{N})}(s)$ and $m_N(t,s) = \sum_{i=1}^N m_i^{N}(t) \mathbf{1}_{[\frac{i-1}{N}, \frac{i}{N})}(s)$, we can compute:
\[
\begin{split}
 \int_{\R^d} \varphi(x) \int_I m_N(t,s) \delta(x-x_N(t,s)) ds = & \int_{\R^d} \varphi(x) \int_I \sum_{i=1}^N m_i^{N}(t) \mathbf{1}_{[\frac{i-1}{N}, \frac{i}{N})}(s) \delta(x-\sum_{i=1}^N x_i^{N}(t) ) ds\\
= & \int_{\R^d} \frac{1}{N} \sum_{i=1}^N  m_i^{N}(t) \varphi(x_i^{N}(t)) dx 
= \int_{\R^d}  \varphi(x) d\mu^N_t(x) .
\end{split}
\]

Secondly, we prove the following weak convergence: 
\begin{equation*}
\tilde{\mu}^N_t \rightharpoonup \tilde{\mu}_t \text{~when~} N \to \infty.
\end{equation*}
Using the definitions of $\tilde{\mu}^N_t$ and $\tilde{\mu}_t$, we write:
\[
\begin{split}
  \bigg|\int_{\R^d} \varphi(x) & (d\tilde\mu^N_t(x) - d\tilde{\mu}_t(x))\bigg| = \left| \int_{\R^d} \varphi(x)  \int_I (m_N(t,s) \delta(x-x_N(t,s)) - m(t,s) \delta(x-x(t,s))) \, ds \, dx \right| \\
 = &\left| \int_{I} \varphi(x_N(t,s)) m_N(t,s)  ds -  \int_{I} \varphi(x(t,s)) m(t,s)  ds \right| \\
 = & \left|   \int_{I} (\varphi(x_N(t,s))- \varphi(x(t,s)) ) m_N(t,s)  ds - \int_{I} \varphi(x(t,s)) (m(t,s)- m_N(t,s) ) ds\right| \\
 \leq & \left|  \left( \int_I (\varphi(x_N(t,s))- \varphi(x(t,s)) )^2ds \right)^{1/2} \left(\int_I m_N(t,s)^2 ds \right)^{1/2} \right| \\
       & + \left(\int_I \varphi(x(t,s))^2 ds \right)^{1/2} \left( \int_I (m_N(t,s) - m(t,s))^2 ds\right)^{1/2}\\
     \leq & \left|  K \left( \int_I ((x_N(t,s)- x(t,s) )^2ds \right)^{1/2} \left(\int_I m_N(t,s)^2 ds \right)^{1/2} \right| \\
       & + \left(\int_I \varphi(x(t,s))^2 ds \right)^{1/2} \left( \int_I (m_N(t,s) - m(t,s))^2 ds\right)^{1/2}
 \end{split}
\]
 for a certain constant $K>0$, using the fact that $\varphi$ is a regular enough test function. Moreover, since the solution to \eqref{eq:syst-gen} 
 satisfies \eqref{eq:bornemN} and since $x \in \mathcal{C}([0;T];L^\infty(I;\R^d))$, $\varphi$ being regular, there exists a constant $C>0$ such that we finally have\\
\[
\begin{split} 
& \left|\int_{\R^d} \varphi(x) (d\tilde\mu^N_t(x) - d\tilde{\mu}_t(x)) dx \right|\\
    \leq &  K \overline{M} \left( \int_I ((x_N(t,s)- x(t,s) )^2ds \right)^{1/2}  + C \left( \int_I (m_N(t,s) - m(t,s))^2 ds\right)^{1/2}.
\end{split}
\]
 Proposition \ref{prop:conv} allows us to conclude that
\[
\lim_{N\rightarrow\infty} \left|\int_{\R^d} \varphi(x) (d\mu^N_t(x) - d\tilde{\mu}_t(x)) dx \right| =0, 
\]
and together with the equality \eqref{eq:equivmeas}, this proves the desired convergence.
\end{proof}

Theorem \ref{Th:mfl2} is thus proven by combining Propositions \ref{prop:conv}, \ref{Prop:link_mfl-gl} and \ref{prop:convmu}.

\section{Numerical simulations}
\label{sec:numeric}

\subsection{Dynamics not preserving indistinguishability}\label{sec:numericnotindisting}

In this section, we illustrate the convergence of the solutions of the microscopic model \eqref{eq:syst-gen} to the solution of the graph limit equation satisfying  \eqref{eq:GraphLimit-gen}, as stated in Theorem \ref{Th:GraphLimit}.
We focus on mass dynamics that \textit{do not} preserve indistinguishability, so for which the classical mean-field limit process does not hold.
In particular, we consider a situation in which the agents are divided into $K$ groups $I_k$, with $k\in\{1,\cdots,K\}$. Each group is composed of leaders $I_k^L$ and followers $I_k^F$ so that $I_k=I_k^L\cup I_k^F$ and $I_k^L\cap I_k^F=\emptyset$. Within each group, the weight of each leader increases proportionally to itself and to the total weight of all the followers of the group. Conversely, the weight of each follower decreases proportionally to itself and to the total weight of all the leaders. More specifically,  
consider the function $\psi$ given by 
\begin{equation}\label{eq:modelsimuGL2}
\forall x\in L^2(I;\R),\quad \forall m\in L^2(I;\R),\qquad \psi(s,x,m) = 
\begin{cases}
\beta m(s) \int_{I_k^F} m(s')ds' \quad \text{ if } s\in I_k^L \\
-\beta m(s) \int_{I_k^L}  m(s')ds' \quad \text{ if } s\in I_k^F.
\end{cases}
\end{equation}

One easily checks that the total mass $\int_I m(s) ds$ is conserved, and that $\psi$ satisfies \eqref{eq:psilip2} on any time interval $[0,T]$ with $T<\infty$.

Given a total number of groups $K\in\N$ and a proportion $r\in (0,1)$ of leaders in each group, we define $I_k := [\frac{k-1}{K},\cdots,\frac{k}{K})$, $I_k^L = [\frac{k-1}{K},\cdots,\frac{k-1}{K}+\frac{r}{K})$ and $I_k^F = [\frac{k-1}{K}+\frac{r}{K},\cdots \frac{k}{K})$.
Provided that $n := N/K \in\N$ and that $rn\in\N$,
the corresponding microscopic dynamics can be written simply as: 
\begin{equation}\label{eq:modelsimumicro2}
\forall k\in\{1,\cdots ,K\}, \quad \dot m_{(k-1)n + i} = 
\begin{cases}
\displaystyle \frac{\beta}{N} m_{(k-1)n +  i} \sum_{j=rn+1}^{n} m_{(k-1)n+j} \quad \text{ if } i\in\{1,\cdots , r n\}, \\
\displaystyle - \frac{\beta}{N} m_{(k-1)n + i} \sum_{j=1}^{rn} m_{(k-1)n+j} \quad \text{ if } i\in\{rn+1,\cdots , n\}. 
 \end{cases}
\end{equation}

We show the behavior of the model and the convergence of the microscopic dynamics to the macroscopic ones in cases in which we have one ($K=1$) and two ($K=2$) groups, with a proportion $r=10\%$ of leaders in each one. The initial conditions for the graph limit equation were taken to be $s\mapsto x_0(s) = \sin^2(4s)$ and $s\mapsto m_0(s)= 1/M_0*s \cos^2(5s)$, with $M_0:=\int_0^1 s \cos^2(5s) ds$. The corresponding initial conditions for the microscopic model were computed from \eqref{eq:ICx}. In all the following examples, the interaction function used was $y\mapsto a(y) = \frac{1}{1+y^2}$.

Figure \ref{fig:Micro_K1} shows the evolution of the opinions and weights of 20 agents divided into 2 leaders (indexed $i=1$ and $i=2$) and 18 followers (indexed $i=3... 20$). Observe that the leaders' weights quickly increase to sum up to the total weight of the group. As a consequence, consensus is achieved at a value $x=0.2$ close to the leaders' initial positions.
Figure \ref{fig:microGL2} illustrate the convergence of the microscopic dynamics to the graph limit ones by comparing the opinions and weights for $N=100$.

\begin{figure}[h]
\begin{center}
\includegraphics[scale=0.18]{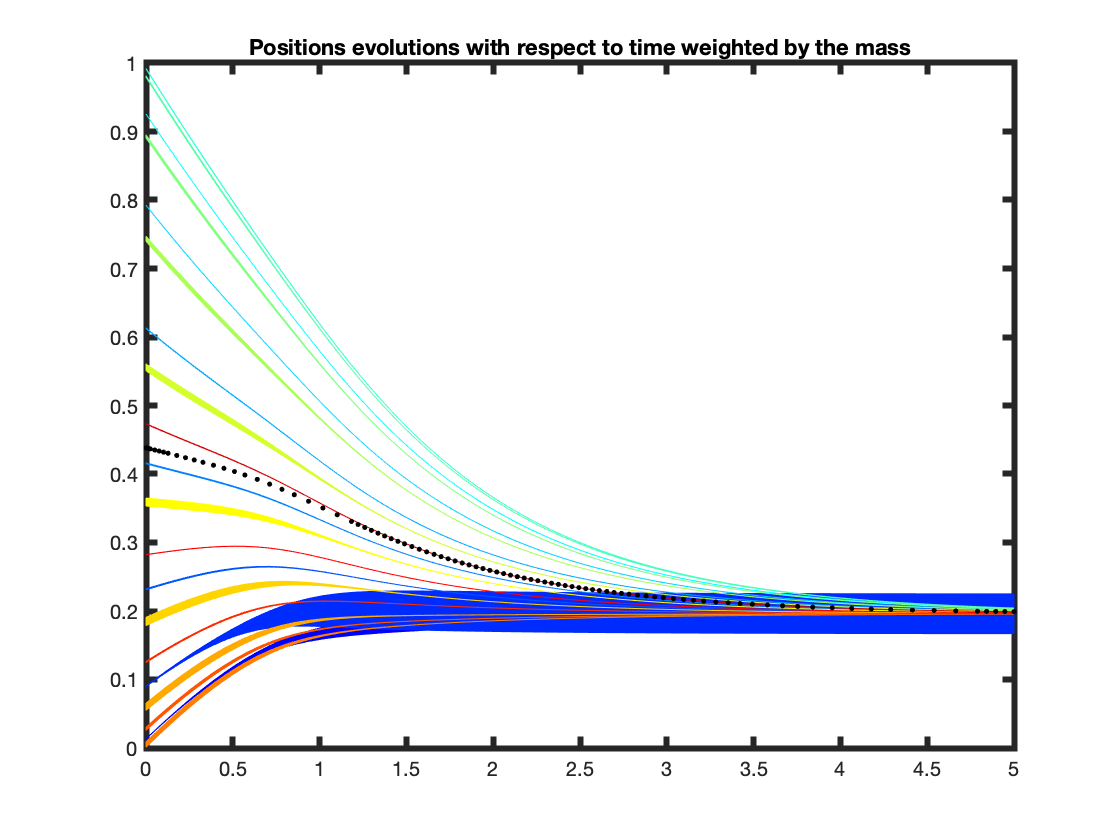}
\includegraphics[scale=0.18]{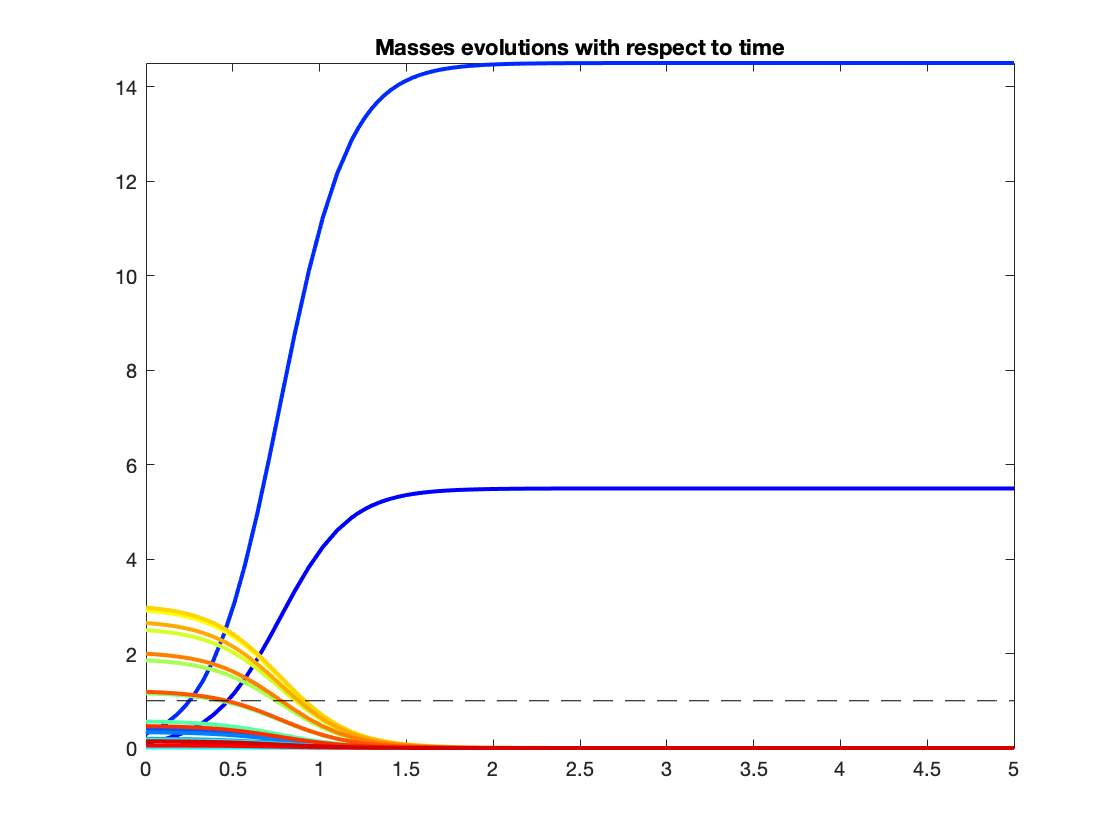}
\includegraphics[scale=0.5, clip=true]{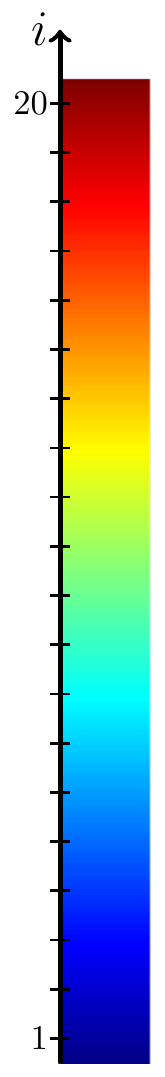}
\end{center}
\caption{ Evolution of opinions (left) and weights (right) of 20 agents. Center: color scale for the index $i$ ranging from 1 (blue) to 20 (red). Parameters: $K=1$, $r=0.1$, $\beta=5$.}\label{fig:Micro_K1}
\end{figure}

\begin{figure}[!h]
    \centering
        \includegraphics[width=0.32\textwidth]{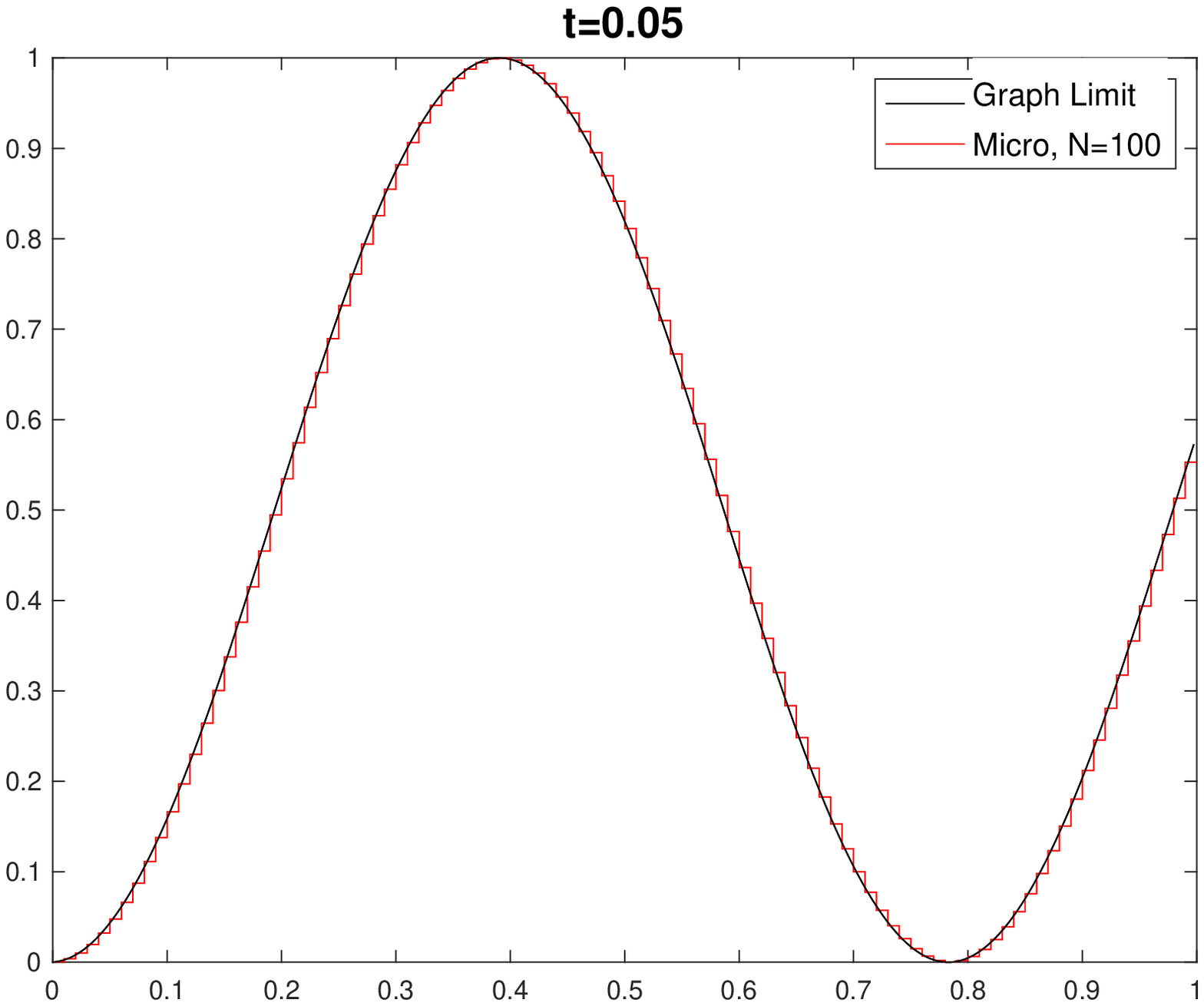}
        \includegraphics[width=0.32\textwidth]{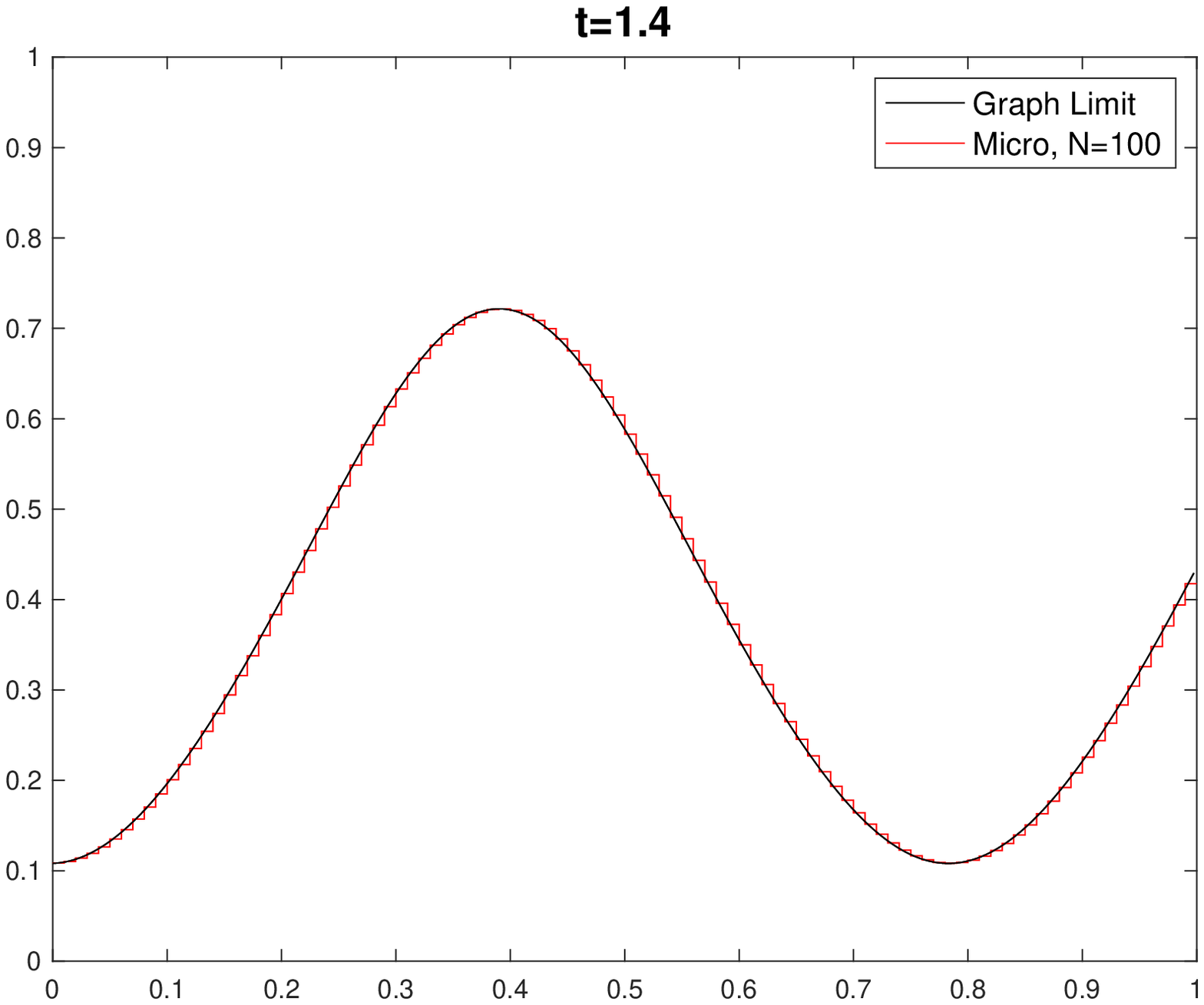}
        \includegraphics[width=0.32\textwidth]{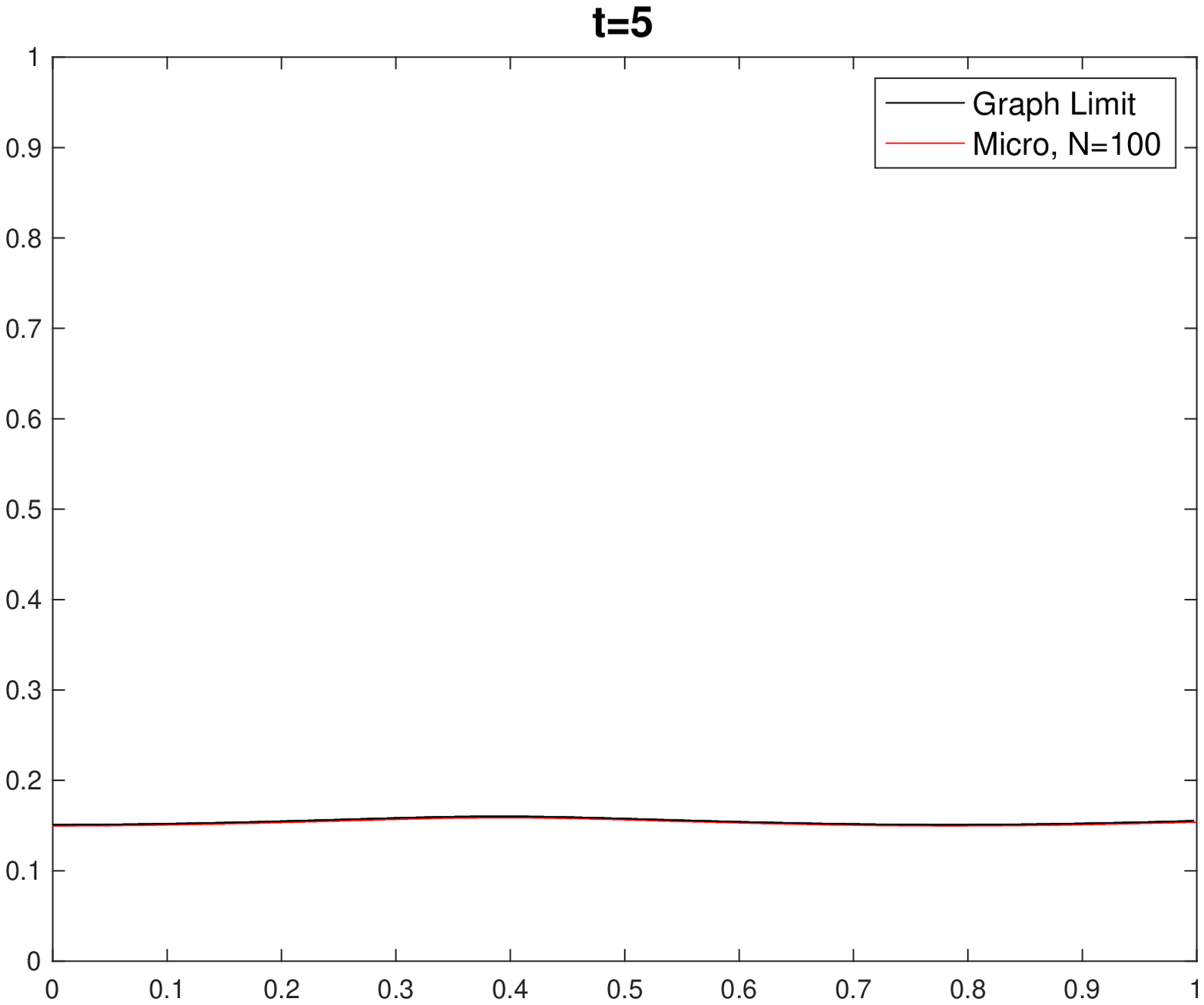} \\
        \includegraphics[width=0.32\textwidth]{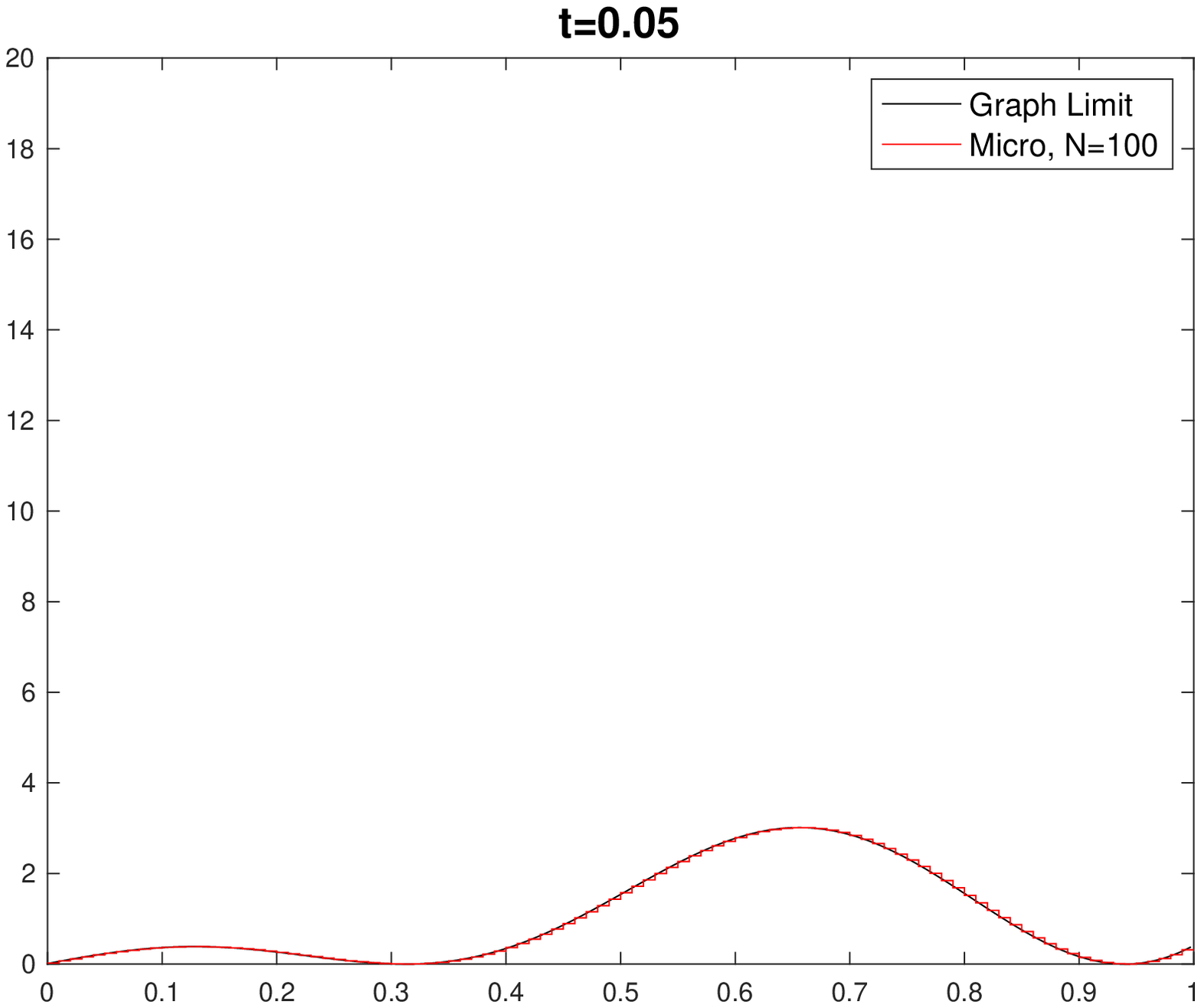}
        \includegraphics[width=0.32\textwidth]{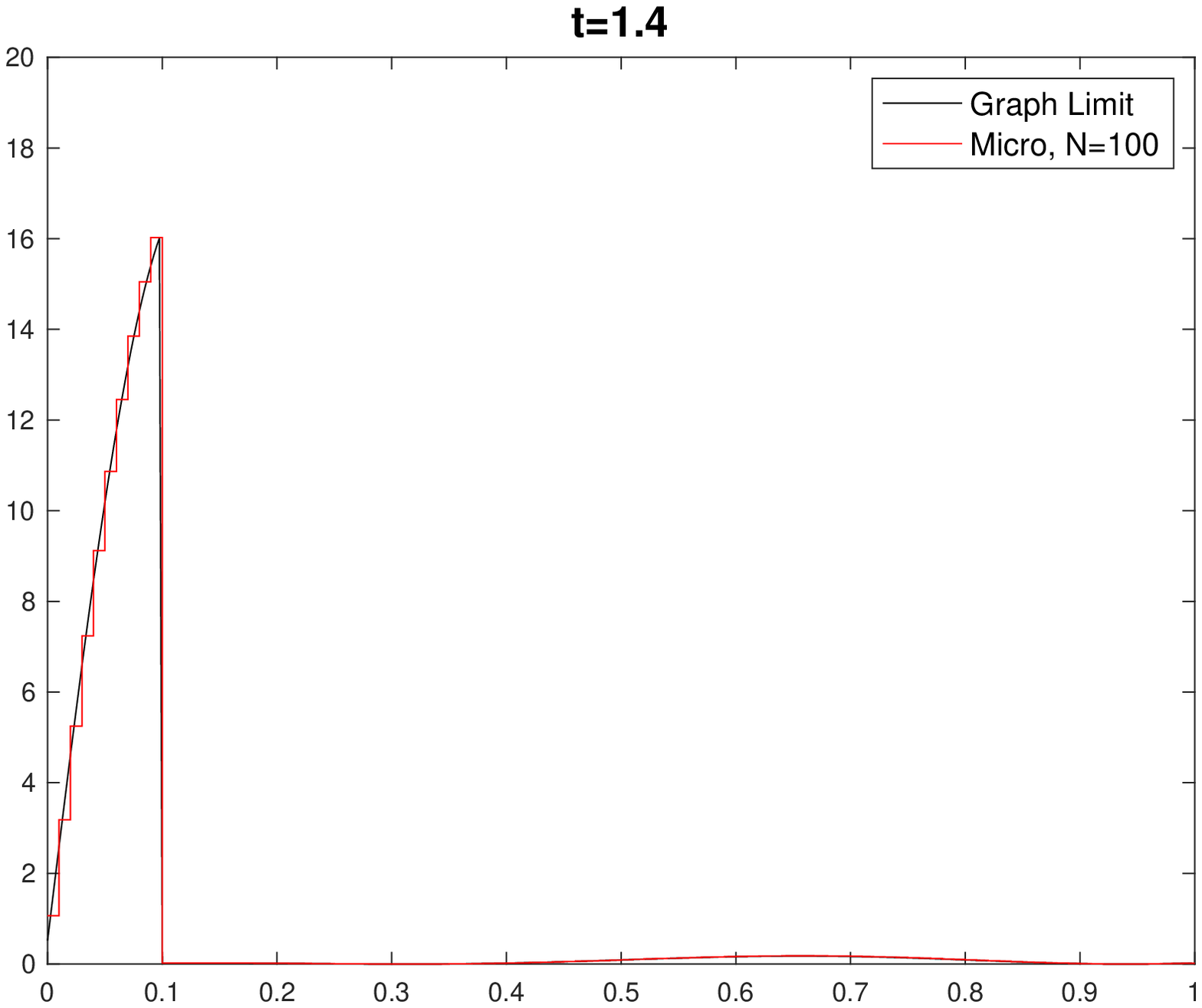}
        \includegraphics[width=0.32\textwidth]{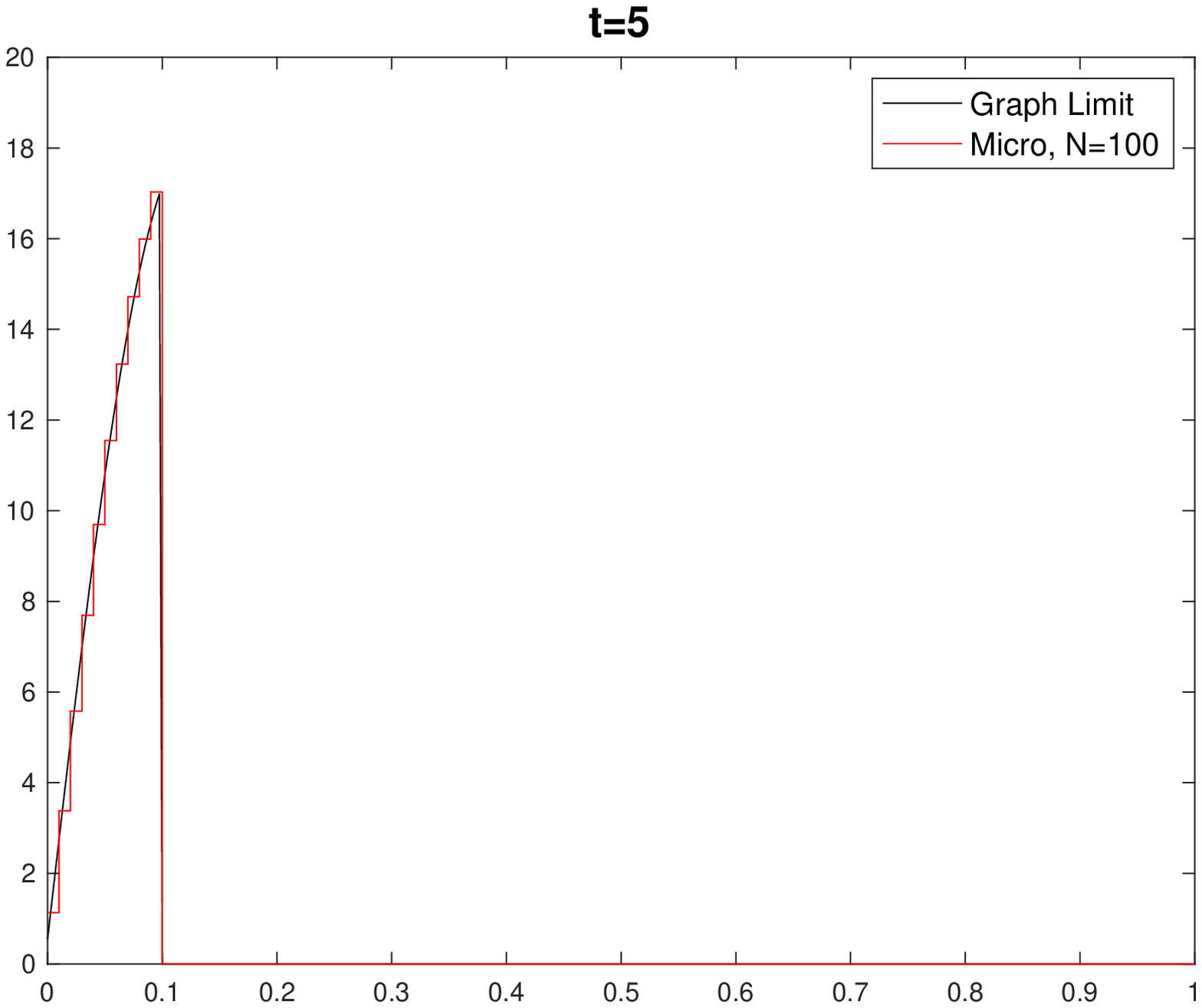} \\    
    \caption{\label{fig:microGL2} Representation of the solutions $s\mapsto x_N(t,s)$ (top) and $s \mapsto m_N(t,s)$ (bottom) to the microscopic dynamics \eqref{eq:syst-gen}-\eqref{eq:modelsimumicro2} and of the solutions $s\mapsto x(t,s)$ (top) and $s\mapsto m(t,s)$ (bottom) to the Graph Limit equation \eqref{eq:GraphLimit-gen}-\eqref{eq:modelsimuGL2} for $t=0.05$, $t=1.4$ and $t=5$. Parameters: $K = 1$, $r = 0.1$, $\beta = 5$.}
\end{figure}

Figure \ref{fig:Micro_K2} shows the evolution of the opinions and weights of 20 agents divided into 2 groups, each one containing one leader (respectively indexed $i=1$ and $i=11$) and 9 followers (respectively indexed $i=2... 10$ and $i=12...20$).
Observe that the second group leaders' weights increases much faster than the first group leaders' weight, due to the fact that the total weigh of the second group is larger than the total weight of the first group. As a consequence, consensus is achieved at a value $x=0.6$ close to the second leader's initial position.
Figure \ref{fig:microGL3} illustrate the convergence of the microscopic dynamics to the graph limit ones by comparing the opinions and weights for $N=100$.

\begin{figure}[h]
\begin{center}
\includegraphics[scale=0.18]{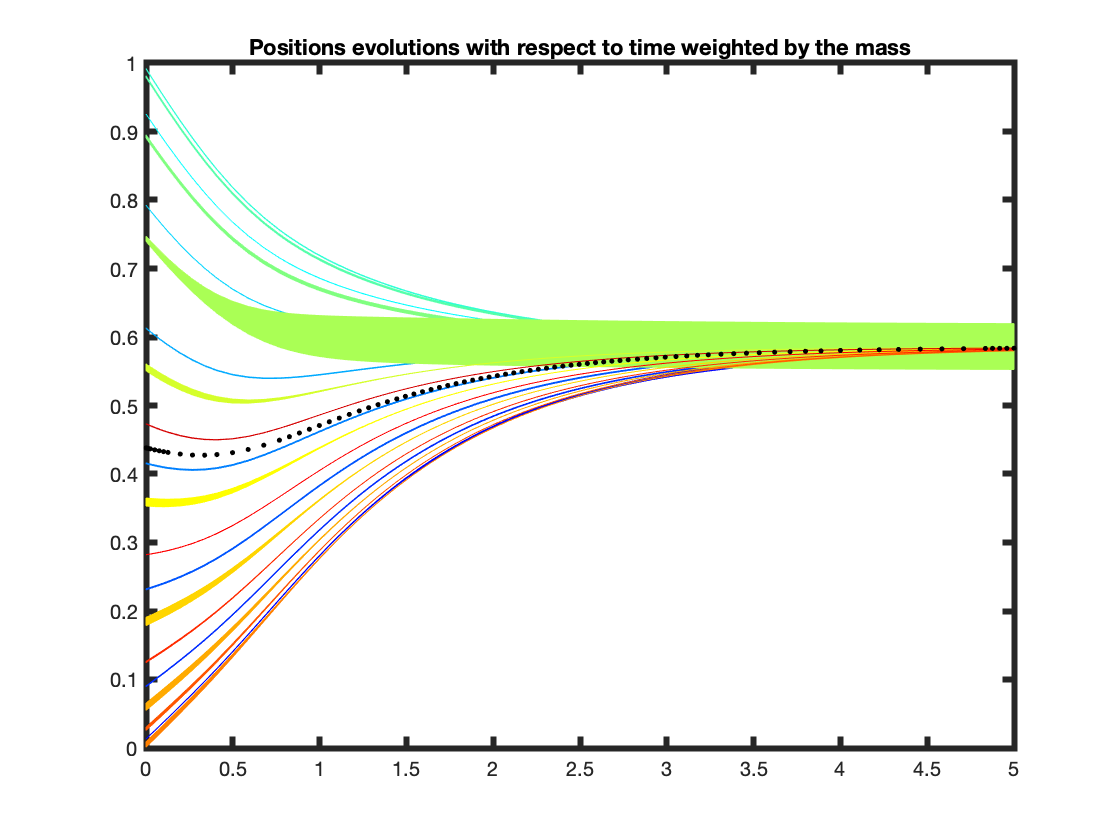}
\includegraphics[scale=0.18]{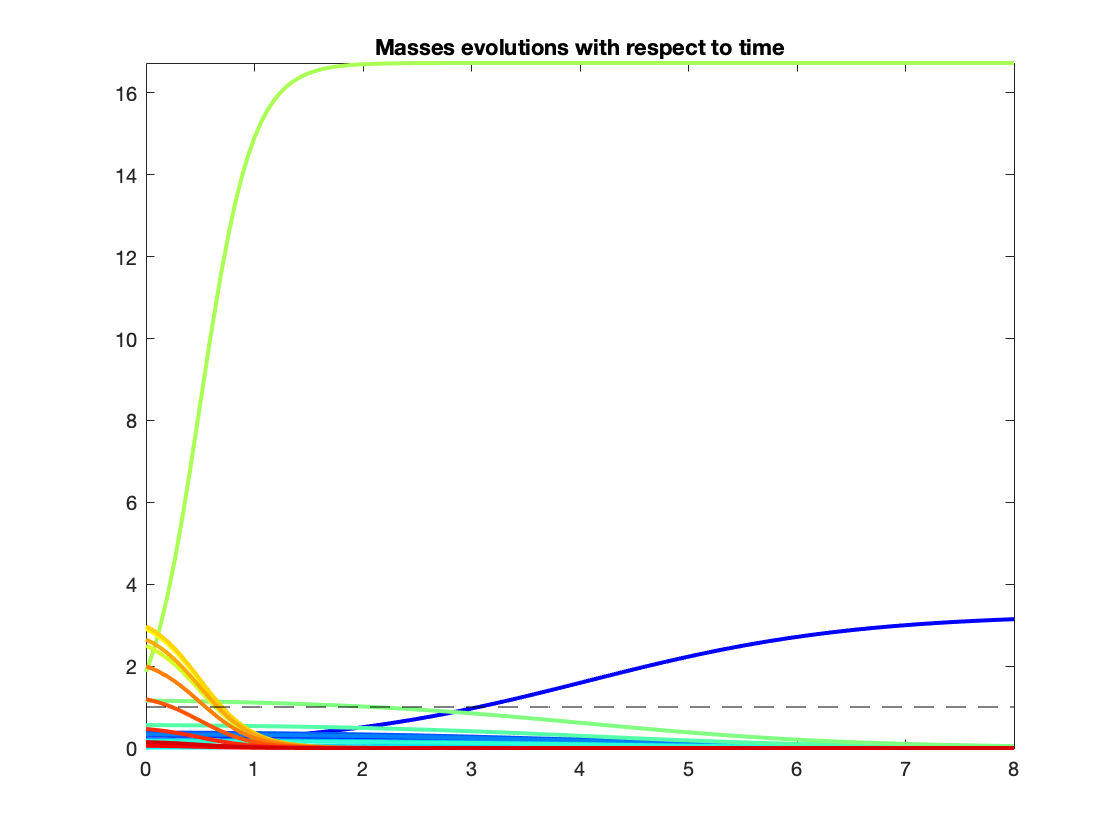}
\includegraphics[scale=0.5, clip=true]{SpaceI}
\end{center}
\caption{ Evolution of opinions (left) and weights (right) of 20 agents. Center: color scale for the index $i$ ranging from 1 (blue) to 20 (red). Parameters: $K=2$, $r=0.1$, $\beta=5$.}\label{fig:Micro_K2}
\end{figure}

\begin{figure}[!h]
    \centering
        \includegraphics[width=0.32\textwidth]{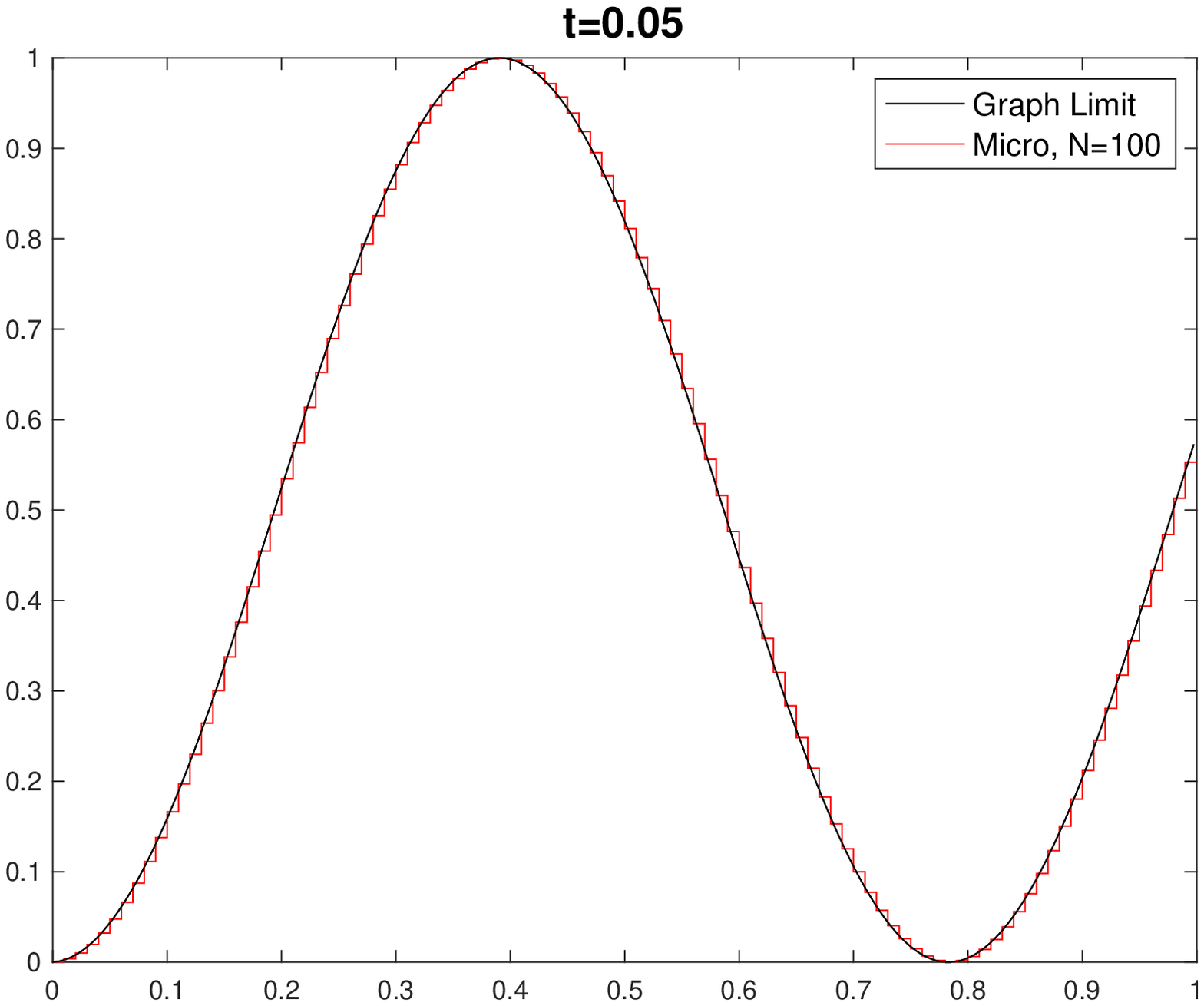}
        \includegraphics[width=0.32\textwidth]{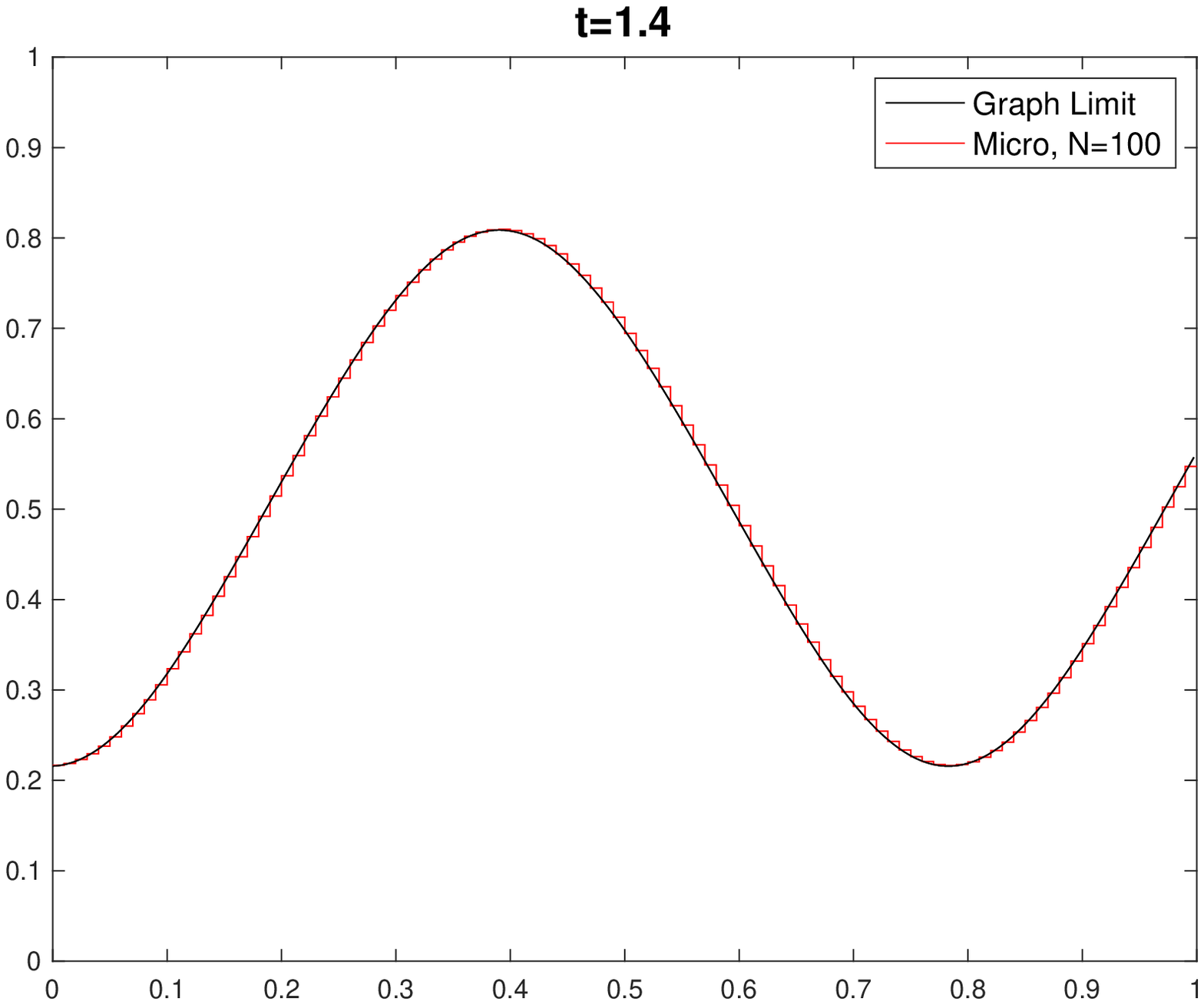}
        \includegraphics[width=0.32\textwidth]{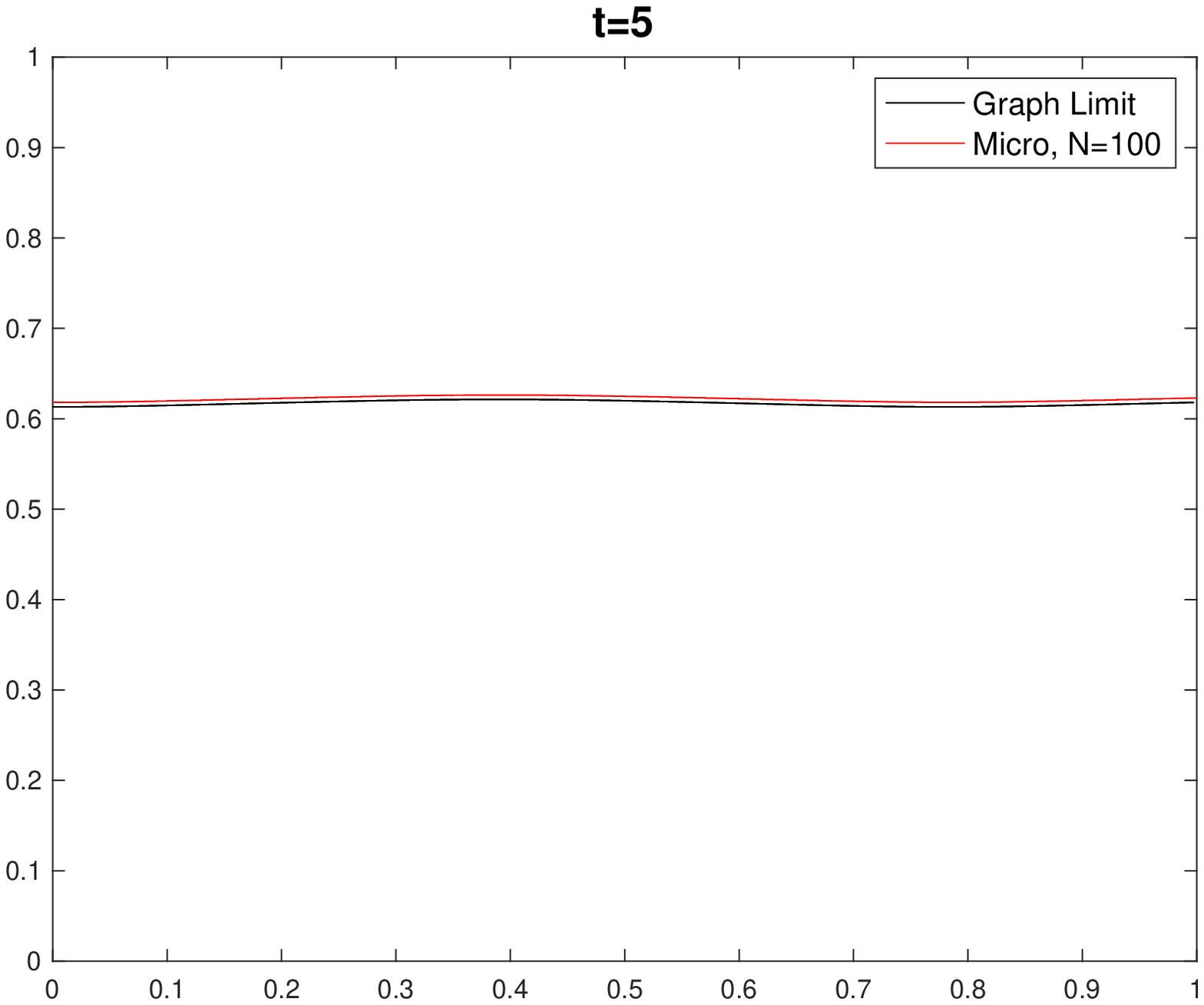} \\
        \includegraphics[width=0.32\textwidth]{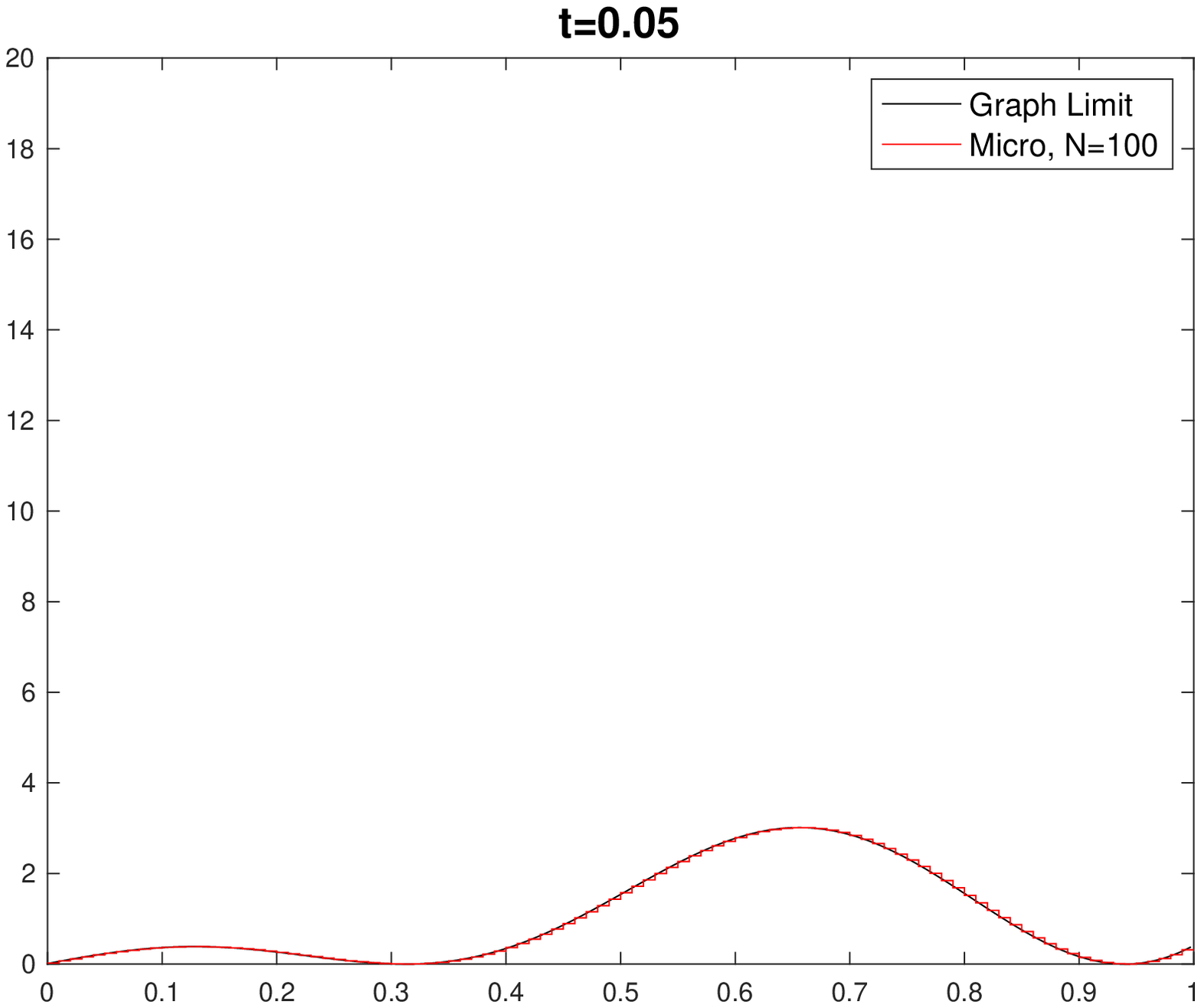}
        \includegraphics[width=0.32\textwidth]{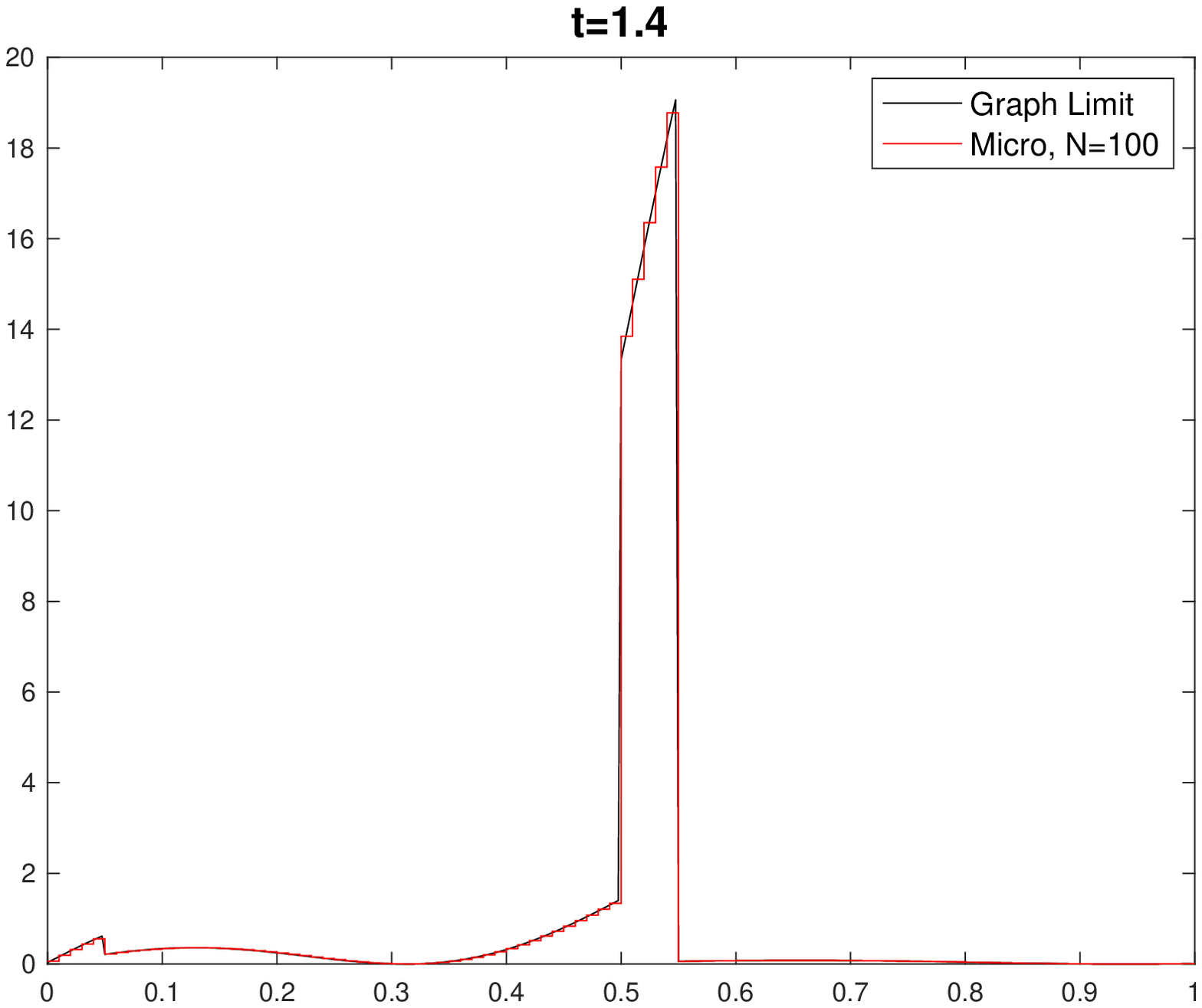}
        \includegraphics[width=0.32\textwidth]{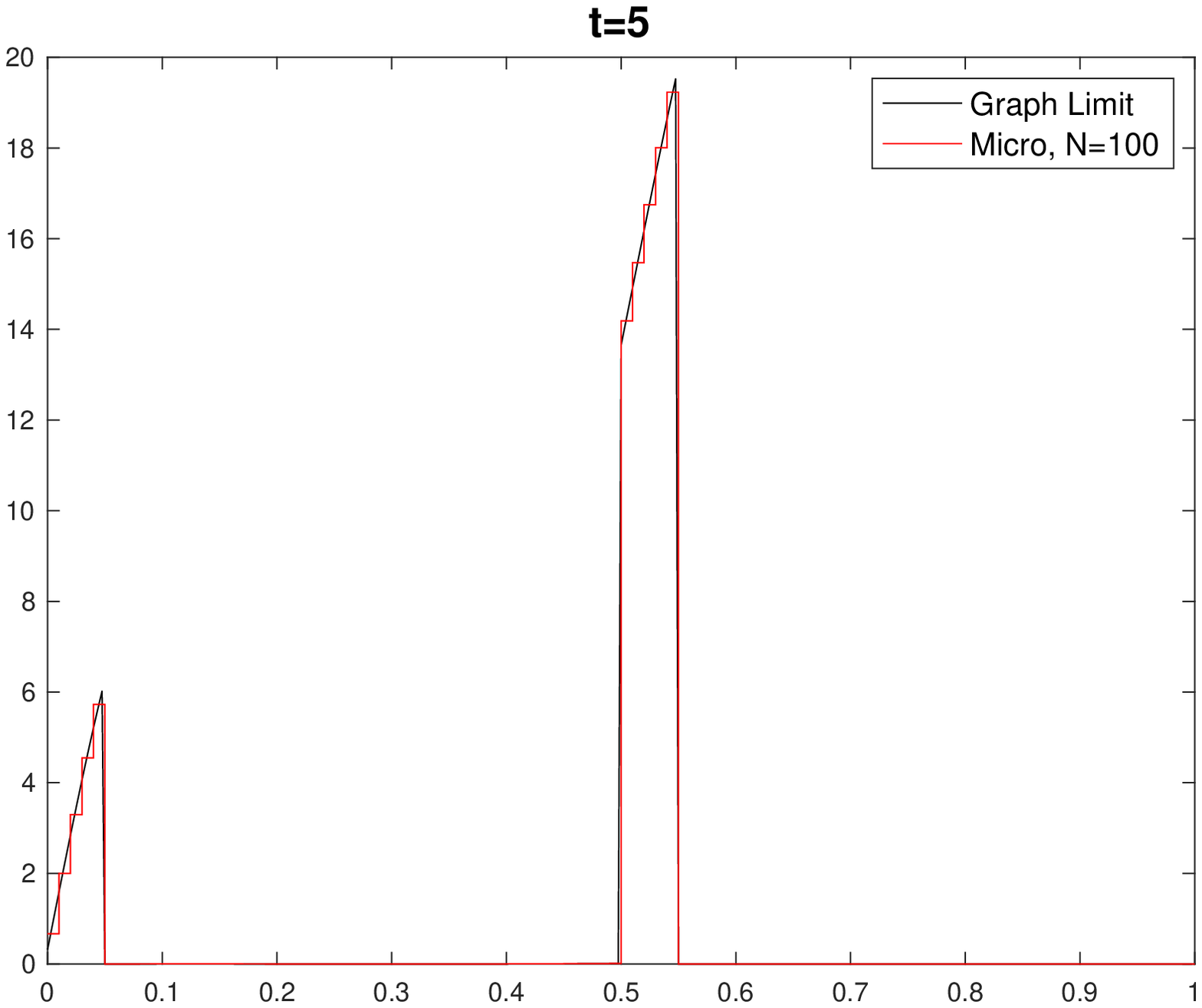} \\    
    \caption{\label{fig:microGL3} Representation of the solutions $s\mapsto x_N(t,s)$ (top) and $s \mapsto m_N(t,s)$ (bottom) to the microscopic dynamics \eqref{eq:syst-gen}-\eqref{eq:modelsimumicro2} and of the solutions $s\mapsto x(t,s)$ (top) and $s\mapsto m(t,s)$ (bottom) to the Graph Limit equation \eqref{eq:GraphLimit-gen}-\eqref{eq:modelsimuGL2} for $t=0.05$, $t=1.4$ and $t=5$. Parameters: $K = 2$, $r = 0.1$, $\beta = 5$.}
\end{figure}

\subsection{Dynamics preserving indistinguishability} \label{sec:numericindisting}

\subsubsection{The model}

In this first series of simulations, we will consider a particular case of mass dynamics of the form \eqref{eq:psiSk}. Recall that as shown in Section \ref{sec:indisting}, all mass dynamics of this form preserve indistinguishability, thus we can study the two limits - graph and mean-field - of the microscopic model. 
More precisely, let us focus on the following  model associated with the weight dynamics 
\begin{equation}\label{eq:modelsimumicro}
\psi_i(x,m) = \frac{1}{N} m_i \left(  \frac{1}{N} \sum_{k=1}^N \sumj m_k m_j \| \phi(x_l-x_j)\| - \sumj m_j \| \phi(x_i-x_j)\|\right).
\end{equation} 
Let us explain its origin. We denote by $e_{j \to i}$ the influence of $j$ on $i$ and define it as $e_{j \to i} = m_j \phi(x_i-x_j)$. Let $e_i$ represent the total group influence on $i$, defined as $$e_i=\sum_{j=1}^N e_{j \to i} = \sum_{j=1}^N m_j \|\phi(x_i-x_j)\|.$$ 
Now denoting by $\overline{e}$ the weighted average of the total group influence $$\overline{e} = \sum_{k=1}^N \frac{m_k}{N} e_k = \sum_{k=1}^N \sum_{j=1}^N \frac{m_k}{N} m_j \|\phi(x_k-x_j)\|,$$ 
the mass dynamics of model \eqref{eq:modelsimumicro} can be rewritten as: 
\[
\psi_i(x,m) = \frac{1}{N} m_i \left( \overline{e} - e_i \right).
\]
Thus, in our model, if the group influence on $i$ is lower than the weighted average of the group influence among the population (i.e. agent $i$ is being less influenced than average), $i$ gains weight proportionally to this difference and to its own weight $m_i$.
On the other hand, if the group influence on $i$ is higher than the weighted average among the population (i.e. agent $i$ is more influenced than average), $i$ loses weight proportionally to this difference and to its own weight $m_i$.
In other words, in this model, the less influenced agents gain weight, thus becoming the more influential. 

The Graph Limit of the microscopic model \eqref{eq:syst-gen}-\eqref{eq:modelsimumicro} is given by \eqref{eq:GraphLimit-gen}, with 
\begin{equation}\label{eq:modelsimuGL}
\psi(s,x,m) = m(s) \left( \int_I \int_I  m(s_*) m(\tilde{s}) \| \phi(x(\tilde s)-x(s_*))\|\, ds_*\, d\tilde{s} - \int_I m(s_*) \| \phi(x(s)-x(s_*))\|\, ds_* \right).
\end{equation}

As seen in Section \ref{sec:mfl}, the mean-field limit of \eqref{eq:syst-gen}-\eqref{eq:modelsimumicro} is given by the transport equation with source \eqref{eq:mfl}, with
\begin{equation}\label{eq:modelsimumacro}
h[\mu](x) = \left(\int_{\R^d} \int_{\R^d} \|\phi(y-z)\| d\mu(z) d\mu(y) - \int_{\R^d} \|\phi(y-x)\| d\mu(y)\right) \mu(x).
\end{equation}

\subsubsection{Numerical results}

For the simplicity of numerical simulations, we take $d=1$ and
we choose an interaction function compactly supported on $[0,R]$. Let $a\in \mathcal{C}(\R^+, \R)$ be defined by $a:\delta \mapsto a(\delta) =  \frac{1}{\delta} \sin^2( \frac{\pi}{R}\delta ) \, \one_{(0,R)}(\delta)$,
so that for all $x\in\R$, 
\[
\ba(x) = a(|x|) x = 
\begin{cases}
 \frac{x}{|x|} \sin^2( \frac{\pi}{R}|x| )\quad \text{ for all } x\in (-R,R)\setminus 0 \\
 0\quad \text{ otherwise}.
 \end{cases}
\]

Initial conditions are given by the functions $s\mapsto x^0(s)$ and $s\mapsto m^0(s)$ defined by 
\begin{equation}\label{eq:ICsimuGL}
\begin{cases}
x_0(s) = \frac{1}{\pi}\arccos(2s-1) \\
m_0(s) = \tilde{m}^0(s) (\int_I \tilde{m}_0(s_*) ds_*)^{-1}
\end{cases}
\end{equation}
where $\tilde{m}_0(s) =  s^{1/4}\cos^2(5s)+ 0.2 s^2+0.5$.
Graphical representations of $x_0$ and $m_0$ can be found in the left panels of Figure \ref{fig:microGL}. 
Notice that all opinions are initially in the interval $[0,1]$, thus if $R\geq 1$, all agents interact with all others, and we expect consensus. From here onward, we choose $R=0.2$.

We begin by showing numerical simulations of the microscopic model \eqref{eq:syst-gen}-\eqref{eq:modelsimumicro} for $N=30$. Initial conditions were computed from \eqref{eq:ICsimuGL}-\eqref{eq:ICx}.
The left panel of Figure \ref{fig:micro_30} shows the time evolution of the opinions $t\mapsto x_i(t)$ in which the (time-dependent) thickness of the lines is proportional to the corresponding weights $t\mapsto m_i(t)$, whose evolution is shown in the right panel.
Notice that due to the compact support of the interaction function, the population divides into three clusters separated by distances greater than $R$, the interaction radius.
Although the weights initially all start within the interval $[0.5,1.6]$, the weight dynamics spread the weights by leading the least influenced agents to gain mass. This can be observed in the evolution of $m_{30}$ (represented in light green), which feels little group influence since $x_{30}$ is at the lower edge of the group. Likewise, $m_1$ (in red) initially increases since $x_1$ is at the upper edge of the group, but after $x_1$ joins its closest neighbors, the group influence that it feels increases and $m_1$ decreases.
Observe also that the total mass is conserved, as shown by the constant evolution of the average mass (black dotted line).

\begin{figure}[!h]
    \centering
    \begin{subfigure}[b]{0.45\textwidth}
        \includegraphics[width=\textwidth]{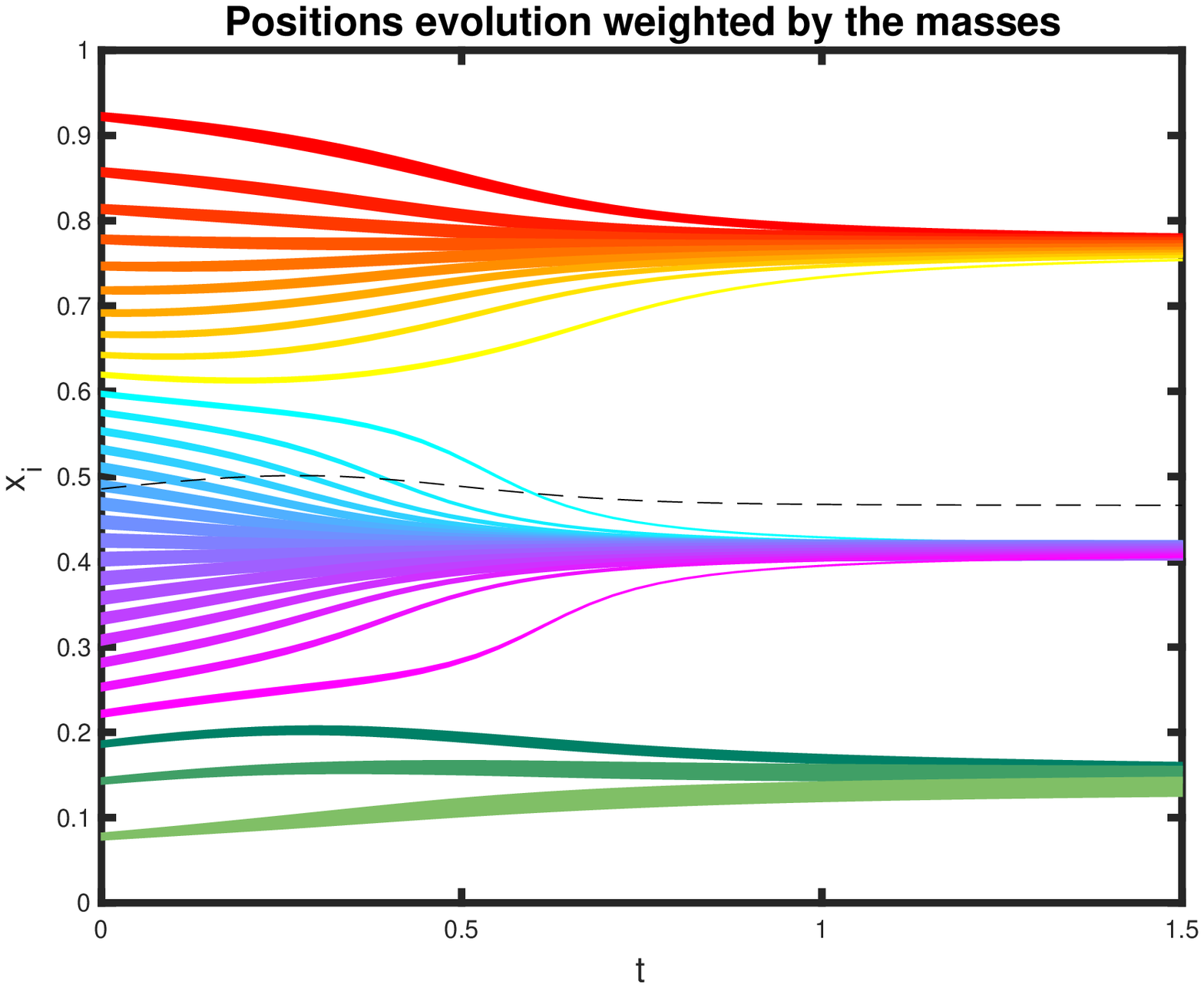}
    \end{subfigure}
    \begin{subfigure}[b]{0.45\textwidth}
        \includegraphics[width=\textwidth]{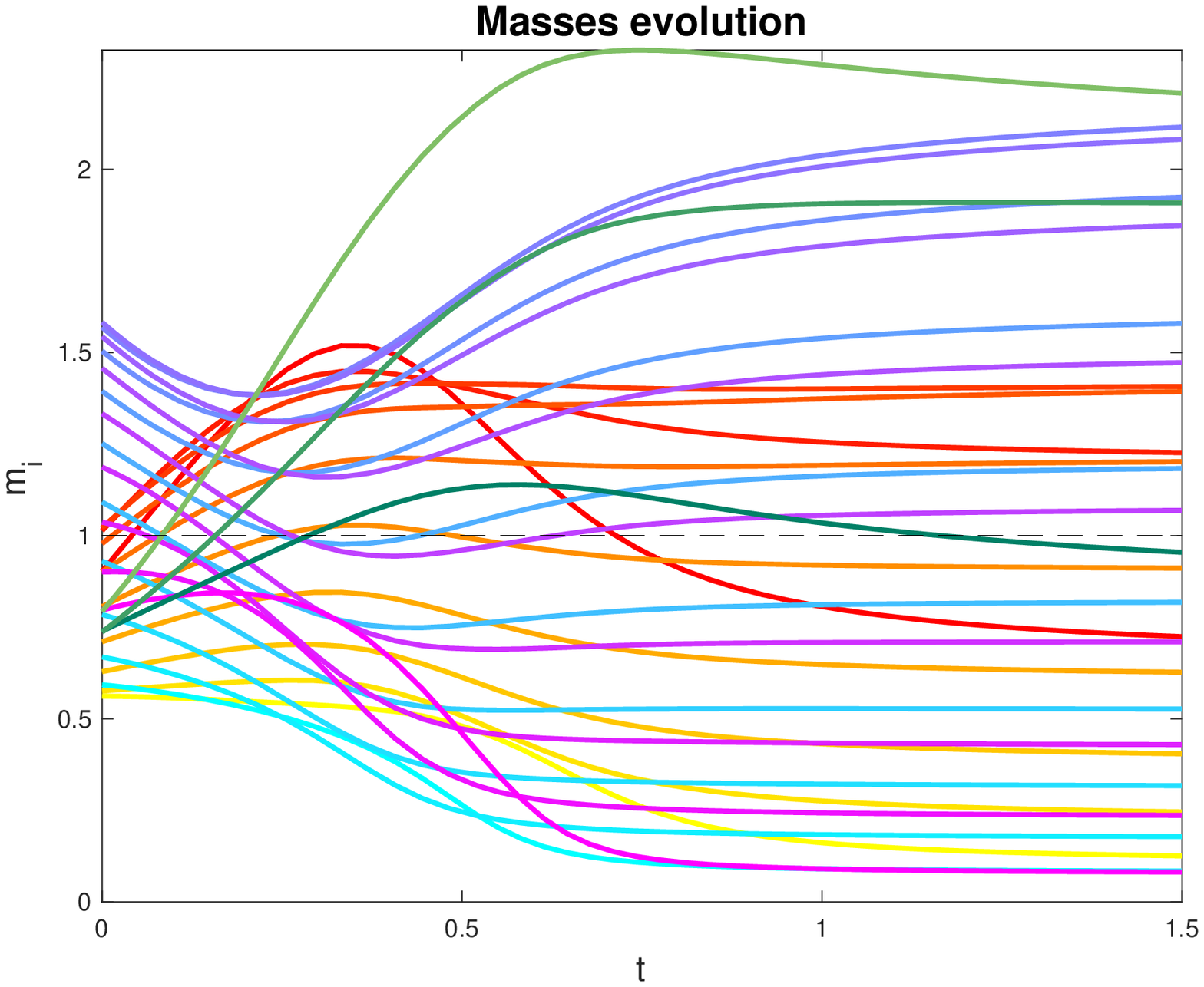}
    \end{subfigure}
    \caption{\label{fig:micro_30} Evolution of opinions (left) and weights (right) for the microscopic model \eqref{eq:syst-gen}-\eqref{eq:modelsimumicro} with $N=30$ agents. The weighted barycenter $\bar x = \frac{1}{N}\sum_{i=1}^N m_i x_i$ and the average mass $\bar m = \frac{1}{N}\sum_{i=1}^N m_i$ are represented by black dotted lines. In the left panel, the thickness of the curve representing $x_i(t)$ is proportional to the corresponding weight $m_i(t)$.}
\end{figure}

Figure \ref{fig:microGL} shows the profile of the mean-field limit $s\mapsto x(t,s)$ and $s\mapsto m(t,s)$ solving \eqref{eq:GraphLimit-gen}-\eqref{eq:modelsimuGL} with initial conditions given by \eqref{eq:ICsimuGL} at times $t=0$, $t=0.45$ and $t=1.5$ (black line). For comparison purposes, the solutions $s\mapsto x^N(t,s)$ and $s\mapsto m^N(t,s)$ to the microscopic model \eqref{eq:syst-gen}-\eqref{eq:modelsimumicro} for $N=50$ is plotted on the same figures, using the representation via step functions given by \eqref{eq:xN} (red line). 
The clustering behavior is now shown by the convergence of $x$ and $x^N$ to a step function taking three distinct values.
Notice that the weight function converges to a function with three local maxima attained at the centers of the three clusters, while the agents at each cluster's edge form local minima.

\begin{figure}[!h]
    \centering
        \includegraphics[width=0.3\textwidth]{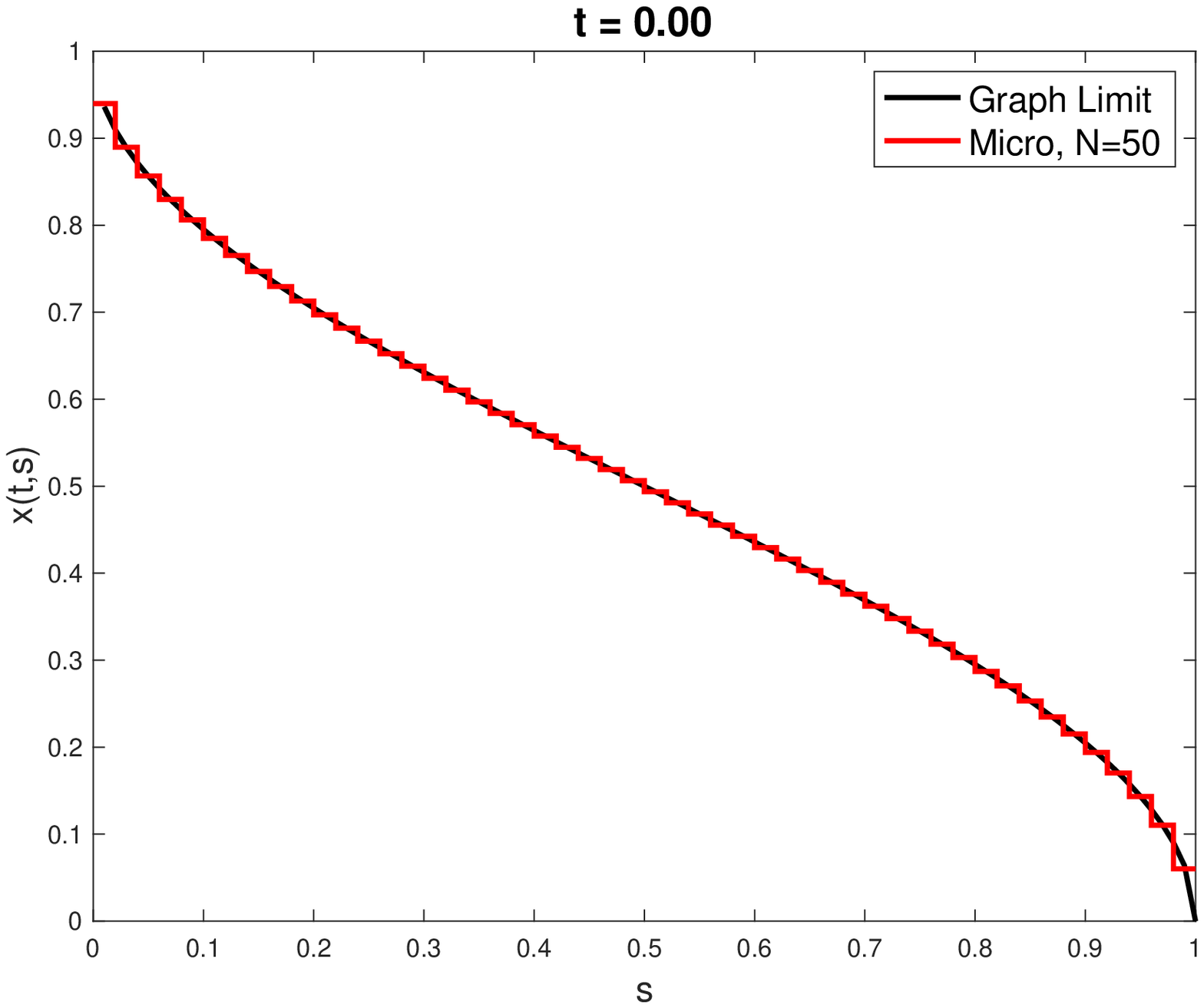}
        \includegraphics[width=0.3\textwidth]{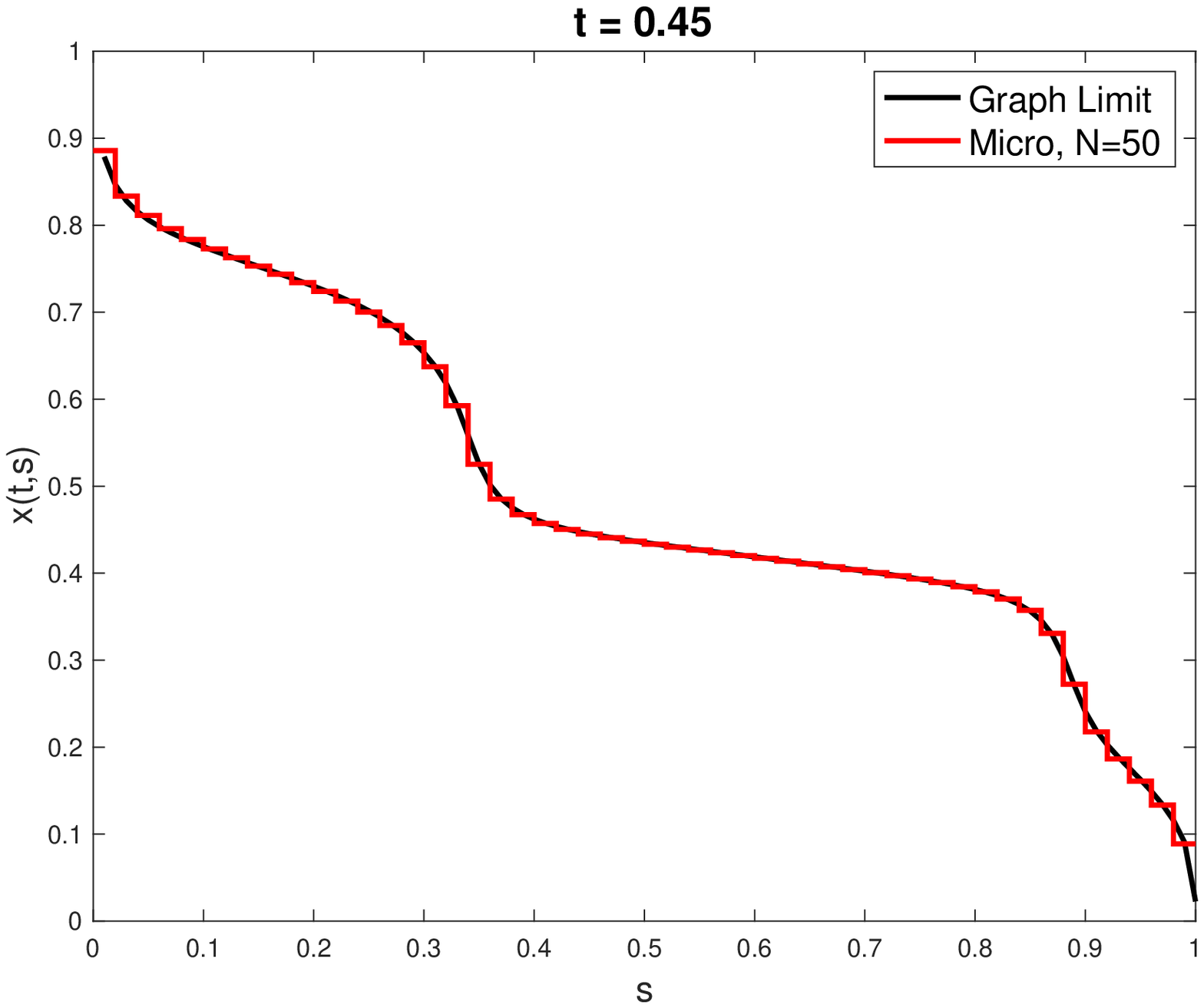}
        \includegraphics[width=0.3\textwidth]{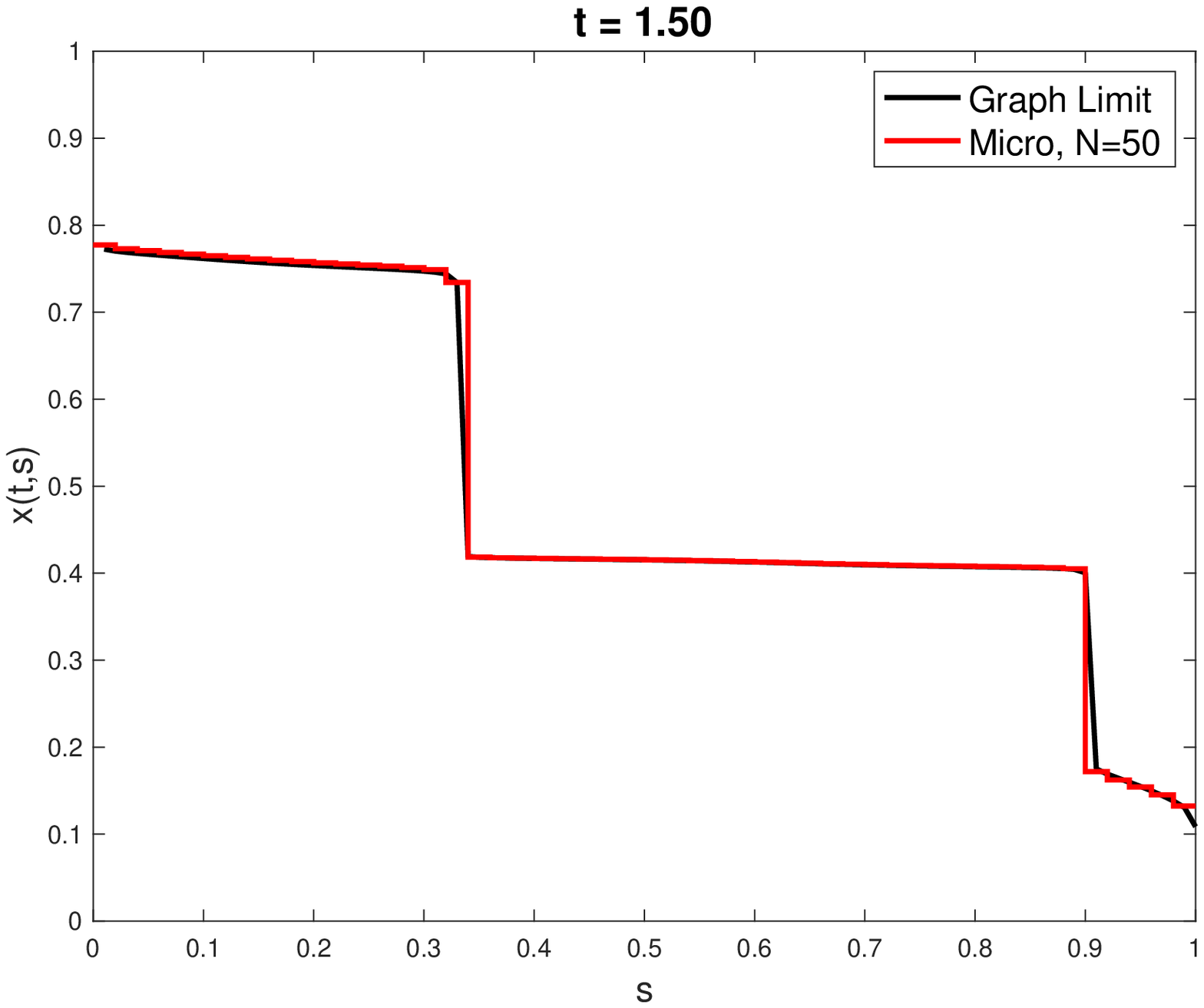} \\
        \includegraphics[width=0.3\textwidth]{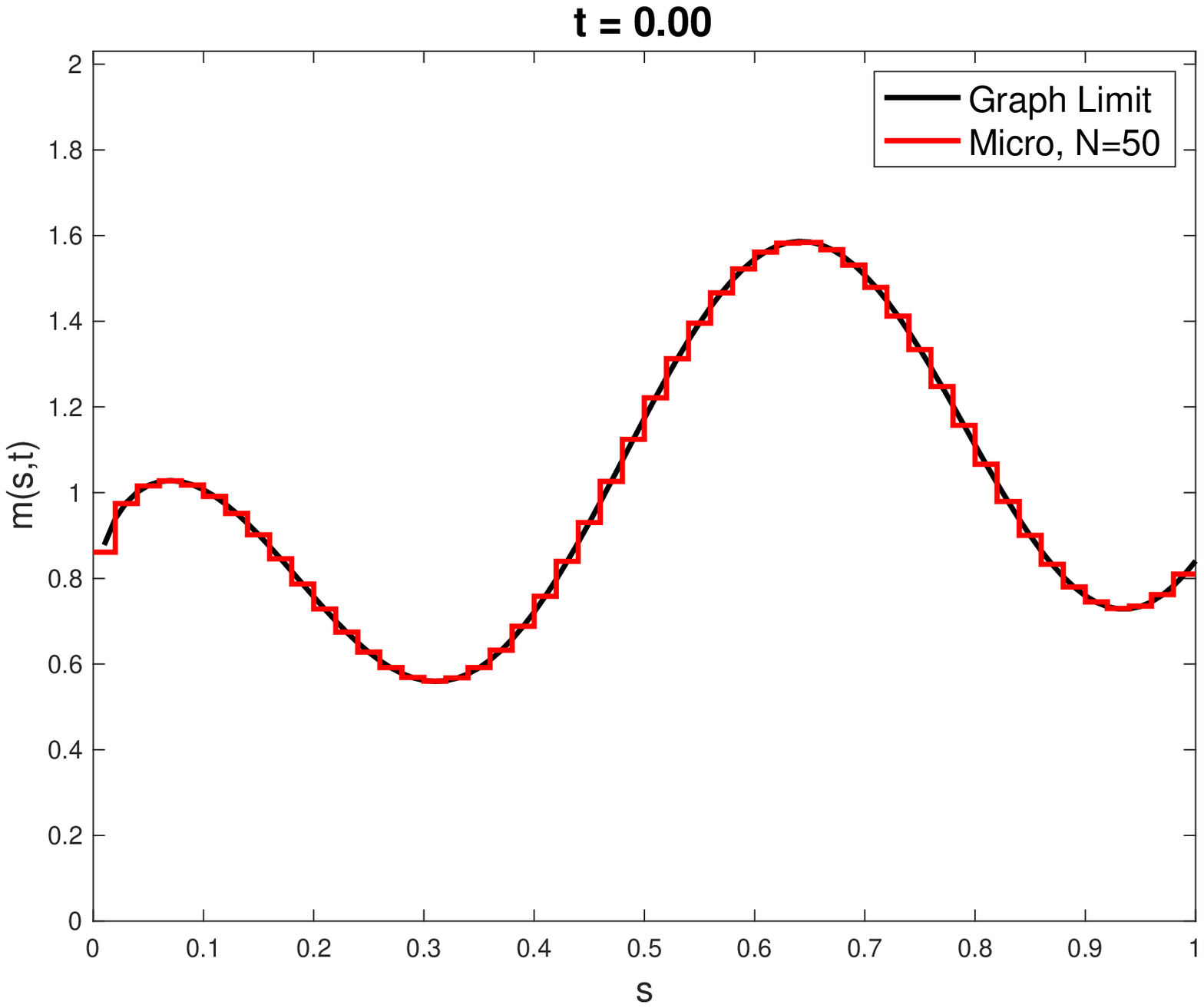}
        \includegraphics[width=0.3\textwidth]{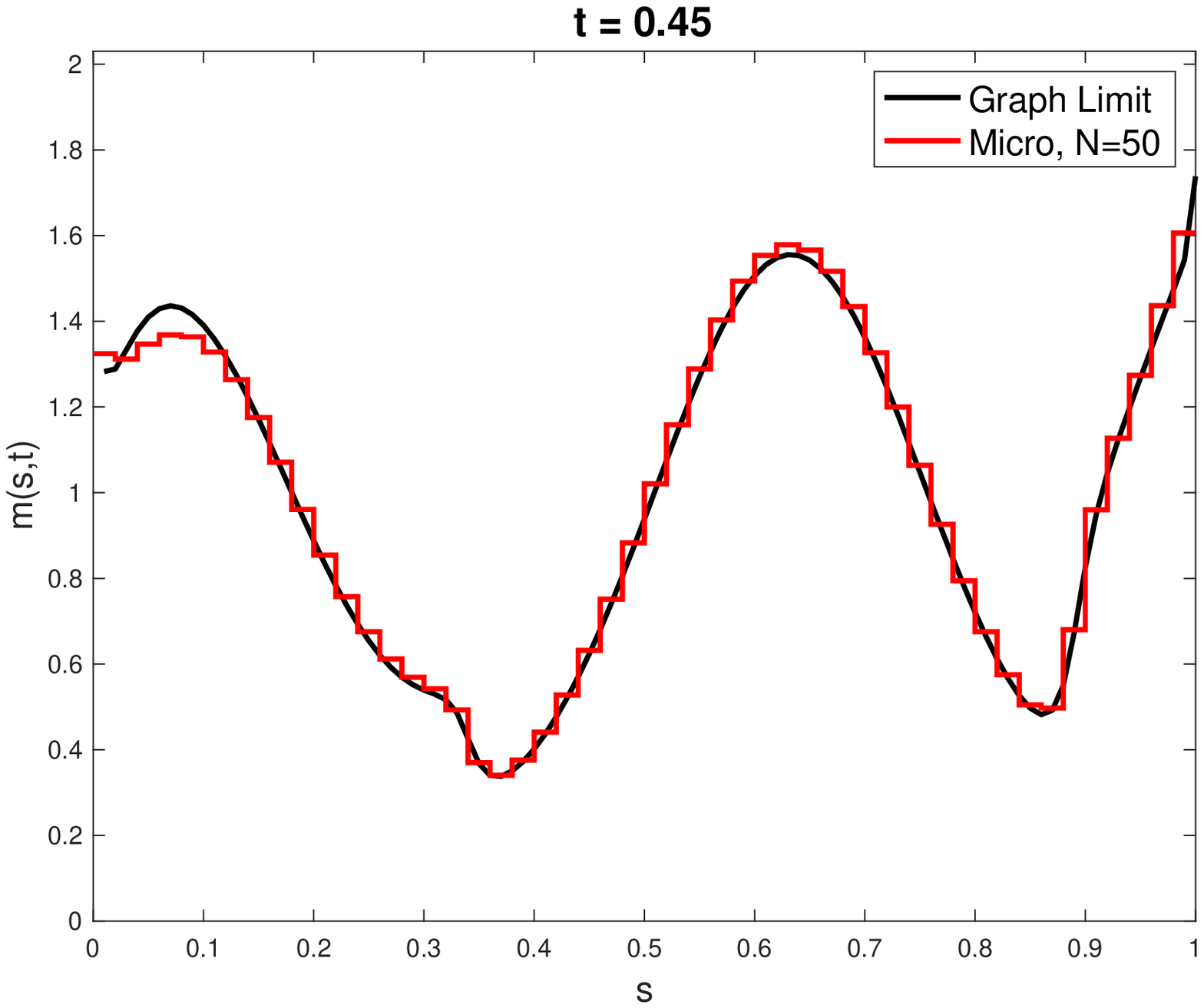}
        \includegraphics[width=0.3\textwidth]{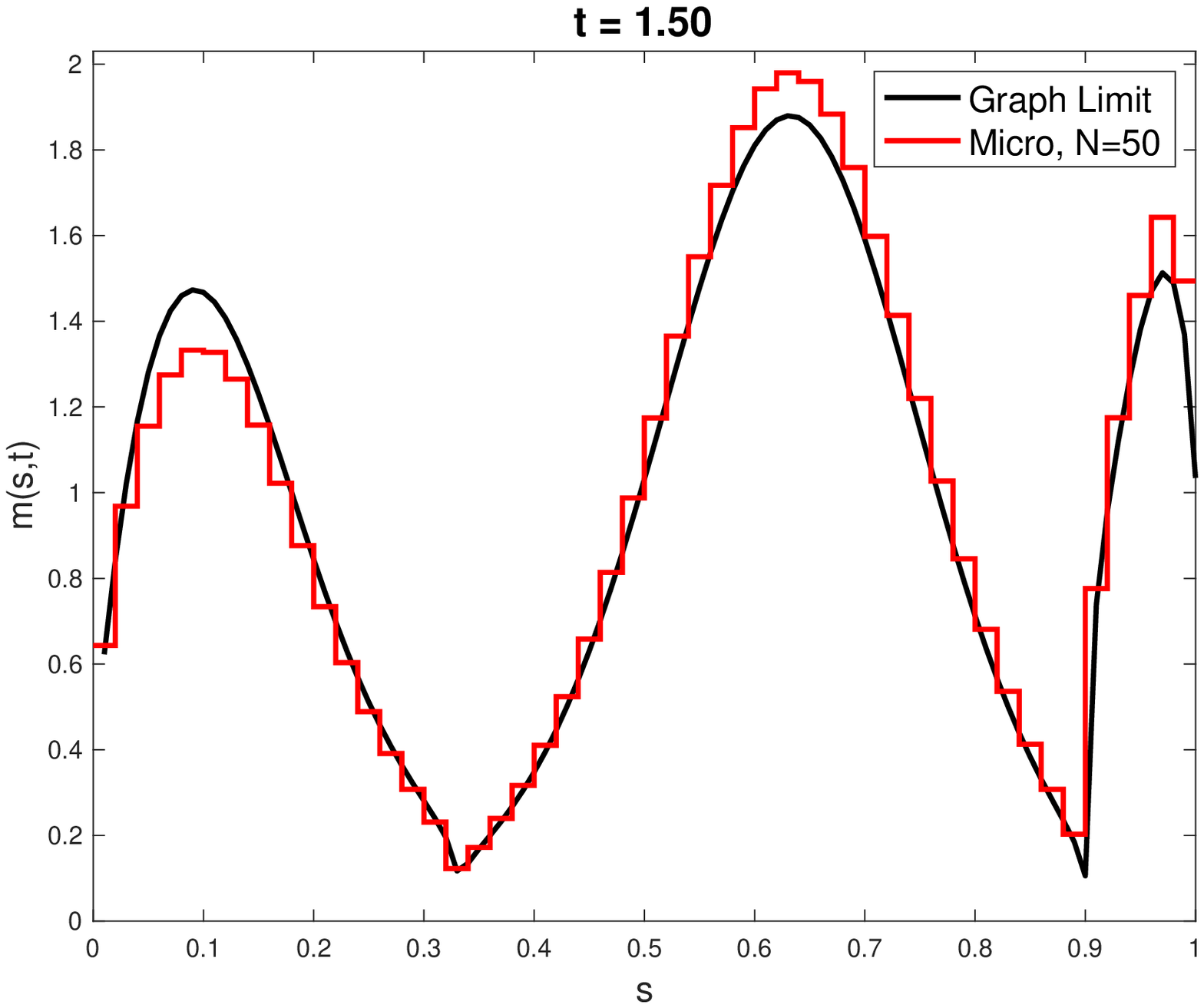} \\    
    \caption{\label{fig:microGL} Representation of the solutions $s\mapsto x_N(t,s)$ (top) and $s \mapsto m_N(t,s)$ (bottom) to the microscopic dynamics \eqref{eq:syst-gen}-\eqref{eq:modelsimumicro} and of the solutions $s\mapsto x(t,s)$ (top) and $s\mapsto m(t,s)$ (bottom) to the Graph Limit equation \eqref{eq:GraphLimit-gen}-\eqref{eq:modelsimuGL} for $t=0$, $t=0.45$ and $t=1.5$.}
\end{figure}

Figure \ref{fig:microGLmacro} represents the solution $x\mapsto \mu_t(x)$ to the mean-field equation \eqref{eq:mfl}-\eqref{eq:modelsimumacro} with initial condition given by
\[
\mu_0(x) = \int_I m_0(s) \delta( x-x_0(s)) ds 
\]

 at times $t=0$, $t=0.45$ and $t=1.5$. Again we observe convergence of the population to three clusters, as the measure converges to three Dirac masses located at the centers of mass of the clusters. As in the previous two representations, convergence to the left-most cluster is slower than convergence to the center and right clusters.
For comparison, the solution to the microscopic model \eqref{eq:syst-gen}-\eqref{eq:modelsimumicro} with $N=50$ was plotted on the same figure (red line) using the step-function representation $\mu_t^{N,n}$ defined as follows:
\[
\forall j\in \{1,...n\}, \quad \forall x\in E_j, \quad \mu^{N,n}_t(x):= \sum_{i=1}^N \frac{m_i(t)}{|E_j|}\one_{E_j}(x_i(t)),
\]
where for all $j\in \{1,...n\}$, $E_j:=[\frac{j-1}{n},\frac{j}{n})$, and $n=25$.
Subordination of the mean-field limit to the graph limit is shown by representing the solution to the graph limit equation \eqref{eq:GraphLimit-gen}-\eqref{eq:modelsimuGL} 
\[
\forall x\in \R, \quad \tilde{\mu}_t(x) = \int_I m(t,s) \delta( x-x(t,s)) ds.
\]
The solution $t\mapsto (x_i(t), m_i(t))_{i\in\elts}$ to the system of ODE was computed using the Matlab solver ode45. 
The solution $(t,s)\mapsto (x(t,s),m(t,s))$ to the integro-differential equation was computed using Euler's method for time differentiation and Simpson's method for space integration.
Lastly, the solution $(t,x)\mapsto\mu(t,x)$ to the transport PDE with source was computed using a standard Lax-Wendroff
scheme.

\begin{figure}[!h]
    \centering
        \includegraphics[width=0.3\textwidth]{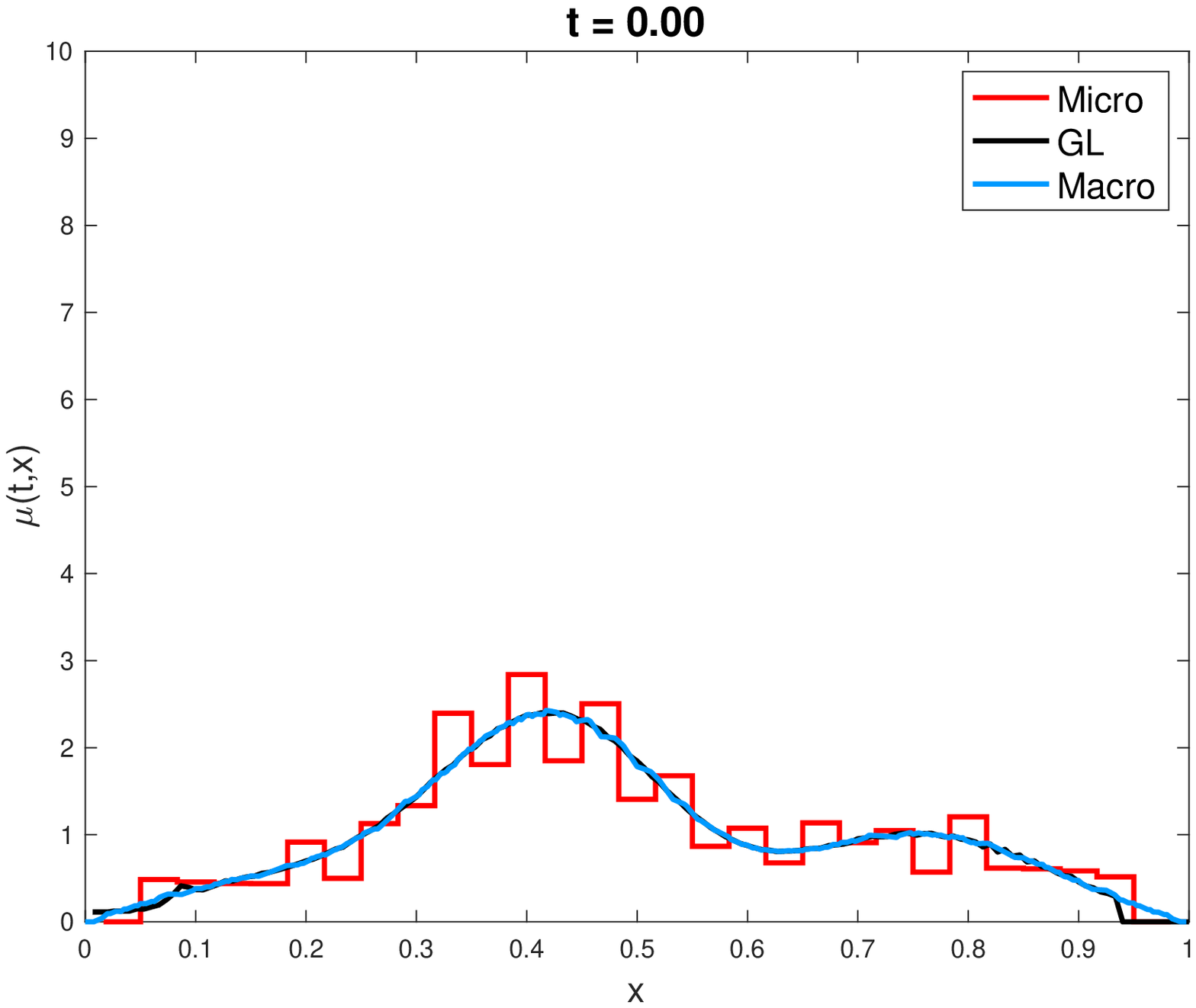}
        \includegraphics[width=0.3\textwidth]{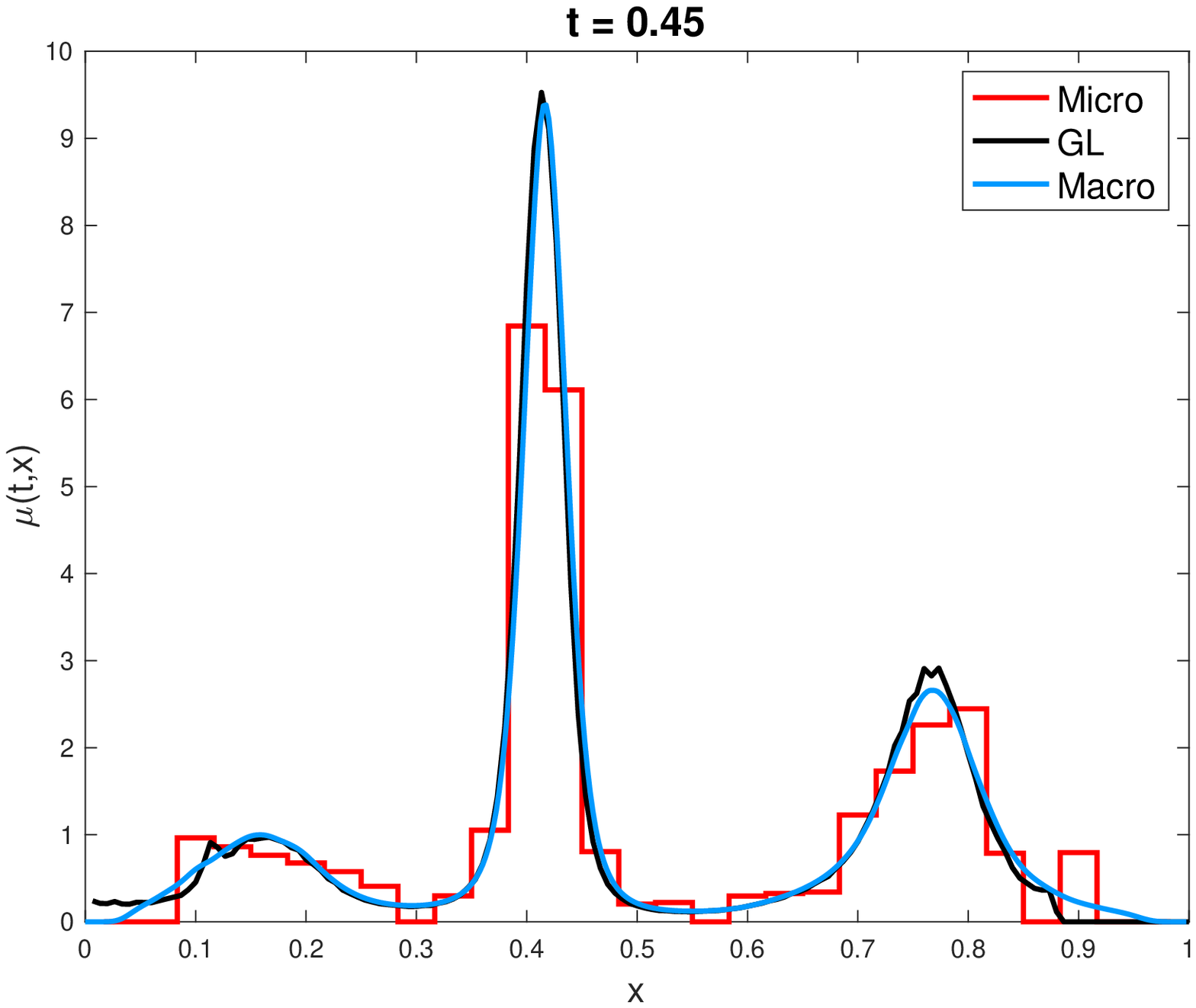}
        \includegraphics[width=0.3\textwidth]{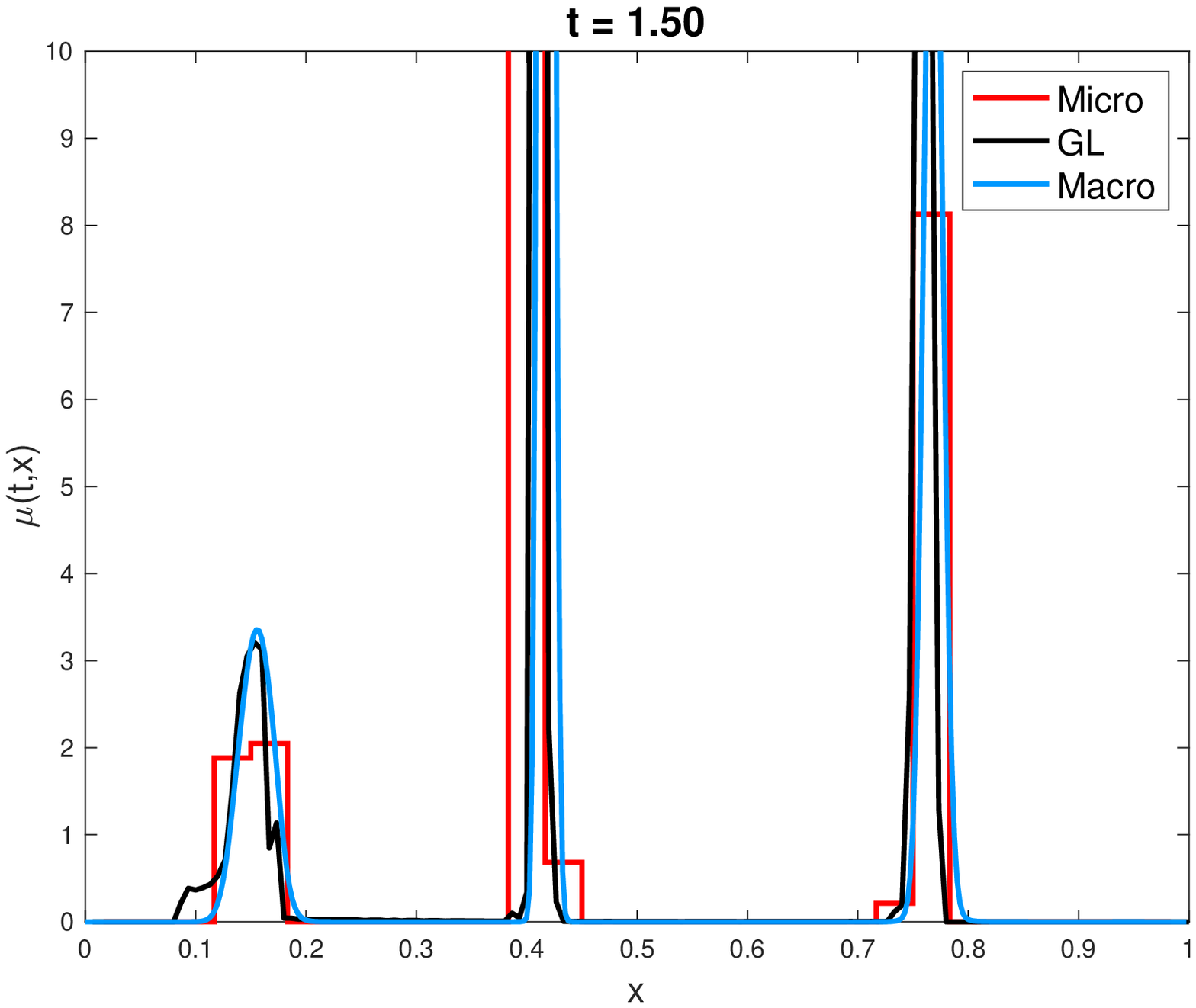} \\   
    \caption{\label{fig:microGLmacro} Representation of the solution $\mu_t^{N,n}$ (red) to the microscopic dynamics \eqref{eq:syst-gen}-\eqref{eq:modelsimumicro} for $N=50$ and $n=25$, of the solution $\tilde{\mu}_t$ (black) to the graph limit equation \eqref{eq:GraphLimit-gen}-\eqref{eq:modelsimuGL} and of the solution $\mu_t$ (blue) to the transport equation with source \eqref{eq:mfl}-\eqref{eq:modelsimumacro} at times $t=0$, $t=0.45$ and $t=1.5$.}
\end{figure}

\newpage
\bibliography{biblio}

\end{document}